\documentclass[11pt]{article}

\usepackage[utf8]{inputenc}
\usepackage[T1]{fontenc}
\usepackage[english]{babel}
\usepackage{amsmath,amssymb,amsfonts}
\usepackage{amsthm}
\usepackage{geometry}
\usepackage{mathtools}
\mathtoolsset{showonlyrefs}
\usepackage{bm}
\usepackage{subcaption}
\usepackage[hidelinks]{hyperref}
\usepackage[shortlabels]{enumitem}
\usepackage{listings}
\usepackage{cite}

\newtheorem{theorem}{Theorem}
\newtheorem{lemma}[theorem]{Lemma} 
\newtheorem{proposition}[theorem]{Proposition}
\newtheorem{remark}[theorem]{Remark}
\newtheorem{definition}[theorem]{Definition}
\newtheorem{corollary}[theorem]{Corollary}
\theoremstyle{definition}
\newtheorem{example}[theorem]{Example}

\newcommand{\R}{\mathbb{R}}
\newcommand{\N}{\mathbb{N}}
\renewcommand{\d}{\mathrm{d}}
\renewcommand{\P}{\mathcal{P}}
\newcommand{\T}{\mathrm{T}}
\newcommand{\M}{\mathcal M}
\newcommand{\D}{\mathrm D}
\newcommand{\HD}{\mathrm H}
\newcommand{\F}{\mathcal F}
\newcommand{\V}{\mathcal V}
\newcommand{\zb}{\bm}
\newcommand{\vb}{\zb v}
\mathtoolsset{showonlyrefs}
\newcommand{\Beta}{\mathrm{B}}

\DeclareMathOperator*{\argmin}{arg\,min}

\DeclareMathOperator{\dom}{dom}
\DeclareMathOperator{\Id}{Id}
\DeclareMathOperator{\Lip}{Lip}
\DeclareMathOperator{\prox}{prox}
\DeclareMathOperator{\supp}{supp}
\newcommand{\tT}{\mathrm{T}}
\newcommand{\opt}{\mathrm{opt}}

\begin{document}

\title{Wasserstein Steepest Descent Flows\\
  of Discrepancies with Riesz Kernels}
\author{
  Johannes Hertrich\footnotemark[1]
  \and
  Manuel Gr\"af\footnotemark[1]
  \and
  Robert Beinert\footnotemark[1]
  \and
  Gabriele Steidl\footnotemark[1] 
}    
\maketitle
\footnotetext[1]{Institute of Mathematics,
  TU Berlin,
  Stra{\ss}e des 17. Juni 136, 
  10623 Berlin, Germany,
  \{j.hertrich, graef, beinert, steidl\}@math.tu-berlin.de
}

\begin{abstract}
The aim of this paper is twofold.
Based on the geometric Wasserstein tangent space, we first introduce Wasserstein steepest descent flows.
These are locally absolutely continuous curves in the Wasserstein space
whose tangent vectors point into a steepest descent direction
of a given functional.
This allows the use of Euler forward schemes instead of
Jordan--Kinderlehrer--Otto schemes.
For $\lambda$-convex functionals,
we show that
Wasserstein steepest descent flows are an equivalent characterization of Wasserstein gradient flows.
The second aim is to study Wasserstein flows 
of the maximum mean discrepancy 
with respect to certain Riesz kernels. 
The crucial part is hereby the treatment of the interaction energy.
Although it is not $\lambda$-convex along generalized geodesics,
we give analytic expressions for Wasserstein steepest descent flows of
the interaction energy starting at Dirac measures.
In contrast to smooth kernels, the particle may explode, i.e., a Dirac measure becomes a non-Dirac one.
The computation of steepest descent flows amounts to finding equilibrium measures with external fields, which
nicely links Wasserstein flows of interaction energies with potential theory.
Finally, we provide numerical simulations of Wasserstein steepest descent flows of discrepancies.
\end{abstract}

\section{Introduction}

Wasserstein gradient flows have received much attention both from the theoretic and application point of view for many years.
For a good overview on the theory, we refer to the books of Ambrosio, Gigli and Savar\'e \cite{BookAmGiSa05} and
Santambrogio \cite{S2015}. 
The theory of gradient flows on probability distributions provides a framework for analyzing 
and constructing particle-based methods by connecting the optimization of functionals with dynamical systems
based on differential geometric ideas.
A pioneering example is given by the overdamped Langevin equation, where the associated Fokker--Planck
equation is just the gradient flow of the Kullback--Leibler functional $\F(\mu) = \text{KL}(\mu,\nu)$ 
in the Wasserstein geometry \cite{JKO1998,Ot01,OW2005,Pav2014}.
Recently similar ideas have been worked out, replacing either the functional or the underlying geometry,
and were also adopted as information flows in deep learning approaches, 
see for instance \cite{ArKoSaGr19,GHLS2019,HHS2021,Liu2017,LBADDP2022,NR2021,RC2013,TS2020,WT2011} 
among the huge amount of papers.

Our interest in Wasserstein flows arises from the approximation of probability measures 
by empirical measures when halftoning images. 
In \cite{EGNS2021,FHS2013,GPS2012}, the gray values of an image
are considered as values of a probability density function $\rho$ of a measure $\nu$,
and the aim consists in approximating this measure by those empirical measure
$\mu = \frac1M \sum_{i=1}^M \delta_{x_i}$, $x_i \in \R^2$,
which minimizes the (maximum mean) discrepancy with the negative distance kernel $K(x,y) = - \| x - y \|$, i.e.,
the functional
\begin{equation} \label{halftoning}
 \mathcal F (\mu) = \mathcal D_{-\|\cdot\|}^2(\mu,\nu)
 =
  \underbrace{
   \sum_{i=1}^M \int_{\R^2} \| x_i -  y \| \rho(y) \, \d y }_{ {\rm attraction} }
  - \, \frac{1}{2M}                                                       
  \underbrace{\sum_{i,j=1}^M  \|x_i - x_j \|}_{{\rm repulsion}}  \, +  \, \text{c}
	.
\end{equation}
The attraction term ensures that the points $x_i$ are pushed to areas where the density is high, while the repulsion term
avoids point clustering. For an illustration see Figure~\ref{fig:halftoning}. 
The discrepancy with negative distance kernel is also known as energy distance \cite{Szekely2002,SSGF2013}.
Note that halftoning with the kernel $K(x,y) = - \| x - y \|^{-1}$ 
was addressed under the name electrostatic halftoning in the initial paper \cite{SGBW2010}, see also \cite{TSGSW2011}.
\begin{figure}[t]
  \begin{center}
    \begin{tabular}{ccc}
      \includegraphics[width=0.25\textwidth]{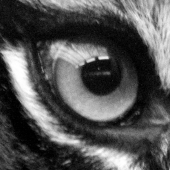}
      &
      &
        \includegraphics[width=0.25\textwidth]{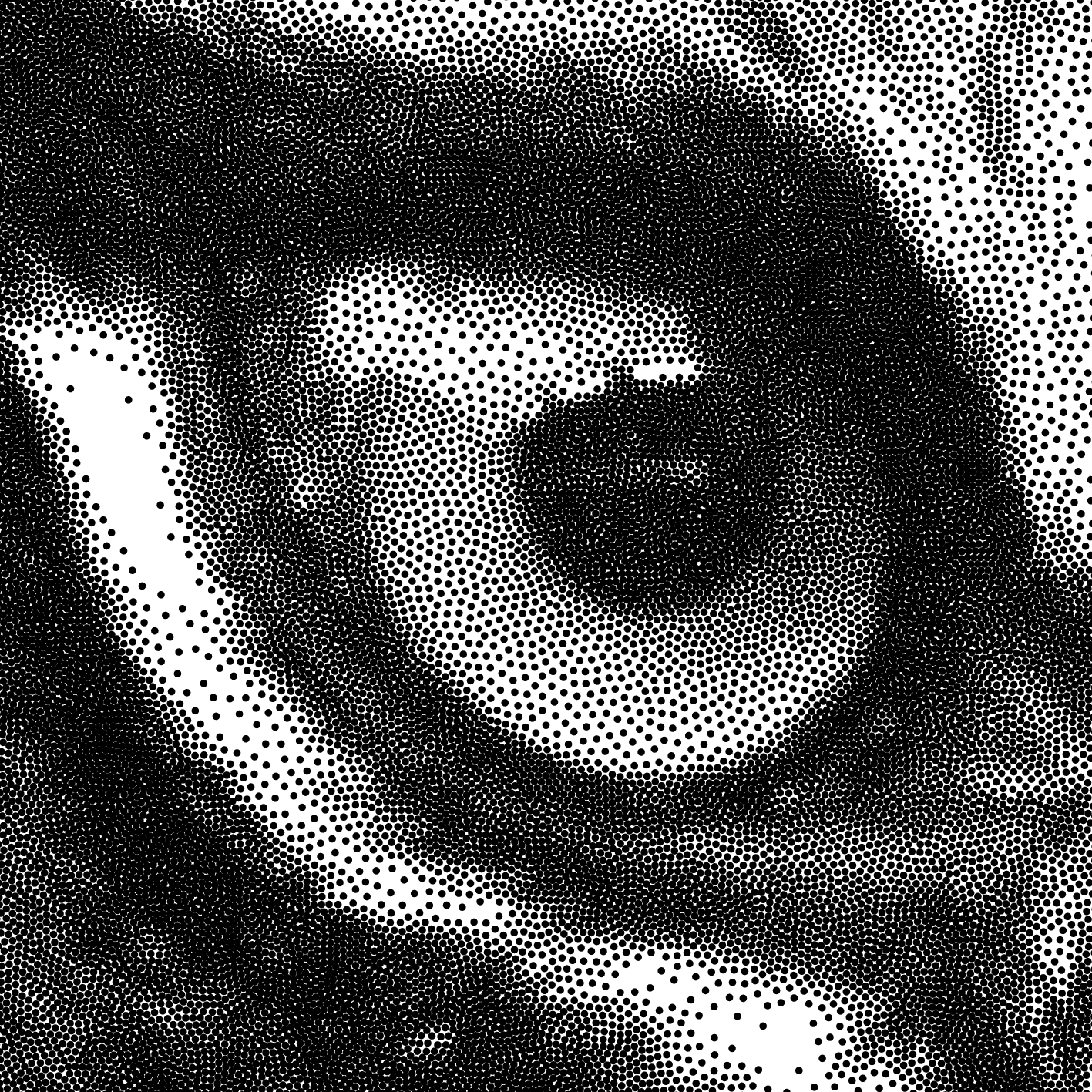}
      \\
      {\small image $\rho$}& &{\small halftoned image}
    \end{tabular}		
    \caption{Halftoning of an image. Gray values 
      are considered as values of a probability density function of a measure which
      is approximated by an empirical measure such that the discrepancy between 
      both measures becomes small. The halftoned image shows the position of the point measures.
    }					\label{fig:halftoning}
  \end{center}
\end{figure}
The halftoning functional \eqref{halftoning} is a special instance of discrepancy functionals
{\small
\begin{align}
  \mathcal F (\mu) = \mathcal D_{K}^2(\mu,\nu)
  = 
    \underbrace{-\int_{\R^{2d}} K(x,y) \d\nu(y) \d \mu(x)}_{\text{potential energy}}
  +                    
    \underbrace{\frac12 \int_{\R^{2d}} K(x,y) \, \d \mu(x) \d \mu(y)}_{\text{interaction energy}} \, +  \,\text{c},
\end{align}
}%
defined for  conditionally positive definite kernels $K: \R^d \times\R^d \to \R$ and arbitrary probability measures $\mu,\nu$ on $\R^d$,
where the first term is the potential energy of $\mu$ with respect to the potential of $\nu$ 
and the second term is known as interaction energy of $\mu$.
The restriction to empirical measures $\mu = \frac1M \sum_{i=1}^M \delta_{x_i}$
leads to the consideration of particle gradient flows of
$F(x_1,\ldots,x_M) \coloneqq \mathcal D_{K}^2(\frac1M \sum_{i=1}^M \delta_{x_i},\nu)$ in $\R^d$.
In \cite{ArKoSaGr19}, it has been established that for smooth kernels $K$ these particle flows are indeed Wasserstein gradient flows.
In other words, Wasserstein gradient flows starting at an empirical measure remain empirical measures and coincide with usual
gradient descent flows in $\R^d$. 
The situation changes for non-smooth kernels like the 
negative distance kernel applied in \eqref{halftoning}.
Here it is known that, for the interaction energy,
the Wasserstein gradient flow starting at an empirical measures cannot remain empirical, see \cite{BaCaLaRa13}.  
This implies that for the negative distance kernel, particle gradient flows of the discrepancy functional
cannot be Wasserstein gradient flows. 
In one dimension, this can be readily seen by the isome\-tric embedding of the Wasserstein space $\P_2(\R)$ 
into the Hilbert space $L_2((0,1))$, see \cite{BoCaFrPe15,CDEFS2020,HBGS2023}.
In dimensions $d \ge 2$, the geometry of the Wasserstein space 
is more complicated, and it is not obvious to answer if (sub)gradients of the above functionals
exist at any measure,
in particular at measures which are not absolutely continuous.
To study such cases, we recall the concept of the geometric tangent space of $\P_2(\R^d)$ 
which generalizes tangent vector fields to tangent velocity plans \cite{BookAmGiSa05, Gi04}. 
Based on this, we introduce the notion of the direction of steepest descent,
which leads us to a pointwise notion of the Wasserstein flows, which we call Wasserstein steepest descent flows.
For functionals, which are $\lambda$-convex along generalized geodesics,  we show that a curve is a Wasserstein gradient flow if and only if 
it is a Wasserstein steepest descent flows. 
If Wasserstein gradient and steepest descent flows coincide in a more general setting, remains an open question.
Unfortunately, for the Riesz kernel $K(x,y) = -\|x-y\|^r$, $r \in (0,2)$, neither the interaction energy nor the discrepancy functional are
$\lambda$-convex along generalized geodesics in dimensions $d \ge 2$. It is not trivial to check if these functionals are regular
such that the theory in \cite[Thm~11.3.2]{BookAmGiSa05} applies to this scenario. 
For the interaction energy, we provide analytic solutions for Wasserstein steepest descent flows starting at Dirac measures, 
which completes the findings for $d=1$ in \cite{BoCaFrPe15} and for $d\ge 2$ in \cite{CaHu17,GCO2021,ChSaWo22b}.
In particular, the direction of steepest descent at $\delta_0$ relates to the well studied optimization problem of equilibrium measures with external field in potential theory \cite{BookLa72, BookSaTo97}.
For the discrepancy functional, we determine steepest descent directions for Riesz kernels with $r \in [1,2)$ and
and show numerical simulations of Wasserstein steepest descent flows starting at Dirac measures
for target Dirac measures in two and three dimensions.
For a simulation of such flows with neural networks we refer to \cite{AHS2023,HWAH2023}.

\paragraph{Outline of the paper} We start by
providing preliminaries on Wasserstein spaces as geodesic spaces in Section \ref{sec:prelim}.
Basic facts on Wasserstein gradient flows, 
in particular, on the existence and uniqueness of Wasserstein proxies and on the convergence of the
MMS to Wasserstein gradient flows are recalled in Section \ref{sec:WGF}.
Then, in Section \ref{sec:gts}, we introduce Wasserstein
steepest descent flows which rely on the concept of the geometric Wasserstein tangent space.
We show  for locally Lipschitz continuous functions which are $\lambda$-convex along generalized geodesics, 
that there exists a unique Wasserstein steepest descent flow which coincides with the Wasserstein gradient flow.
Then we turn to special functionals arising from discrepancies defined with respect to  Riesz kernels
in Section \ref{sec:discr}.
Discrepancy functionals are, up to a constant, the sum of an
interaction energy and a potential energy.
In Section \ref{sec:mmd}, we investigate Wasserstein steepest descent flows of the interaction energy
starting at Dirac measures. This leads to the task of solving a constrained optimization problem 
related to a penalized one
which has to be solved when computing Wasserstein proxies.
We provide an analytic formula for the Wasserstein steepest descent flow. 
Finally, in Section \ref{sec:R}, we present numerically computed particle gradient flows 
for the whole discrepancy functional,  which are in good 
agreement with our findings for small time intervals.

\section{Preliminaries} \label{sec:prelim}
\paragraph{Wasserstein Space}

Let $\mathcal M(\R^d)$ denote the space of $\sigma$-additive signed Borel measures, 
$\mathcal P(\R^d)$ the set of all Borel probability measures,
and $\P_2(\R^d)$ its subset of measures with finite second moments,
i.e.\
\begin{equation}
  \label{def:P_2}
  \P_2(\R^d) 
  \coloneqq 
  \Bigl\{ \mu \in \P(\R^d) \colon \int_{\R^d}\|x\|_2^2 \, \d \mu(x) < \infty \Bigr\}.
\end{equation}
The set $\P_2^r(\R^d)$ of absolutely continuous probability measures with respect to the Lebesgue measure is a dense subset of $\P_2(\R^d)$. 
For $\mu \in \M(\R^d)$ and measurable $T\colon\R^d \to \R^n$,
the \emph{push-forward} of $\mu$ via $T$ is given by
$T_{\#}\mu \coloneqq \mu \circ T^{-1}$.
For $x=(x_1,\ldots,x_n) \in (\R^d)^{\times n}$,
the projection to the $(i_1,\ldots,i_k)$th components is denoted by
\begin{equation}
    \pi_{i_1,\ldots,i_k}(x) \coloneqq (x_{i_1},\ldots, x_{i_k}).
\end{equation}
The \emph{Wasserstein distance} 
$W_2\colon\P_2(\R^d) \times \P_2(\R^d) \to [0,\infty)$ 
is given by
\begin{equation}        \label{def:W_2}
  W_2^2(\mu, \nu)
  \coloneqq 
  \min_{\zb \pi \in \Gamma(\mu, \nu)} 
  \int_{\R^d\times \R^d}
  \|x_1 -  x_2\|_2^2
  \, \d \zb \pi(x_1, x_2),
  \qquad \mu,\nu \in \P_2(\R^d),
\end{equation}
where 
\begin{equation*}
  \Gamma(\mu, \nu)
  \coloneqq
  \{ \zb \pi \in \P_2(\R^d \times \R^d):
  (\pi_{1})_{\#}\zb \pi = \mu,\; (\pi_{2})_{\#}\zb \pi = \nu\}.
\end{equation*}
The set of optimal transport plans $\zb \pi$ 
realizing the minimum in \eqref{def:W_2}
is denote by $\Gamma^{\rm{opt}}(\mu, \nu)$.
If $\mu \in \P_2^{r}(\R^d)$, then the optimal transport plan is unique 
and is moreover given by a so-called transport map, see \cite{Brenier1987} 
and \cite[Thm~6.2.10]{BookAmGiSa05}. Let $L_2(\mu,\R^d)$ denote the 
space of (equivalence classes of) functions $f:\R^d \to \R^d$
with $\int_{\R^d} \|f\|_2^2 \, {\rm d} \mu(x) < \infty$.
								
	\begin{theorem}\label{thm:map-plan} 
   Let $\mu \in \mathcal P_2^{r}(\R^d)$ and $\nu \in \mathcal P_2(\R^d)$. 
	Then there is a unique plan $\zb \pi \in \Gamma^{\rm{opt}}(\mu, \nu)$
	 which is  induced by a unique
	measurable optimal transport map $T\colon \R^d \to \R^d$, i.e.,
	\begin{equation} \label{eq:T}
	\zb \pi = ({\rm{Id}}, T)_\# \mu
	\end{equation}
	and
  \begin{equation} \label{eq:monge}
	 W_2^2(\mu, \nu) =
 \min_{T\colon  \R^d \to  \R^d}
  \int_{\R^d}  \|T(x) -x \|_2^2  \, \d \mu(x) 
  \quad \text{subject \; to} \quad T_\# \mu = \nu.
  \end{equation}
	Further, $T = \nabla \psi$, where $\psi\colon \R^d \to (-\infty,+\infty]$
	is convex, lower semi-con\-tin\-u\-ous (lsc) and $\mu$-a.e.\ differentiable. 
	Conversely, if $\psi$ is convex, lsc and $\mu$-a.e.\ differentiable with $\nabla \psi \in L_2(\mu,\R^d)$,
	then $T \coloneqq \nabla \psi$ is an optimal map from $\mu$ to $\nu \coloneqq   T_\# \mu \in \mathcal P_2(\R^d)$.
 \end{theorem}

\paragraph{Wasserstein Geodesics}
A curve $\gamma \colon I \to \P_2(\R^d)$ on an interval $I \subset \R$, is called a \emph{(length-minimizing) geodesic} 
if there exists a constant $C \ge 0$ 
such that 
\begin{equation}
    \label{eq:geodesic}
    W_2(\gamma(t_1), \gamma(t_2)) = C |t_2 - t_1|, \qquad \text{for all } t_1, t_2 \in I.
\end{equation}
 The constant $C$ is the speed of the geodesic.
The Wasserstein space is geodesic, 
i.e.\ any two measures $\mu, \nu \in \P_2(\R^d)$ 
can be connected by a geodesic.
These geodesics may be characterized by optimal plans.

\begin{proposition}[{\cite[Thm~7.2.2]{BookAmGiSa05}}]
\label{prop:geo}
    Let $\epsilon > 0$.
    Any geodesic $\gamma \colon [0,\epsilon] \to \P_2(\R^d)$ connecting $\mu, \nu \in \P_2(\R^d)$
    is determined by an optimal plan $\zb \pi \in \Gamma^{\rm{opt}}(\mu, \nu)$ via
    \begin{equation} \label{eq:geodesic_plan}
      \gamma(t) \coloneqq \bigl((1-\tfrac t \epsilon) \, \pi_1 
      + \tfrac t \epsilon \, \pi_2 \bigr)_{\#} \zb \pi, \quad \qquad t \in [0,\epsilon].
    \end{equation}
    Conversely, any $\zb \pi \in \Gamma^{\rm{opt}}(\mu, \nu)$ gives rise to a 
    geodesic $\gamma \colon [0,\epsilon] \to \P_2(\R^d)$ connecting $\mu$ and $\nu$.
\end{proposition}

The optimal $\zb \pi$ in \eqref{eq:geodesic_plan} may be replaced by non-optimal plans
to obtain more general interpolating curves.
For instance,
based on the set of three-plans with base $\sigma \in \P_2(\R^d)$ given by
\begin{equation}
    \Gamma_\sigma(\mu,\nu)
    \coloneqq
    \bigl\{ 
    \zb{\alpha} \in \P_2(\R^d \times \R^d \times \R^d)
    :
    (\pi_1)_\# \zb \alpha = \sigma,
    (\pi_2)_\# \zb \alpha = \mu,
    (\pi_3)_\# \zb \alpha = \nu
    \bigr\},
\end{equation}
the so-called
\emph{generalized geodesics} $\gamma \colon [0, \epsilon] \to \P_2(\R^d)$
joining $\mu$ and $\nu$ (with base $\sigma$) is defined as
\begin{equation} \label{ggd}
  \gamma(t) 
  \coloneqq 
  \bigl(
  (1-\tfrac t \epsilon) \pi_2 
  + \tfrac t \epsilon \pi_3
  \bigr)_\# \zb \alpha,
  \qquad
  t \in [0, \epsilon],
\end{equation}
where $\zb \alpha \in \Gamma_\sigma (\mu,\nu)$
with
$(\pi_{1,2})_\# \zb{\alpha} \in \Gamma^{\opt}(\sigma,\mu)$
and
$(\pi_{1,3})_\# \zb{\alpha} \in \Gamma^{\opt}(\sigma,\nu)$, see \cite[Def~9.2.2]{BookAmGiSa05}.
The plan $\zb{\alpha}$ may be interpreted as transport from $\mu$ to $\nu$ via $\sigma$.

\paragraph{$\lambda$-Convexity along Wasserstein Geodesics}
Let $\lambda \in \R$ be a fixed constant.
A function $\F\colon \P_2(\R^d) \to (-\infty,+\infty]$ is called \emph{$\lambda$-convex along geodesics} \cite[Def~9.1.1]{BookAmGiSa05}
if, for every 
$\mu, \nu \in \dom \F \coloneqq \{\mu \in  \P_2(\R^d): \F(\mu) < \infty\}$,
there exists at least one geodesic $\gamma \colon [0, 1] \to \P_2(\R^d)$ 
between $\mu$ and $\nu$ such that
\begin{equation}     
    \label{def:lambda_convex}
    \F(\gamma(t)) 
    \le
    (1-t) \, \F(\mu) + t \, \F(\nu) 
    - \tfrac{\lambda}{2} \, t (1-t) \,  W_2^2(\mu, \nu), 
    \qquad  t \in [0,1].
\end{equation}
Analogously,
a function $\F\colon \mathcal P_2(\R^d) \to (-\infty,\infty]$ is called
\emph{$\lambda$-convex along generalized geodesics} \cite[Def~9.2.4]{BookAmGiSa05}, 
if for every $\sigma,\mu,\nu \in \dom \F$,
there exists at least one generalized geodesic $\gamma \colon [0,1] \to \P_2(\R^d)$ 
related to some $\zb{\alpha}$ as in \eqref{ggd}
such that
\begin{equation} \label{eq:gg}
  \F(\gamma(t))
  \le 
  (1-t) \, \F(\mu) 
  + t \, \F(\nu) 
  - \tfrac\lambda2 \, t(1-t) \, W_{\zb \alpha}^2(\mu,\nu), 
  \qquad t \in [0,1],
\end{equation}
where 
\begin{equation}
    \label{eq:W2alpha}
W_{\zb{\alpha}}^2 (\mu, \nu)
\coloneqq \int_{\R^d \times \R^d \times \R^d} \|x_2 - x_3\|_2^2 \, \d \zb{\alpha}(x_1,x_2,x_3).
\end{equation}
Further, $\F$ is called \emph{convex} (along generalized geodesics) if it is $0$-convex (along generalized geodesics).
Every function being $\lambda$-convex along generalized geodesics 
is also $\lambda$-convex along geodesics
since generalized geodesics with base $\sigma = \mu$ are actual geodesics.
A $\lambda$-convex function 
$\F \colon \P_2(\R^d) \to (-\infty, +\infty]$ 
is called \emph{coercive},  cf. \cite[(11.2.1b)]{BookAmGiSa05},
if there exists an $r > 0$ such that
\begin{equation}
\label{eq:coercive}
    \inf
    \biggl\{
    \F(\mu)
    : 
    \mu \in \P_2(\R^d)
    ,
    \int_{\R^d} \|x\|_2^2 \, \d \mu(x) \le r
    \biggr\}
    > -\infty.
\end{equation}

\section{Wasserstein Gradient Flows}\label{sec:WGF}
A curve $\gamma\colon I \to \P_2(\R^d)$ 
on the open interval $I \subset \R$
is called \emph{absolutely continuous} 
if there exists a function $m \in L_1(I)$ such that
\begin{equation}\label{eq_basic}
  W_2(\gamma(s), \gamma(t)) \le \int_s^t m(s) \, \d s, \qquad s,t \in I.
\end{equation}
Absolutely continuous curves are characterized by the \emph{continuity equation}
\cite[Thm~8.3.1]{BookAmGiSa05}.
More precisely, a continuous curve
$\gamma$ is absolutely continuous if and only if
there exists a Borel velocity field $v_t\colon \R^d \to \R^d$, 
$t \in I$ 
with $\int_I \| v_t \|_{L_2(\gamma(t),\R^d)} \, \d t < + \infty$
such that
\begin{equation} \label{eq:CE}
  \partial_t \gamma(t) + \nabla \cdot ( v_t \, \gamma(t)) = 0
\end{equation}
holds on $I \times \R^d$ in the distributive sense
\begin{equation}    
    \label{eq:CE_distr}
    \int_I \int_{\R^d} \partial_t \varphi(t, x) + v_t(x) \cdot \nabla_x \, \varphi(t, x) \, \d \gamma(t) \, \d t = 0
\end{equation}
for all smooth functions 
$\varphi\colon I \times \R^d \to \R$ with compact support,
i.e.\ $\varphi \in C_{\mathrm c}^\infty(I \times \R^d)$.
Moreover, there exists a unique velocity field,
henceforth also denoted by $v_t$,
such that
$m(t) \coloneqq \|v_t\|_{L_2(\gamma(t),\R^d)}$ becomes minimal in \eqref{eq_basic}.
Furthermore, 
the minimizing vector field is characterized
by the condition $v_t \in \T_{\gamma(t)} \mathcal P_2(\R^d)$
for almost every $t \in I$,
where $\T_{\mu}\P_2(\R^d)$ with $\mu \in \P_2(\R^d)$ denotes the \emph{regular tangent space}
\begin{align} 
\label{tan_reg}
    \T_{\mu}\mathcal P_2(\R^d)
    &\coloneqq 
      \overline{
      \left\{ 
      \nabla \phi: \phi \in C^\infty_{\mathrm c}(\R^d) 
      \right\}}^{L_2(\mu,\R^d)} \\
    &=
      \overline{
      \left\{ \lambda (T- \Id): (\Id ,T)_{\#} \mu \in \Gamma^{\opt} (\mu , T_{\#} \mu), \; \lambda >0
      \right\} }^{L_2(\mu,\R^d)},
\end{align}
see \cite[§~8]{BookAmGiSa05}.
Note that  $\T_{\mu} \mathcal P_2(\R^d)$ is an
	infinite dimensional subspace of $L_2(\mu,\R^d)$
	if $\mu \in \mathcal P_2^r(\R^d)$, and it is just $\R^d$ if $\mu = \delta_x$, $x \in \R^d$.
	
For a proper and lower semi-continuous (lsc) function $\F \colon \P_2(\R^d) \to (-\infty, \infty]$ 
and $\mu \in \P_2(\R^d)$, 
the \emph{reduced Fréchet subdifferential at $\mu$} 
is defined as the set $\partial \F(\mu)$
consisting of all $\zb \xi \in L_2(\mu, \R^d)$ 
satisfying
\begin{equation}\label{need}
    \F(\nu) - \F(\mu)
    \ge 
    \inf_{\zb \pi \in \Gamma^{\opt}(\mu,\nu)}
    \int_{\R^d \times \R^d}
    \langle \zb \xi(x_1), x_2 - x_1 \rangle
    \, \d \zb \pi (x_1, x_2)
    + o(W_2(\mu,\nu))
\end{equation}
for all $\nu \in \mathcal P_2(\R^d)$ or equivalently
\begin{equation}
    \liminf_{\nu \to \mu}
    \frac
    {\F(\nu) - \F(\mu) 
    - \inf_{\zb \pi \in \Gamma^{\opt}(\mu,\nu)}
    \int_{\R^d \times \R^d}
    \langle \zb \xi(x_1), x_2 - x_1 \rangle
    \, \d \zb \pi (x_1, x_2)}
    {W_2(\mu,\nu)} \ge 0,
\end{equation}
where $\nu$ converges to $\mu$ in $(\P_2(\R^d), W_2)$,
see \cite[(10.3.13)]{BookAmGiSa05}.
On the basis of this subdifferential, 
Wasserstein gradient flows may be defined as follows.

  An absolutely continuous curve  
  $\gamma \colon (0,+\infty) \to \P_2(\R^d)$ 
  with velocity field $v_t \in \T_{\gamma(t)} \mathcal P_2(\R^d)$
  is called a \emph{Wasserstein gradient flow 
  of} 
	$\F\colon \P_2(\R^d) \to (-\infty, +\infty]$
  if 
  \begin{equation}\label{wgf}
      v_t  \in - \partial \F(\gamma(t)), \quad \text{for a.e. } t > 0.
  \end{equation}

The existence of Wasserstein gradient flows is usually shown 
by using the generalized \emph{minimizing moment scheme}
\cite{Gi93, JKO1998}, 
which can be considered as Euler backward scheme for computing the Wasserstein gradient flow
\eqref{wgf}.
It is explained in the following.
For a proper and lsc function
$\F\colon \mathcal P_2(\R^d) \to (-\infty,\infty]$
and fixed $\tau > 0$, 
the \emph{proximal mapping}
$\prox_{\tau \F}$ is defined as the set-valued function
\begin{equation}
    \label{prox}
    \prox_{\tau \F} (\mu) 
    =
    \argmin_{\nu \in \mathcal P_2(\R^d)}
    \left\{ \frac{1}{2\tau} W_2^2(\mu,\nu) + \F(\nu) \right\}, 
    \qquad \mu \in \P_2(\R^d).
\end{equation}
Note that, for every $\mu \in \dom \F$, the existence 
and uniqueness of the minimizer in \eqref{prox} is assured if
$\F\colon \mathcal P_2(X) \to (-\infty,\infty]$ is $\lambda$-convex along generalized geodesics,
where $\lambda > -1/\tau$, 
see \cite[Lem~9.2.7]{BookAmGiSa05}.

Assuming that the proximal mapping is non-empty,
and starting with some $\mu_\tau^0 \coloneqq \mu_0 \in \P_2(\R^d)$,
we consider the piecewise constant curves given by the \emph{minimizing movement scheme} (MMS),  which is  also known as \emph{Jordan--Kinderlehrer--Otto scheme}:
\begin{equation}
    \label{otto_curve}
    \gamma_\tau \big|_{( (n-1)\tau, n \tau]}
    \coloneqq 
    \mu_\tau^n 
    \qquad\text{with}\qquad 
    \mu_\tau^n \in \prox_{\tau\F}(\mu_\tau^{n-1}).
\end{equation}
If $\F$ is $\lambda$-convex along generalized geodesics,
then there exists a $\tau^* >0$ such that 
$\prox_{\tau\F}(\mu)$ becomes single-valued 
for all $\tau < \tau^*$ and $\mu \in \P_2(\R^d)$.
Then we can study the limit of the curves $\gamma_\tau$.

\begin{theorem}[{\cite[Thm~11.2.1]{BookAmGiSa05}}]   \label{thm:existence_gflows_ggd}
    Let $\F\colon \P_2(\R^d) \to (-\infty,+\infty]$ be 
    proper, 
    lsc, 
    coercive,
    and $\lambda$-convex along generalized geodesics,
    and let $\mu_0 \in \overline{\dom \F}$.
    Then the curves $\gamma_\tau$ defined via the minimizing movement scheme \eqref{otto_curve} 
    converge for $\tau \to 0$ locally uniformly to 
    a locally Lipschitz curve $\gamma \colon (0,+\infty) \to \P_2(\R^d)$
    which is the unique Wasserstein gradient flow of $\F$ with $\gamma(0+) = \mu_0$.  
\end{theorem} 

Theorem~\ref{thm:existence_gflows_ggd} gives a pointwise definition of $\gamma$ for all $t \in [0,+\infty)$.
However, we will see that the interaction functional with distance kernel
is not $\lambda$-convex along geodesics.

\section{Geodesic Directions and Geodesic Tangents} \label{sec:gts}
For general $\F$, the velocity field $v_t \in \T_{\gamma(t)} \mathcal P_2(\R^d)$ in \eqref{wgf} 
is only determined for almost every $t >0$, but 
we want to give a definition of so-called steepest descent flows pointwise.
To this end, we recall the notion of the geometric tangent space, see \cite[Chap~4]{Gi04} or \cite[§~12.4]{BookAmGiSa05}, which generalizes tangent vector fields to so-called tangent velocity plans.

Note, any transport plan $\zb \pi\in\Gamma(\mu,\nu)$
is associated to a \emph{velocity plan} $\zb v \in \P_2(\R^d \times \R^d)$ by the relation
\begin{equation}
\label{plan-velocity}
  \zb v = (\pi_1, \pi_2 - \pi_1 )_{\#} \zb \pi,\quad\text{or equivalently} \quad \zb \pi = (\pi_1, \pi_1+\pi_2 )_{\#} \zb v.
\end{equation}
The set of 
all velocity plans at $\mu \in \P_2(\R^d)$ is defined by
\begin{equation}
    \label{eq:ms-vel-plan}
    \zb V(\mu)
    \coloneqq
    \{ \zb v \in \P_2(\R^d \times \R^d)
    :
    (\pi_1)_\# \zb v = \mu \}.
\end{equation}
We equip $\zb V(\mu)$ with the metric $W_\mu$ defined by
\begin{equation}
  W_{\mu}^2(\zb v, \zb w) 
  \coloneqq 
  \inf_{\zb \alpha \in \Gamma_{\mu}(\zb v, \zb w)} W_{\zb \alpha}^2((\pi_2)_{\#}\zb v, (\pi_2)_{\#} \zb w),
\end{equation}
where
\begin{equation}
    \Gamma_{\mu}(\zb v, \zb w)
    \coloneqq
    \{ \zb \alpha \in \P_2(\R^d \times \R^d \times \R^d) \;:\;
    (\pi_{1,2})_{\#}\zb \alpha = \zb v,\; (\pi_{1,3})_{\#}\zb \alpha = \zb w\}.
\end{equation}
Then, it was proven in \cite[Thm~4.5]{Gi04} that $(\zb V(\mu),W_\mu)$ is a complete metric space.
For a velocity plan $\zb v\in\zb V(\mu)$ and corresponding transport plan $\zb \pi=(\pi_1,\pi_1+\pi_2)_\#\zb v$, the curve $\gamma_{\zb v}\colon[0,\infty)\to\mathcal P_2(\R^d)$ determined by
\begin{equation} 
    \label{eq:exp-plan}
    \gamma_{\zb v}(t) 
    \coloneqq
    (\pi_1 + t \pi_2)_{\#} \zb v, \quad  t \geq 0.
\end{equation}
is equal to the interpolation 
\begin{equation}\label{eq:gamma_v_interpolation}
\gamma_{\zb v}(t)=((1-t)\pi_1+t\pi_2)_\#\zb \pi,\quad\text{for }t\in[0,1].
\end{equation}
In the case that the velocity plan $\zb v$ corresponds to an optimal transport plan $\zb \pi\in\Gamma^\opt(\mu,\nu)$, we obtain by Proposition~\ref{prop:geo} that $\gamma_{\zb v}$ is a geodesic on $[0,1]$.

In the following, we aim to characterize, for arbitrary $\epsilon > 0$, a geodesic $\gamma\colon[0,\epsilon]\to\mathcal P_2(\R^d)$ by velocity plans.
Therefore, we define the scaling of a velocity plan $\zb v \in \zb V(\mu)$ 
by a factor $c \in \R$ as
\begin{equation}
    c \cdot \zb v \coloneqq (\pi_1, c \, \pi_2)_{\#}\zb v. 
\end{equation}
Then, by definition the curve $\gamma_{c\cdot \zb v}$ fulfills
\begin{equation}\label{eq:scaling}
\gamma_{c\cdot \zb v}(t)=(\pi_1 + t \pi_2)_{\#} (c\cdot\zb v)=(\pi_1 + t \pi_2)_{\#} (\pi_1,c\pi_2)_\#\zb v=(\pi_1 + c t \pi_2)_{\#}\zb v=\gamma_{\zb v}(ct),
\end{equation}
i.e., $\gamma_{c\cdot\zb v}$ is the curve $\gamma_{\zb v}$ scaled by the factor $c$. For $c=0$ we obtain that 
\begin{equation*}
  \zb 0_\mu\coloneqq 0\cdot \zb v=\mu\otimes \delta_0.
\end{equation*}
Using this scaling and \eqref{eq:gamma_v_interpolation}, we obtain that a geodesic $\gamma\colon [0,\epsilon]\to\mathcal P_2(\R^d)$ related to $\zb \pi \in \Gamma^{\opt}(\mu,\nu)$ 
by \eqref{eq:geodesic_plan}
belongs to the velocity plan
\begin{equation}
  \zb v = \tfrac{1}{\epsilon}\cdot \bigl((\pi_1, \pi_2 - \pi_1 )_{\#} \zb \pi\bigr) = \bigl(\pi_1, \tfrac{1}{\epsilon} (\pi_2 - \pi_1) \bigr)_{\#} \zb \pi.
\end{equation}
in the sense that $\gamma_{\zb v}\equiv\gamma$ on $[0,\epsilon]$.
The main advantage of this characterization is that
two geodesics $\gamma_1\colon [0, \epsilon_1] \to \P_2(\R^d)$ 
and $\gamma_2\colon [0, \epsilon_2] \to \P_2(\R^d)$
with $\gamma_1|_{[0,\epsilon]} \equiv \gamma_2|_{[0,\epsilon]}$ 
for some $\epsilon \le \min\{\epsilon_1, \epsilon_2\}$
correspond to the same velocity plan $\zb v$.
Hence,
$\zb v$ may be interpreted as \emph{geodesic direction}.
We denote the subset $\zb V(\mu)$ consisting of  all geodesic directions at $\mu \in \P_2(\R^d)$
by
{\small 
\begin{align}   \label{def:geomTS}  
  \zb G(\mu) 
  &\coloneqq  
  \bigl\{ 
    \zb v \in \zb V(\mu):  \exists \epsilon >0
    \text{ such that } \zb \pi  = (\pi_1, \pi_1 + \tfrac1\epsilon \pi_2)_{\#} 
    \zb v \in \Gamma^{\opt}(\mu, (\pi_2) _{\#} \zb \pi ) 
  \bigr\}\\
  &=  \bigl\{ 
    \zb v \in \zb V(\mu):  \exists \epsilon >0
    \text{ such that } \gamma_{\zb v} \text{ is a geodesic on }[0,\epsilon] 
  \bigr\}.
\end{align}
}%
Then the \emph{geometric tangent space} at $\mu \in \P_2(\R^d)$ is given by \begin{equation}
  \zb \T_{\mu}\P_2(\R^d) \coloneqq \overline{ \zb G(\mu)}^{W_\mu}.
\end{equation} 
Here are some special cases.

\begin{proposition}\label{the:geomTSac}
\begin{enumerate}[\upshape(i)]
\item
  For $\mu \in \P^{r}_2(\R^d)$, it holds
  \begin{equation*}
    \zb \T_{\mu}\P_2(\R^d) =  \{\zb v = (\mathrm{Id}, v)_{\#}\mu:v \in \T_{\mu}\P_2(\R^d)\}
  \end{equation*}
  and
  $W_\mu^2(\zb v, \zb 0_\mu) =  \| v \|_{L_2(\mu,\R^d)}^2$.
\item
  For $\mu = \delta_p$, $p \in \R^d$, it holds
  \begin{equation*}
    \zb \T_{\delta_p} \P_2(\R^d) =  \{\zb v = \delta_p \otimes \eta \in \P_2(\R^d)\}
  \end{equation*}
  and
  $W_{\delta_p}^2(\zb v, \zb 0_{\delta_p}) = 
  \int_{\R^d} \|x\|_2^2 \, \d \eta(x)
  $.
\end{enumerate}		
\end{proposition}

\begin{proof}
    (i) The first part follows from \cite[Thm~12.4.4]{BookAmGiSa05}, where it was shown that
    for $\mu \in \P_2^r(\R^d)$, the so-called barycentric projection is an isometric one-to-one correspondence
between  $\zb \T_{\mu}\P_2(\R^d)$ and $\T_{\mu} \P_2(\R^d)$.

(ii) The second statement on the tangential space follows immediately from the fact that  
any probability measure with $\delta_p$ as one marginal is a product measure. 
Similarly, we obtain from 
$\Gamma_{\delta_p}(\delta_p\otimes\eta,\delta_p\otimes\delta_0)=\{\zb\alpha\}$, 
where $\zb \alpha=\delta_p\otimes\eta\otimes\delta_0$, that
  \begin{align}
    W_{\delta_p}(\delta_p\otimes\mu,\delta_p\otimes\delta_0)
    =
    W_{\zb \alpha}(\eta,\delta_0)
    &=\int_{\R^d\times\R^d\times\R^d}\|x_2-x_3\|_2^2 \, \d \zb \alpha (x_1,x_2,x_3)
    \\
    &=
    \int_{\R^d}\|x\|_2^2 \, \d \eta(x).
  \end{align}
\end{proof}

We define the 
\emph{exponential map} 
$\exp_{\mu}\colon  \zb\T_{\mu}\P_2(\R^d) \to \P_2(\R^d)$
by
\begin{equation}
       \exp_{\mu}(\zb v) \coloneqq \gamma_{\zb v}(1)= (\pi_1 + \pi_2)_{\#} \zb v.
\end{equation}
The \emph{inverse exponential map} 
$\exp_{\mu}^{-1}\colon\P_2(\R^d) \to  \T_{\mu}\P_2(\R^d)$
is given by the (multivalued) function
\begin{equation}
    \exp_{\mu}^{-1}(\nu) 
    \coloneqq 
    \bigl\{ (\pi_1, \pi_2 - \pi_1)_{\#} \zb \pi: 
    \zb \pi \in \Gamma^{\opt}(\mu,\nu) \bigr\}
\end{equation}
 and consists of all velocity plans $\zb v\in\zb V(\mu)$ such that $\gamma_{\zb v}|_{[0,1]}$ is a geodesic connecting $\mu$ and $\nu$.
Note that 
\begin{align}
  &\exp_\mu^{-1}(\nu)
  \\
  &=\{\zb v\in\zb \T_\mu\mathcal P_2(\R^d):\exp_\mu(\zb v)=\nu\}\cap\{\zb v\in \zb G(\mu):\gamma_{\zb v}|_{[0,1]} \text{ is a geodesic}\},
\end{align}
i.e.,
$\exp_\mu^{-1}$ is only the inverse of $\exp_\mu$ restricted to the set $\{\zb v\in \zb G(\mu):\gamma_{\zb v}|_{[0,1]} \text{ is a geodesic}\}$.

For a curve $\gamma\colon I \to \P_2(\R^d)$,
a velocity plan $\zb v_t \in \zb \T_{\gamma(t)}\P_2(\R^d)$ 
is called a \emph{(geometric) tangent vector} of $\gamma$ at $t \in I$ if,
for every $h > 0$ and $\zb v_{t, h} \in \exp_{\gamma(t)}^{-1}(\gamma(t+h))$,
it holds
\begin{equation}  \label{eq:geomTangentVector}
  \lim_{h \to 0+} W_{\gamma(t)}(\zb v_t,  \tfrac1h \cdot \zb v_{t, h}) = 0.
\end{equation}
If a tangent vector $\zb v_t$ exists, then the above limit 
is uniquely determined since $W_{\gamma(t)}$ is a metric on $\zb V(\gamma(t))$, and we write
\begin{equation}
      \dot \gamma(t)\coloneqq \zb v_t.
\end{equation}
In \cite[Thm~4.19]{Gi04}, it is shown that 
\begin{equation}\label{eq:consistency_tangent_vector}
\dot\gamma_{\zb v}(0)=\zb v\quad\text{for all }\zb v\in\zb \T_\mu\P_2(\R^d).
\end{equation}
Therefore, the definition of a tangent vector of a curve is consistent with the interpretation of $\gamma_{\zb v}$ as a curve in direction of $\zb v$.
For $\zb v\in\zb G(\mu)$, we can also compute the tangent vector $\dot \gamma_{\zb v}(t)$ for $t>0$ by the following lemma.

\begin{lemma}
    \label{lem:tan-geo}
    Let  $\zb v \in \zb G(\mu)$ be a velocity plan and $\epsilon>0$ such that $\gamma_{\zb v}$ is a geodesic on $[0,\epsilon]$.
    Then the (geometric) tangent vector of $\gamma_{\zb v}$ is given by
    \begin{equation}
        \dot \gamma_{\zb v} (t) 
        =
        (\pi_1 + t \, \pi_2, \pi_2)_\# \zb v,
        \qquad t \in [0,\epsilon).
    \end{equation}
\end{lemma}

\begin{proof}
Let $t\in[0,\epsilon)$ and define 
$\zb v_t\coloneqq (\pi_1 + t \, \pi_2, \pi_2)_\# \zb v$. 
By definition, we have
  \begin{align}
    \gamma_{\zb v_t}(s)
    &=(\pi_1+s\,\pi_2)_\#\zb v_t=(\pi_1+s\,\pi_2)_\#(\pi_1 + t \, \pi_2, \pi_2)_\# \zb v
    \\
    &=(\pi_1+(t+s)\,\pi_2)_\#\zb v=\gamma_{\zb v}(s+t).
  \end{align}
Since $\gamma_{\zb v}$ is a geodesic on $[0,\epsilon]$ and $t<\epsilon$, this implies that $\gamma_{\zb v_t}$ is a geodesic on $[0,t-\epsilon]$.
In particular, it holds $\zb v_t\in\zb G(\gamma_{\zb v_t}(0))\subset \zb \T_{\gamma_{\zb v_t}(0)}\P_2(\R^d)$.
Consequently, \eqref{eq:consistency_tangent_vector} implies $
\dot \gamma_{\zb v_t}(0)=\zb v_t
$.
On the other hand, it follows from
$\gamma_{\zb v}(t+s)=\gamma_{\zb v_t}(s)$ and the definition \eqref{eq:geomTangentVector} of the tangent vector that 
$\dot \gamma_{\zb v}(t)=\dot \gamma_{\zb v_t}(0)$ which yields the assertion 
$\dot \gamma_{\zb v}(t) = \zb v_t$.
\end{proof}

For reparameterization, we need the following chain rule of differentiation which proof is given in Appendix~\ref{proof:chain_rule}.

\begin{lemma}
  \label{lem:chain_rule}
  Let $\gamma : [0, T)\to\mathcal \P_2(\R^d)$, $T>0$ and $f \colon J\to [0,T)$ be differentiable and monotone increasing. 
  If the tangent vector of $\gamma$ at $f(t)$, $t \in J$, exists,
  then it holds
  \begin{equation*}
    \dot \nu (t) =\dot f(t)\cdot \dot \gamma(f(t)),\quad \nu(t)\coloneqq \gamma(f(t)).
  \end{equation*}
\end{lemma}

\section{Wasserstein Steepest Descent Flows}\label{sec:frech-diff}
In this section, we provide an alternative view on Wasserstein gradient flows  \eqref{wgf} 
based on the geometric interpretation that at any point $t\ge 0$ the tangent vector $\dot \gamma(t)$ 
points into an appropriately defined direction of steepest descent. 
Our approach allows the use of Euler forward schemes which are
often easier to implement in comparison to MMSs, which are based on the Euler backward scheme. In particular, the computation of particle gradient flows by simple gradient descent methods, can be seen as space and time discretization of the Euler forward scheme, see Section \ref{sec:R}.
We like to mention that measure differential equations with a different
definition of the ``solution'' inclusive Euler forward schemes were considered, e.g. in \cite{Pi19}.

For $\F\colon \P_2(\R^d) \to (-\infty,+\infty]$, we define the direction of steepest descent
using the following two notations of directional derivatives, where
the naming is adopted from \cite[§~1.2]{BookDeRu00}.
First, 
we consider the derivative along the curves $\gamma_{\zb v}$, where $\zb v$ belongs to the (geometric) tangent space.
More precisely, the \emph{Dini derivative}
of $\F\colon \P_2(\R^d) \to (-\infty,+\infty]$ at $\mu \in \dom \F$
in direction $\zb v \in \zb \T_{\mu}(\P_2(\R^d))$
is defined (if it exists) by
\begin{equation}        
\label{eq:dirDF}
  \D_{\zb v} \F(\mu) 
  \coloneqq 
  \lim_{t \to 0+} 
  \frac{\F(\gamma_{\zb v}(t)) - \F(\mu)}{t} 
  = 
  \frac{\d }{\d t} \,
  \F \circ \gamma_{\zb v}(t) \Big|_{t=0+}.
\end{equation}
Unfortunately, already in Euclidean spaces the derivative of a function along a curve $\gamma$ at $t$
does not necessarily coincide with the Dini derivative in direction of the tangent of $\gamma$ at $t$.
Therefore, we will need a more technical definition.
The \emph{lower/upper Hadamard derivative} of 
$\F\colon\P_2(\R^d) \to (-\infty,+\infty]$ 
at $\mu \in \dom \F$
in direction $\zb v \in \zb \T_{\mu}\P_2(\R^d)$
is defined by
{\small
\begin{align}  \label{eq:dirHF+-}
  \HD_{\zb v}^{-} \F(\mu)
  \coloneqq
  \smashoperator{\liminf_{\substack{\zb w \to \zb v, \, t \to 0+,\\ 
  \gamma_{\zb w} |_{[0,t]} \, \text{is geodesic}}}} 
  \frac{\F(\gamma_{\zb w}(t)) - \F(\mu)}{t}, 
  \quad
  \HD_{\zb v}^{+} \F(\mu) 
  \coloneqq
  \smashoperator{\limsup_{\substack{\zb w \to \zb v, \, t \to 0+,\\ 
  \gamma_{\zb w} |_{[0,t]} \, \text{is geodesic}}}}
  \frac{\F(\gamma_{\zb w}(t)) - \F(\mu)}{t}
\end{align}
}%
and the \emph{Hadamard derivative} (if the upper and lower limit coincide) by 
\begin{equation}
  \label{eq:dirHF}
  \HD_{\zb v} \F(\mu)
  \coloneqq 
  \lim_{\substack{\zb w \to \zb v, \, t \to 0+,\\ 
  \gamma_{\zb w} |_{[0,t]} \, \text{is geodesic}}} 
  \frac{\F(\gamma_{\zb w}(t)) - \F(\mu)}{t},
\end{equation}
where the convergence $\zb w \to \zb v$ is with respect to $W_\mu$.
The functional $\F$ is called \emph{Dini} or \emph{Hadamard differentiable} at $\mu$ 
if its Dini or Hadamard derivative exists 
for all directions $\zb v \in \zb \T_{\mu}\P_2(\R^d)$.
Note that all these directional derivatives are positively homogeneous (of degree 1) in $\zb v$.
If $\F$ is Hadamard differentiable, 
then it is also Dini differentiable and Hadamard and Dini derivative coincide. 
For locally Lipschitz continuous functions we have also the opposite direction.
Recall that
a function $\F \colon \P_2(\R^d) \to (-\infty,+\infty]$ is called
\emph{locally Lipschitz continuous at $\mu \in \P_2(\R^d)$},
if there exist $L,r>0$ such that
\begin{equation}
    |\F(\nu_1) - \F(\nu_2)| \le L W_2(\nu_1, \nu_2), \qquad
    \nu_1,\nu_2 \in B_r(\mu) 
\end{equation}
for all $\nu_1,\nu_2 \in B_r(\mu) \coloneqq  \{ \nu \in \P_2(\R^d): W_2(\nu, \mu) < r \}  \subset \dom \F$ and locally Lipschitz continuous, if this holds true for all $\mu \in \dom \F$.
Note that if $\F\colon \P_2(\R^d) \to (-\infty, +\infty]$ is locally Lipschitz, then it is also coercive since
\begin{equation}
  \F(\delta_0)-\F(\mu)\leq LW_2(\mu,\delta_0),
\end{equation}
i.e.
\begin{equation}
  F(\mu)\geq\F(\delta_0)-LW_2(\mu,\delta_0)\geq\F(\delta_0)-L\sqrt{r}
\end{equation}
for all $\mu$ with $W_2^2(\mu,\delta_0) = \int_{\R^d} \|x\|_2^2 \, \d \mu(x) \leq r$. 
Then it is not hard to show the following proposition. The proof is outlined in Appendix~\ref{proof:dini-hard}.

\begin{proposition}  \label{prop:dini-hard}
  Let $\F \colon \P_2(\R^d) \to \R$ be locally Lipschitz continuous around $\mu \in \P_2(\R^d)$.
  If $\D_{\zb v} \F(\mu)$ exists for $\zb v \in \zb \T_{\mu}\P_2(\R^d)$, 
  then $\D_{\zb v} \F(\mu) = \HD_{\zb v} \F(\mu)$.
\end{proposition}

For $\F \colon \P_2(\R^d) \to (-\infty+\infty]$,
the \emph{set of directions of steepest descent} at $\mu \in \dom \F$ is defined by
\begin{align}
\label{def:HD-}
  \HD_{-} \F(\mu) 
  & \coloneqq 
    \Big\{
    \left( \HD_{\zb v}^- \F(\mu)\right)^- \cdot \zb v :
    \zb v \in \argmin_{\substack{\zb w \in \zb \T_{\mu}\P_2(\R^d),\\ W_\mu(\zb w,\zb 0_\mu) = 1}} \HD_{\zb w}^- \F(\mu) \Big\},
\end{align}
if $\HD_{\zb v}^- \F(\mu)$ exists in $(-\infty,\infty]$ for all $\zb v \in \zb \T_{\mu}\P_2(\R^d)$,
where $(t)^- \coloneqq \max\{-t,0\}$ for $t \in \R$.
There may be no minimal direction,
i.e. $\HD_- \F(\mu)$ may be empty. 
On the basis of the introduced directional directions,
we are now interested in curves whose tangent $\dot \gamma(t)$ points into the direction of steepest descent.

\begin{definition} \label{def:wsgf}
A locally absolutely continuous curve
$\gamma \colon [0,+\infty) \to \P_2(\R^d)$ 
is called a \emph{Wasserstein steepest descent flow with respect to} $\F$ 
if $\dot \gamma(t)$ exists and satisfies
\begin{equation}         
  \label{eq:geomWGF}
  \dot \gamma(t) \in \HD_{-}\F(\gamma(t)), \qquad t \in [0,+\infty).
\end{equation}
\end{definition}

It is an open question if every Wasserstein steepest descent flow \eqref{eq:geomWGF} also satisfies the (weaker) Wasserstein gradient flow equation \eqref{wgf}.
Note that steepest descent directions can exist also in cases where the so-called extended Fr\'echet subdifferential related to the Wasserstein gradient flow  is empty.
However, under certain assumptions on $\F$, there exists a unique Wasserstein steepest descent flow and it coincides with the Wasserstein gradient flow of $\F$ in \eqref{wgf} \emph{for all} $t \in [0,+\infty)$.

\begin{proposition}
\label{ggd_case}
Let $\F \colon \P_2(\R^d) \to \R$ be locally Lipschitz continuous 
and $\lambda$-convex along generalized geodesics and  $\mu_0\in\P_2(\R^d)$.
Then, there exists a unique Wasserstein steepest descent flow of $\F$ starting at $\mu_0$. 
Moreover, it coincides with the unique Wasserstein gradient flow 
of $\F$ starting at $\mu_0$ determined by Theorem \ref{thm:existence_gflows_ggd}. 
\end{proposition}

The assumption of the local Lipschitz continuity can be weakened. However, as the exact formulation 
of the weaker assumptions requires some more technical notations, we include the more general version of the proposition as well as the proof in \ref{proof:ggd_case}.

\begin{remark}
   Definition~\ref{def:wsgf} allows the existence of Wasserstein gradient flows $\gamma$ with initial point $\mu_0 = \lim_{t \to 0+} \gamma(t) \not \in \dom \F$ (if it exists) and slope $\lim_{t \to 0+} \| \dot \gamma(t) \|_{\gamma(t)} = \infty$. That means the steepest descent direction at $\mu_0$ may not exist. 
  In Definition~\ref{def:wsgf} we excluded such curves, since we assume the existence of tangent velocity plans $\dot \gamma(t)$ for any $t \ge 0$.
\end{remark}

\section{Discrepancies} \label{sec:discr}
In this paper, we are interested in Wasserstein flows 
of so-called discrepancies defined with respect to kernels.
We restrict our attention to symmetric and \emph{conditionally positive definite kernels} 
$K\colon\R^d \times \R^d \to \R$ of order one, i.e.,
for any $n \in \N$, any pairwise different points $x^{1},\dots,x^{n} \in \R^d$ and any
$a_1, \dots, a_n \in \R$ with $\sum_{i=1}^{n} a_i = 0$
the following relation is satisfied:
    \begin{equation}        \label{eq:CPD}
        \sum_{i,j=1}^n a_i a_j K(x^i, x^j) \ge 0.
    \end{equation}
If \eqref{eq:CPD} is fulfilled for all $a_1, \dots, a_n \in \R$, the kernel 
is just called  \emph{positive definite}. We speak about \emph{(conditionally) strictly positive definiteness}
if we have strict inequality in \eqref{eq:CPD} except for all $a_j$, $j=1,\ldots,n$ being zero.
Examples of  strictly positive definite kernels are
the Gaussian $K(x_1,x_2) \coloneqq \exp(\|x_1-x_2\|/c)$, $c>0$ 
and the inverse multiquadric
$K(x_1,x_2) \coloneqq \left( c^2 + \|x_1-x_2\|^2 \right) ^{-r}$, $c,r>0$.
Strictly conditionally positive definite kernels are the multiquadric
$K(x_1,x_2) \coloneqq - \left( c^2 + \|x_1-x_2\|^2 \right) ^{-r}$, $r \in (0,1)$
and the  Riesz kernels
\begin{equation} \label{eq:riesz}
K(x_1,x_2) \coloneqq - \|x_1-x_2\|^r, \quad r \in (0,2),
\end{equation}
see \cite[p~115]{wendland2005} and
for more information on Riesz kernels \cite{BookRoWe}.

The \emph{$L_2$-discrepancy} $\mathcal D_K^2\colon\P(\R^d) \times \P(\R^d) \to \R$ 
between two measures $\mu, \nu \in \P(\R^d)$ is defined by
    \begin{equation}         \label{eq:DK2}
        \mathcal D_K^2(\mu, \nu) \coloneqq \mathcal E_K (\mu - \nu) \quad 
    \end{equation}
    with the so-called \emph{$K$-energy} on signed measures
    \begin{equation}         \label{eq:EK}
        \mathcal E_K(\sigma) \coloneqq \frac12 \int_{\R^d} \int_{\R^d} K(x_1,x_2) \, \d \sigma(x_1) \d \sigma(x_2), \qquad \sigma \in \mathcal M(\R^d).
    \end{equation}
The relation between discrepancies and Wasserstein distances is discussed in     \cite{NS2023}.
For fixed $\nu\in \P(\R^d)$, the $L_2$-discrepancy is a functional in $\mu$ and
can be decomposed as
\begin{equation}
    \label{eq:dis-decomp}
    \F_\nu(\mu) = \mathcal D_K^2(\mu, \nu) 
    = \mathcal E_K(\mu) + \V_{K, \nu}(\mu) + \underbrace{\mathcal E_K(\nu)}_{\text{const.}}
\end{equation}
with the \emph{interaction energy}  on probability measures
\begin{align} \label{eq:interaction}
     \mathcal E_K(\mu) &= \frac12 \int_{\R^d} \int_{\R^d} K(x_1,x_2) \, \d \mu(x_1) \d \mu(x_2), \quad \mu \in \P_2(\R^d)
\end{align}
and the \emph{potential energy} of $\mu$
with respect to the \emph{potential} of $\nu$,
\begin{align}\label{eq:potential}
    \V_{K, \nu}(\mu) &\coloneqq \int_{\R^d} V_{K, \nu}(x_1) \d \mu(x_1), 
    \quad 
    V_{K, \nu}(x_1) \coloneqq - \int_{\R^d} K(x_1,x_2) \d\nu(x_2).       
\end{align}

By the following proposition, the discrepancy of the Riesz kernel \eqref{eq:riesz} with $r \in [1,2)$
is  locally Lipschitz continuous in each argument for $r \in [1,2)$.
The proof is given in Appendix~\ref{proof:lip-int-pot}.

\begin{proposition}   \label{prop:lip-int-pot}
  For the Riesz kernel \eqref{eq:riesz} with $r \in [1,2)$, 
  the interaction energy $\mathcal E_K$ in \eqref{eq:interaction}
  and 
  the potential energy $\V_{K,\nu}$, $\nu \in \P_2(\R^d)$, in 
  \eqref{eq:potential}
  are locally Lipschitz continuous.
\end{proposition}

Since the negative Riesz kernel \eqref{eq:riesz} is
convex for $r \in [1,2)$, the negative interaction energy $- \mathcal E_K$ is convex along generalized geodesics by 
\cite[Prop~9.3.5]{BookAmGiSa05}. Similarly, the potential energy $\V_{K,\nu}$ is convex along (generalized) geodesics by \cite[Prop~9.3.2]{BookAmGiSa05}.
However, by the following proposition, $\mathcal E_K$ itself
and then discrepancies $\mathcal D_K^2$ are not $\lambda$-convex along geodesics.

\begin{proposition}  \label{prop:conv-int-pot}
  Let $K$ be the Riesz kernel \eqref{eq:riesz} on $\R^d$, $d \ge 2$. Then we have for any $\lambda \in \R$ the following:
  \begin{enumerate}[\upshape(i)]
      \item The interaction energy $\mathcal E_K$ 
      is not $\lambda$-convex along geodesics.
      \item The discrepancy $\mathcal D^2_K(\cdot, \delta_x)$, $x \in \R^d$ 
      is not $\lambda$-convex along geodesics.
  \end{enumerate}
\end{proposition}

\begin{proof}
\textbf{Part (i).}
To this end, we consider the line segments
\begin{equation}
    x_s(t) \coloneqq \bigl(s, (1-2t)\, \tfrac{s}{2}, 0 , \dots, 0\bigr) \in \R^d,
    \quad
    t \in [0,1],
\end{equation}
and the geodesics between
$\mu
\coloneqq
\tfrac 1 2 \, \delta_{0}
+ \tfrac 1 2 \, \delta_{x_s(0)}$
and
$\nu
\coloneqq
\tfrac 1 2 \, \delta_{0}
+ \tfrac 1 2 \, \delta_{x_s(1)}$
for $s > 0$.
Since the unique optimal transport between $\mu$ and $\nu$
is induced by the map $0 \mapsto 0$ and $x_s(0) \mapsto x_s(1)$,
these geodesics may be written as 
$\gamma_s(t)
\coloneqq
\tfrac 1 2 \, \delta_{0}
+ \tfrac 1 2 \, \delta_{x_s(t)}$,
see Proposition~\ref{prop:geo}.
Furthermore,
the Wasserstein distance is given by
$W_2^2(\mu, \nu)
= \tfrac 1 2 \, s^2$. 
Evaluating the interaction energy gives
\begin{equation*}
\mathcal E_K(\mu)
=
\mathcal E_K(\nu)
=
  -\frac{\| x_s(0) \|^r}{4}
=
  - \bigl(\tfrac{5}{4}\bigr)^{\frac{r}{2}} \, \frac{s^r}{4}
\quad\text{and}\quad
\mathcal E_K\bigl(\gamma_s\bigl(\tfrac12\bigr)\bigr)
=
- \frac{s^r}{4}.
\end{equation*}
If $\mathcal E_K$ is $\lambda$-convex along $\gamma_s$,
then, for $\lambda = \tfrac{1}{2}$, it has to fulfill
\begin{equation*}
-\frac{s^r}{4}
\le
- \bigl(\tfrac{5}{4}\bigr)^{\frac{r}{2}} \, \frac{s^r}{4}  
- \lambda\, \frac{s^2}{16} 
\quad\text{and thus}\quad
\lambda
\le
\Bigl(1  - \bigl(\tfrac{5}{4}\bigr)^{\frac{r}{2}} \Bigr)
\,
\frac{4}{s^{2-r}}.
\end{equation*}
Considering the limit $s \to 0+$,
we notice that
$\lambda$ cannot be bounded from above.

\noindent
\textbf{Part (ii).}
Without loss of generality,
we consider the case $x \coloneqq -e_1$,
where $e_1 \in \R^d$ is the first unit vector.
Then we obtain
\begin{equation}
    \D^2_K(\gamma_s(t), \delta_{-e_1})
    =
    \frac{\|e_1\|^r}{2}
    + \frac{\|e_1 + x_s(t)\|^r}{2}
    - \frac{\|x_s(t)\|^r}{4}
\end{equation}
and thus
\begin{align}
    \D^2_K(\mu, \delta_{-e_1})
    = \D^2_K (\nu, \delta_{-e_1})
    &= \frac{1}{2}
    + \bigl(\bigl(1+ \tfrac{1}{s}\bigr)^2 + \tfrac{1}{4}\bigr)^{\frac{r}{2}}
    \, \frac{s^r}{2}
    - \bigl(\tfrac{5}{4}\bigr)^{\frac{r}{2}}
    \, \frac{s^r}{4},
    \\
    \D^2_K \bigl(\gamma_s\bigl(\tfrac{1}{2}\bigr), \delta_{-e_1}\bigr)
    &= \frac{1}{2}
    + \bigl(\bigl(1+ \tfrac{1}{s}\bigr)^2\bigr)^{\frac{r}{2}}
    \, \frac{s^r}{2}
    - 
    \, \frac{s^r}{4}.
\end{align}
If $\D^2_K$ is $\lambda$-convex along $\gamma_s$,
then, for $\lambda = \tfrac{1}{2}$, it has to fulfill
\begin{equation}
    \frac{1}{2}
    + \bigl(\bigl(1+ \tfrac{1}{s}\bigr)^2\bigr)^{\frac{r}{2}}
    \, \frac{s^r}{2}
    - 
    \, \frac{s^r}{4}
    \le 
    \frac{1}{2}
    + \bigl(\bigl(1+ \tfrac{1}{s}\bigr)^2 + \tfrac{1}{4}\bigr)^{\frac{r}{2}}
    \, \frac{s^r}{2}
    - \bigl(\tfrac{5}{4}\bigr)^{\frac{r}{2}}
    \, \frac{s^r}{4}
    - \lambda \, \frac{s^2}{16}
\end{equation}
and thus
\begin{equation}
\label{eq:lambda-bound}
    \lambda
    \le
    \biggl[
    2 \, \bigl(\bigl(1+ \tfrac{1}{s}\bigr)^2 + \tfrac{1}{4}\bigr)^{\frac{r}{2}}
    - 2 \,\bigl(\bigl(1+ \tfrac{1}{s}\bigr)^2\bigr)^{\frac{r}{2}}
    + 1
    - \bigl(\tfrac{5}{4}\bigr)^{\frac{r}{2}}
    \biggr]
    \, 
    \frac{4}{s^{2-r}}.
\end{equation}
The first difference in the bracket may be estimated 
using the mean value theorem
and the monotonicity of the derivative of the exponential $x^{\frac{r}{2}}$,
which yields
\begin{equation}
    \Bigl|
    \bigl(\bigl(1+ \tfrac{1}{s}\bigr)^2 + \tfrac{1}{4}\bigr)^{\frac{r}{2}}
    - \bigl(\bigl(1+ \tfrac{1}{s}\bigr)^2\bigr)^{\frac{r}{2}}
    \Bigr|
    \le
    \frac{1}{4}
    \,
    \frac{r}{2}
    \,
    \frac{1}{\bigl( 1 + \frac{1}{s}\bigr)^{2-r}}.
\end{equation}
Thus the first difference converges to zero for $s \to 0+$.
Since the second difference is a negative constant,
the right-hand side of \eqref{eq:lambda-bound} tend to $-\infty$;
so $\lambda$ cannot be a global constant.
\end{proof}

\section{Interaction Energy Flows} \label{sec:mmd}
%
In this section, we focus on the explicit calculation of Wasserstein steepest descent flows 
of the interaction energy $\mathcal E_K$ for the Riesz kernels \eqref{eq:riesz} 
in particular, when starting at $\mu = \delta_p$, $p \in \R^d$.
Since  the functional $\mathcal E_K$ is no longer $\lambda$-convex along geodesics in $d \ge 2$ dimensions,  
the analysis of \cite{BookAmGiSa05} is not applicable for this case. 
However,  we show that the MMS \eqref{otto_curve} still converges, and that the limit curve is a Wasserstein steepest descent flow for $r \in [1,2)$. 
We like to mention that this strengthens the convergence result in \cite[Prop~4.2.2]{ThesisBo11}, 
where it is only shown that the MMS has a cluster point for $r=1$.

Recall that the set $\HD_-\mathcal E_K(\mu)$
of steepest descent directions is given by all 
$(\HD_{\zb v^*}^- \mathcal E_K(\mu))^{-} \cdot \zb v^*$, 
where $\zb v^*$ solves the constrained optimization problem
\begin{equation} \label{eq:constr}
\argmin_{\zb v\in\zb \T_\mu\P_2(\R^d)}\HD_{\zb v}^-\mathcal E_K(\mu)
\quad \text{s.t.}\quad 
\int_{\R^d\times \R^d}\|x_2\|_2^2 \, \d\zb v(x_1,x_2)=1.
\end{equation}
Therefore, we start by computing the directional derivatives of $\mathcal E_K$.

\begin{theorem}  \label{the:EKdirDF}
  Let $K$ be a Riesz kernel \eqref{eq:riesz}.
  Then the Hadamard derivative at $\delta_p$
  in direction $\zb v=\delta_p\otimes\eta\in\zb T_{\delta_p}\P_2(\R^d)$ is given by
    \begin{equation}
      \HD_{\zb v}\mathcal E_K(\delta_p)=
      \begin{cases}
        \mathcal E_K(\eta),&r=1,\\
        0,&r\in(1, 2).
      \end{cases}
    \end{equation}
  For $r\in(0,1)$, we have for the Dini derivative at $\delta_p$ in direction $\zb v=\delta_p\otimes\eta\in\zb T_{\delta_p}\P_2(\R^d)$ that
    \begin{equation}
      \HD_{\zb v}^-\mathcal E_K(\delta_p) \le
      \D_{\zb v}\mathcal E_K(\delta_p)=
      \begin{cases}
        -\infty,&\eta\not\in\{\delta_q:q\in\R^d\},\\
        0,&\eta\in\{\delta_q:q\in\R^d\}.
      \end{cases}
    \end{equation}
\end{theorem}

\begin{proof}
For $\zb v = \delta_p\otimes\eta$, we have $\gamma_{\zb v}(t) = \gamma_{t \cdot \zb v}(1) = ((1-t) p + t \Id)_{\#} \eta$ and then
\begin{align}
    \mathcal E_{K}(\gamma_{\zb v}(t))
    &= - \frac12 \int_{\R^d} \int_{\R^d} \|x_1 - x_2\|_2^r \, 
    \d [\gamma_{\zb v}(t)](x_1) \,
    \d [\gamma_{\zb v}(t)](x_2)
    \\
    &= -|t|^r \frac12 \int_{\R^d} \int_{\R^d} \|x_1 - x_2 \|_2^r \,
    \d \eta(x_1) \, \d \eta(x_2).
  \end{align}
Then the assertion follows for the Dini derivative by taking the right-hand side derivative of   $\mathcal E_{K}\circ \gamma_{\zb v}$  at $t=0$.
  By Proposition~\ref{prop:lip-int-pot}, we know that $\mathcal E_{K}$ is locally Lipschitz for Riesz kernels \eqref{eq:riesz} with $r \in [1,2]$. Then, by Proposition \ref{prop:dini-hard}, the Dini- and Hadamard derivative coincide which completes the proof.
\end{proof}

Part (i) of the theorem implies in particular, that there exists no Wasserstein steepest descent flow of $\mathcal E_K$ starting at $\delta_p$, if $K$ is a Riesz kernel \eqref{eq:riesz} with $r\in(0,1)$.
Moreover, for $r\in(1,2]$, a possible Wasserstein steepest descent flow starting at $\delta_p$ is given by the constant curve $\gamma(t)=\delta_p$. This curves are moreover Wasserstein gradient flows in the sense of \eqref{wgf}.

\begin{proposition}\label{cor:constant_WGF}
Let $K$ be the Riesz kernel \eqref{eq:riesz} for $r\in(1,2]$ and $p \in \R^d$. 
Then $\gamma\colon [0,+\infty)\to\P_2(\R^d) : t \mapsto \delta_p$ is a Wasserstein gradient flow.
\end{proposition}

\begin{proof}
Wlog, let $p=0$.
The velocity field $v_t\in \T_{\gamma(t)}\P_2(\R^d)$ corresponding to $\gamma$ is given by $v_t=0$.
Thus, we have to show that $0\in\partial \mathcal E_K(\delta_0)$, 
i.e., that $\liminf_{\nu\to\delta_0}{\mathcal E_K(\nu)}/{W_2(\nu,\delta_0)}\geq 0$.
To this end, we bound $\mathcal E(\nu)$ from below.
Since $\mathcal E_K$ is locally Lipschitz continuous, there exist $L > 0$ and $\epsilon>0$ 
such that for all $\nu\in \P_2(\R^d)$ with $W_2(\nu,\delta_0)\leq\epsilon$ 
it holds $\mathcal E_K(\nu)=\mathcal E_K(\nu)-\mathcal E_K(\delta_0)\geq - L\,W_2(\nu,\delta_0)\geq - L\epsilon$.
Moreover, we have by the definition of $\mathcal E_K$ for $\nu\in P_2(\R^d)$ and $c\geq 0$ that
\begin{equation}
\mathcal E_K((c\Id)_\#\nu)=c^r \mathcal E_K(\nu),\quad W_2((c\Id)_\#\nu,\delta_0)=c W_2(\nu,\delta_0).
\end{equation}
Now let $\nu \in \P_2(\R^d)$ with $W_2(\nu,\delta_0)\leq \epsilon$. 
Then, we get for $c={\epsilon}/{W_2(\nu,\delta_0)}$ that 
$W_2((c\Id)_\#\nu,\delta_0)=\epsilon$, which implies 
\begin{align}
    \mathcal E_K(\nu)
    &=\mathcal E_K((c^{-1}\Id)_\#(c\Id)_\#\nu)=c^{-r}\mathcal E_K((c\Id)_\#\nu)
    \\&
    \geq -c^{-r} L\epsilon=-L\epsilon^{1-r} W_2(\nu,\delta_0)^r.
\end{align}
Finally, we conclude
\begin{equation}
\liminf_{\nu\to\delta_0}\frac{\mathcal E_K(\nu)}{W_2(\nu,\delta_0)}\geq \liminf_{\nu\to\delta_0} -L\epsilon^{1-r} W_2(\nu,\delta_0)^{r-1}=0.
\end{equation}
\end{proof}

For simplicity, we restrict our attention 
to the case $p=0$, but similar conclusions can be drawn for arbitrary $p\in\R^d$. 
For $r = 1$ and $\mu=\delta_0$, 
we obtain by Theorem~\ref{the:EKdirDF} that 
the solution of constrained problem \eqref{eq:constr}
is given by $\zb v^* =\delta_0\otimes\eta^*$, where 
\begin{equation}\label{eq:constraint_problem}
\eta^*\in\argmin_{\eta\in\P_2(\R^d)}\mathcal E_K(\eta)\quad \text{s.t.}\quad\int_{\R^d}\|x\|_2^2 \, \d\eta(x)=1.
\end{equation}
On the other hand, 
the first step of the MMS \eqref{otto_curve} 
for $\F=\mathcal E_K$ starting at $\mu_0=\delta_0$ reads as 
\begin{equation}\label{eq:penalized_problem}
    \eta^*_\tau 
    \coloneqq 
    \prox_{\tau \mathcal E_K} (\delta_0)
    \in
    \argmin_{\eta_\tau\in\P_2(\R^d)}
    \mathcal E_K(\eta_\tau) 
    + \frac{1}{2\tau} 
    \underbrace{
    \int_{\R^d} \|x\|_2^2 \, \d\eta_\tau(x)
    }_{=W_2^2(\delta_0,\eta_\tau)},
\end{equation}
 which appears to be the penalized form of  \eqref{eq:constraint_problem}.
The minimization problem \eqref{eq:penalized_problem} is a special case 
of the classical potential theory problem
\begin{equation}\label{eq:opt_problem_penalized_lit}
\inf_{\eta \in \P(A)} \mathcal E_K(\eta) + \int_{A} V(x) \d \eta(x),
\end{equation}
where $\P(A)$ denotes the set of Borel probability measures on  $A \subset \R^d$,
see \cite{BookBoHaSa19,BookSaTo97}.
This problem is still a field of active research,
see, e.g., \cite{CaMaMoRoScVe21, CaSaWo22,MoRoSc19}.
If the minimizer of \eqref{eq:opt_problem_penalized_lit} exists, 
then it is called an equilibrium measure for the external field $V$.
In this context, the existence and uniqueness of solutions of the penalized problem \eqref{eq:penalized_problem}
are immediate consequences of well-established results in potential theory as we
will see in the next proposition.

\begin{proposition} \label{prop:penalized_problem}
Let $K$ be the Riesz kernel with $r \in (0,2)$ and $\tau > 0$ be given. Then  problem \eqref{eq:penalized_problem} has a unique solution $\eta^*_{\tau} \in \P_2(\R^d)$ which fulfills
\begin{enumerate}[\upshape(i)]
    \item $|\mathcal E_K(\eta^*_{\tau})| < \infty$,
    \item $\supp (\eta^*_{\tau})$ is compact,
    \item orthogonal invariance 
    $O_\#\eta^*_{\tau} = \eta^*_{\tau}$, 
    where
    $O \in \mathrm{O}(d) \coloneqq \{ O \in \R^{d\times d} \;:\; O^{\tT} O = I \}$.
\end{enumerate}
Furthermore, $\eta^*_\tau$ is the minimizer of \eqref{eq:penalized_problem}
if and only if
there exist $C_{K,\tau} \in \R$ such that
\begin{equation} \label{eq:opt_cond}
\begin{aligned}
  \int_{\R^d} K(x_1,x_2) \d \eta^*_{\tau}(x_2) + \frac{1}{2\tau} \|x_1\|_2^2 & \ge C_{K,\tau}, \qquad x_1 \in \R^d,\\
  \int_{\R^d} K(x_1,x_2) \d \eta^*_{\tau}(x_2) + \frac{1}{2\tau} \|x_1\|_2^2 & = C_{K,\tau}, \qquad x_1 \in \mathrm{supp}(\eta^*_{\tau}).
\end{aligned}
\end{equation}
\end{proposition}

\begin{proof}
    Considering \eqref{eq:opt_problem_penalized_lit} with $V(x) = \frac{1}{2\tau} \|x\|_2^2$ and $A = \R^d$,
    we obtain \eqref{eq:penalized_problem}
    up to the subtle difference
    that the minimization takes place over $\P(\R^d)$ instead of $\P_2(\R^d)$.
    The unique minimizer $\eta^*_\tau \in \P(\R^d)$ of this problem
    satisfies (i) and (ii), 
    see \cite[Cor~4.4.16(c)]{BookBoHaSa19},
    such that $\eta^*_{\tau} \in \P_2(\R^d)$ is also the (unique) minimizer of \eqref{eq:penalized_problem}.
    In particular, for any compact $A \subset \R^d$ with
    $\mathrm{supp} (\eta^*_{\tau}) \subset A$, the minimizer  of \eqref{eq:opt_problem_penalized_lit} is $\eta^*_{\tau}|_A$.
     Moreover, $\eta^*_\tau|_A$ is characterized by the optimality conditions \eqref{eq:opt_cond} (restricted to $x \in A$) , cf. \cite[Thm~4.2.14--Thm~4.2.16]{BookBoHaSa19}.
    Since $A$ can be arbitrarily large, the optimality conditions characterize also $\eta^*_{\tau}$ on $\R^d$, which concludes the proof.
\end{proof}

The following proposition shows how the penalized problem is related to the constrained.

\begin{proposition}\label{prop:penalizer_vs_constraint}
Let $K$ be a Riesz kernel \eqref{eq:riesz} with $r\in(0,2)$. 
\begin{enumerate}[\upshape(i)]
\item Then $\eta^*\in\P_2(\R^d)$ minimizes \eqref{eq:constraint_problem} 
if and only if 
$\eta^*_\tau \coloneqq (c_\tau \Id)_\#\eta^*$ 
minimizes \eqref{eq:penalized_problem}, where 
$c_\tau \coloneqq (-\tau \, r \, \mathcal E_K(\eta^*))^{1/(2-r)}$.
\item Vice versa, $\eta^*_\tau \in\P_2(\R^d)$ minimizes \eqref{eq:penalized_problem} if and only if 
$\eta^*=(c_\tau^{-1} \Id)_\#\eta^*_\tau$ minimizes \eqref{eq:constraint_problem}, where 
$c_\tau \coloneqq (\int_{\R^d}\|x\|_2^2 \, \d \eta^*_\tau(x))^{1/2}$.
\end{enumerate}
\end{proposition}

\begin{proof}
Let $\tau >0$ be the fixed step in \eqref{eq:penalized_problem}.
For any $c_\tau>0$ and $\eta_\tau \in\P_2(\R^d)$, we have 
$
\mathcal E_K((c_\tau \Id)_\#\eta_\tau)=c^r_\tau\mathcal E_K(\eta_\tau).
$
Then the objective in \eqref{eq:penalized_problem} can be rewritten as
\begin{align}
&\inf_{c_\tau> 0} \inf_{\eta_\tau\in\P_2(\R^d)} \mathcal E_K(\eta_\tau) + \frac{1}{2\tau} c_\tau^2
\quad\text{s.t.}\quad 
\int_{\R^d}\|x\|_2^2\, \d\eta_\tau(x)=c_\tau^2\label{eq:reformulated_opt_problem}\\
&=
\inf_{c_\tau>0}\inf_{\eta_\tau\in\P_2(\R^d)}
c_\tau^r\mathcal E_K((c_\tau^{-1}\Id)_\#\eta_\tau) 
+ \frac{1}{2\tau} c_\tau^2
\;\;\text{s.t.}\; 
\int_{\R^d}\|x\|_2^2 \, \d(c_\tau^{-1}\Id)_\#\eta_\tau(x)=1\\
&=\inf_{c_\tau>0}c_\tau^r \Big(\inf_{\eta\in\P_2(\R^d)} \mathcal E_K(\eta)\Big) + \frac{1}{2\tau} c_\tau^2\quad\text{s.t.}\quad \int_{\R^d}\|x\|_2^2\d\eta(x)=1,
\end{align}
where $\eta = (c_\tau^{-1}\Id)_\#\eta_\tau$.
Now, the set of minimizer with respect to $\eta$ is given by the set of all solutions $\eta^*$ of \eqref{eq:constraint_problem}. 
Thus, the set of solutions of \eqref{eq:penalized_problem} is empty if and only if the set of solutions of \eqref{eq:constraint_problem} is empty.
Further, setting the derivative with respect to $c_\tau$ to $0$ yields that the minimizer $c_\tau^*$ has to fulfills
  \begin{equation}
    0=r  (c_\tau^*)^{r-1}  \mathcal E_K(\eta^*) +\frac1\tau c_\tau^*
    \quad\Leftrightarrow\quad 
    c_\tau^*=(-\tau r \mathcal E_K(\eta^*))^{1/(2-r)}.
  \end{equation}
It is easy to verify that $\eta^*\neq\delta_0$ 
such that $\mathcal E_K(\eta^*)<0$ and $2-r>0$ ensuring that 
this expression is well-defined.
In summary, we obtain that $\eta^*_\tau$ is a solution of \eqref{eq:penalized_problem} 
if and only if $\eta^*_\tau =(c_\tau\Id)_\#\eta^*$ for some solution $\eta^*$ of \eqref{eq:constraint_problem} and $c_\tau=(-\tau \, r \, \mathcal E_K(\eta^*))^{1/(2-r)}$.
Following the arguments in the reverse direction and noting that \eqref{eq:reformulated_opt_problem} implies  $c_\tau^2=\int_{\R^d}\|x\|_2^2\d \eta^*$,
we obtain the second claim.
\end{proof}

In the following, we denote by $\mathcal U_A$ the uniform distribution on $A$.

\begin{theorem}    \label{thm:HD-spec}
Let $K$ be a Riesz kernel \eqref{eq:riesz} with $r\in (0,2)$. Then the solution $\eta_\tau^*$ of \eqref{eq:penalized_problem} is
 \begin{enumerate}[\upshape(i)]
 \item for  $d+r< 4$  given by
   \begin{equation*}
     \eta^*_{\tau} = \rho_{s_{\tau}} \mathcal U_{s_{\tau} \mathbb B^d}, \quad
     \rho_s(x) \coloneqq A_s \, (s^2 - \|x\|_2^2)^{1-\frac{r+d}{2}},\qquad x \in s \mathbb B^d,
   \end{equation*}
   where
   $\mathbb{B}^d \coloneqq \{x \in \R^d : \|x\|_2 \le 1\}$ and
   \begin{equation*}
     A_s \coloneqq \tfrac{\Gamma\left(\frac{d}{2} \right)}{\pi^{\frac{d}{2}} \Beta\bigl(\frac{d}{2},2-\frac{r+d}{2}\bigr)} s^{-(2-r)} , \quad 
     s_{\tau} \coloneqq  \left( \tfrac{\Gamma(2-\frac{r}{2}) \, \Gamma(\frac{d+r}{2}) \,r \, \tau}{\frac{d}{2} \,\Gamma(\frac{d}{2})} \right)^{\frac{1}{2-r}}
   \end{equation*}
   with the Beta function $\Beta$ and the Gamma function $\Gamma$,
 \item for $d+r \ge 4$ given by 
   \begin{equation*}
     \eta_\tau^*= \mathcal U_{c_\tau \mathbb S^{d-1}}, \quad
     c_\tau \coloneqq (-\tau \, r \, \mathcal E_K(\eta^*))^{1/(2-r)},
   \end{equation*}
   where  $\eta^* = \mathcal U_{\mathbb S^{d-1}}$,
   $\mathcal E_K(\eta^*)=-\tfrac12{_2F_1}\big(-\tfrac{r}{2},\tfrac{2-r-d}{2};\tfrac{d}{2};1\big)$
  with the hypergeometric function ${_2F_1}$
  and
  $\mathbb{S}^{d-1} \coloneqq \{x \in \R^d : \|x\| = 1\}$.
  \end{enumerate}
\end{theorem}

The proof is given in \ref{app_supp}. 
The special case $d=3$ and $r=1$ was recently also handled in \cite{ChSaWo22b}.
Note that for $d \ge 3$ and $r \in [1,2)$ the densities $\rho_s$ are not integrable anymore.

Interestingly, for $d = 3$, we observe a so-called \emph{condensation phenomenon} starting at $r=1$, where the absolutely continuous measure switches to a singular one. 
A similar phenomenon was recognized for the logarithmic kernel $K(x_1,x_2) = - \log\|x_1-x_2\|$, corresponding to $r=0$ and $d \ge 4$ in \cite[Thm~1.2:~(i)(b)]{ChSaWo22b}.

For the case $r=1$, in which we are mainly interested, 
we obtain the following analytic expressions for 
the solution $\eta^*$ of the constrained problem 
and the steepest descent direction $\HD_- \mathcal E_K (\delta_p)$.
The corollary straightforwardly follows 
from the relation 
$\eta^* = (c_\tau^{-1} \Id)_\#\eta^*_\tau$ in Proposition~\ref{prop:penalizer_vs_constraint} 
and since  $- \HD_{\delta_p \otimes \eta^*} \mathcal E_K(\delta_p)
      = -\mathcal E_K(\eta^*)$
      and
$\HD_- \mathcal E_K (\delta_p) = \left(- \HD_{\delta_p \otimes \eta^*} \mathcal E_K(\delta_p) \right) \cdot (\delta_p \otimes \eta^*)$  by Theorem \ref{the:EKdirDF}.

\begin{corollary}    \label{thm:HD-spec_0}
    Let $K$ be a Riesz kernel \eqref{eq:riesz} with $r=1$. Then the solution $\eta^*$ of \eqref{eq:constraint_problem} and the  steepest descent directions of $\mathcal E_K$ at $\delta_p$ are given as follows:
    \begin{enumerate}[\upshape(i)]
        \item For $d=1$: 
        $ \eta^*= \mathcal U_{[-\sqrt{3}, \sqrt{3}]}
        \; \text{ and } \;
            \HD_- \mathcal E_K (\delta_p)
            =
            \{\delta_p \otimes \mathcal U_{[-1,1]}\}
        $.
        \item For $d=2$: 
        $
            \eta^*=   \rho_{ \sqrt{3/2} }\, 
            \mathcal U_{\sqrt{3/2} \, \mathbb B^2}$
        and 
        $    \HD_- \mathcal E_K (\delta_p)
            =
            \{\delta_p \otimes \rho_{\frac{\pi}4} \mathcal U_{\frac{\pi}4 \mathbb{B}^2}\},
       $
       where $\rho_s$ is the density function
        \begin{equation}
            \rho_s(x) 
            \coloneqq
                \frac{1}{2\pi s} \bigl(s^2 - \|x\|_2^2\bigr)^{-\frac12}, \qquad x \in s \, \mathbb B^2.
        \end{equation}
        \item For $d \ge 3$:
        $
            \eta^*= \mathcal U_{\mathbb S^{d-1}}
        $
        and
        $
            \HD_- E_K (\delta_p)
            =
            \{\delta_p \otimes \mathcal U_{R_d \mathbb S^{d-1}} \},
        $
        where\\
        $
        R_d \coloneqq  \tfrac{1}{2} \,_2F_1(-\tfrac12, -\tfrac{d-1}{2}; \tfrac{d}{2}; 1)
        $
        with the hypergeometric function $_2F_1$.
    \end{enumerate}
\end{corollary}

There is the following interesting link between the measures $\eta^*$ in dimensions $d=1,2,3$ which states that they 
follow by projecting the measure from the higher dimensional space to the lower dimensional one. 

\begin{corollary}\label{cor:proj}
     Consider the rescaled measure $\mu_d \coloneqq (C_d \mathrm{Id})_{\#} \eta^*$ of Corollary~\ref{thm:HD-spec} with $r=1$, where $C_d > 0$ is chosen such that 
        \[
        \mathrm{supp} (\mu_d)
        = \begin{cases}
            \mathbb B^d, & d = 1,2,\\
            \mathbb S^2, & d=3.
        \end{cases}
        \]
        Then $\mu_d$ can be considered as a projection of $\mu_{d'}$ onto a $d$-dimensional subspace $X_d \subset \R^{d'}$, more precisely  
        $
        \mu_d = (\pi_{1,\dots,d})_{\#} \mu_{d'}
        $, $1 \le d \le d' \le 3$.
\end{corollary}
\begin{proof}
For the case $d=2$, we obtain for the sets
        \[
       A_{\theta_1, \theta_2}^{\varphi_1, \varphi_2}   \coloneqq  \{ x \in \R^2: x = (\sin(\theta) \cos(\varphi), \sin(\theta) \sin(\varphi)), \; \theta_1 \le \theta \le \theta_2,\,  \varphi_1 \le \varphi \le \varphi_2 \},
       \]
where 
$0 \le \theta_1 \le \theta \le \theta_2 \le \frac{\pi}{2}$, $\varphi_1 \le \varphi \le \varphi_2$
that
        \[
           (\pi_{1,2})_{\#} \mathcal U_{\mathbb S^2} (A_{\theta_1,\theta_2}^{\varphi_1,\varphi_2}) = \frac{2}{4\pi} 
           \int_{\theta_1}^{\theta_2} \int_{\varphi_1}^{\varphi_2} \d \varphi \sin(\theta) \d \theta
            = \frac{1}{2\pi} (\varphi_2 - \varphi_1) (\cos(\theta_1)-\cos(\theta_2)).
        \]
        Since 
        \begin{align}
           \rho_{1} \mathcal U_{\mathbb B^2}(A_{\theta_1,\theta^2}^{\varphi_1,\varphi_2}) 
           &= \frac{1}{2\pi} \int_{\sin(\theta_1)}^{\sin(\theta_2)} \int_{\varphi_1}^{\varphi_2} \d \varphi (1-r^2)^{-\frac12} r\d r \\
           &= \frac{1}{2\pi}(\varphi_2-\varphi_1)(\sqrt{1-\sin(\theta_1)^2} - \sqrt{1-\sin(\theta_2)^2}),
        \end{align}
        we have
        $
           (\pi_{1,2})_{\#} \mathcal U_{\mathbb S^2} (A_{\theta_1,\theta_2}^{\varphi_1,\varphi_2})
           =  \rho_{1} \mathcal U_{\mathbb B^2}(A_{\theta_1,\theta^2}^{\varphi_1,\varphi_2})
        $
        and arrive at the assertion since the sets $A_{\theta_1, \theta_2}^{\varphi_1, \varphi_2}$ generate the Borel $\sigma$-algebra on $\mathbb B^2$.
        
        For the case $d=1$, using
        $
           (\pi_{1})_{\#} \mathcal U_{\mathbb S^2} = (\pi_{1})_{\#} (\pi_{1,2})_{\#} \mathcal U_{\mathbb S^2} = (\pi_{1})_{\#} \rho_1 \mathcal U_{\mathbb B^2}
        $,
        it is sufficient to show that $(\pi_{1})_{\#} \rho_1 \mathcal U_{\mathbb B^2} = \frac{1}{2} \mathcal U_{[-1,1]}$. 
        This follows from integration of
        the density $\rho_1(x)$ along the lines $l_s = \{ x=(x_1,x_2) \in \R^2 \;:\; x_1 = s\}$ giving
        \[
          \frac{2}{2\pi} \int_{0}^{\sqrt{1-s^2}} (1 - (s^2 + t^2))^{-\frac12} \d t
          = \int_0^{s'} (s'^2 - t^2)^{-\frac12} \d t = \frac12, \qquad -1 < s < 1,
        \]
        which is the density of $\mathcal U_{\mathbb B^1}$.
\end{proof}

To determine the whole steepest descent flow,
we need also the steepest descent directions at more general measures than just point measures.
The proof is in \ref{proof:directional_derivatives_abs_cont}.

\begin{theorem}\label{thm:directional_derivatives_abs_cont}
  Let $K$ be a Riesz kernel \eqref{eq:riesz} with $r \in [1,2)$
  and $\mu \in \P_2(\R^d)$, where we assume that  $\mu(\{x\}) = 0$, $x \in \R^d$ in case $r=1$.
  Then the unique steepest descent direction is given by
    \begin{equation}
      \HD_-\mathcal E_K(\mu)=\{(\Id,-\nabla G)_\#\mu\},\qquad  G(x_1)\coloneqq \int_{\R^d} K(x_1,x_2) \d \mu(x_2).
    \end{equation}
\end{theorem}

For the one-dimensional setting,
a complete formula for the steepest descent direction was given in \cite[Prop~5.4]{BoCaFrPe15}.
Now we are in the position to show the existence of non-trivial Wasserstein steepest descent flows for $r \in [1,2)$.

\begin{theorem}\label{prop:EK_steepest_descent_flow}
Let $K$ be a Riesz kernel with $r \in [1,2)$ and $\eta^*$ be the unique solution of the constrained problem~\eqref{eq:constraint_problem}.
Then the curve $\gamma\colon[0,\infty)\to\P_2(\R^d)$ given by
\begin{equation}
\label{eq:EK_steepest_descent_flow}
 \gamma(t) \coloneqq (\alpha_t \mathrm{Id})_{\#} \eta^*, \qquad 
 \alpha_t \coloneqq \left(- t\, r(2-r) \, \mathcal E_K(\eta^*) \right)^{\frac{1}{2-r}},
 \end{equation}
is a steepest descent flow starting at $\gamma(0) = \delta_0$.
\end{theorem}

\noindent
The proof is given in \ref{app:desc_flow}.
\medskip

In the case $r=1$, we  obtain for example the curves
\begin{equation}\label{eq:JKO_flow_E1}
\gamma(t) = (-t\mathcal E_K(\eta^*)\Id)_\#\eta^* 
=
\begin{cases}
           \mathcal U_{[-t,t]}, & d = 1,\\
           \frac{\pi t}{8} \bigl(\frac{\pi^2t^2}{16}  - \|x\|_2^2\bigr)^{-\frac12} 
           \mathcal U_{t \frac\pi4 \mathbb B^2},
           & d = 2,\\
           \mathcal U_{t R_{d} \mathbb S^{d-1}}, & d \ge 3,
         \end{cases}
\end{equation}
where the constant $R_d$ is given in Corollary \ref{thm:HD-spec}(iii).

\begin{remark}
(i) With the same proof, we can see  that the curve \eqref{eq:EK_steepest_descent_flow} 
also fulfills the steepest descent condition
$
\dot\gamma(t)\in\HD_-\mathcal E_K(\gamma(t))
$
for $r \in (0,1)$ if  $t>0$. 
However, by Theorem~\ref{the:EKdirDF}, the set $\HD_-\mathcal E_K(\gamma(0))=\HD_-\mathcal E_K(\delta_0)$ is empty for $r\in(0,1)$,
so that this curve is not a Wasserstein steepest descent flow in the sense of Definition~\ref{def:wsgf}.

(ii) For $r \in (1,2)$, we obtain together with the trivial solution an infinite family of Wasserstein steepest descent flows starting at $\delta_0$. These are parameterized by the 'length of stay' $t_0\in\R_{\geq 0}$ at $\delta_0$ due to
\begin{equation}
    \label{eq:EK_descent_flow_family}
    \gamma(t)=
    \begin{cases}
        \delta_0,& $for $ t<t_0,\\
        (\alpha_{t-t_0}\Id)_\#\eta^*,&$for $t\geq t_0.
    \end{cases}
\end{equation}
\end{remark}

\begin{remark}[Relation to MMS and Wasserstein Gradient Flows]
In \cite{AHS2023}, 
the MMS steps for $\mathcal E_K$ are computed analytically,
and it turns out that the scheme converges to the curves in \eqref{eq:EK_steepest_descent_flow}.
Note that Theorem~\ref{thm:existence_gflows_ggd} here cannot be applied
since $\mathcal E_K$ is not $\lambda$-convex along geodesics.
Indeed, 
Proposition~\ref{cor:constant_WGF} shows the existence of Wasserstein gradient flows 
that cannot be represented as MMS limits.
Vice versa, it is an open question if limits of MMS are Wasserstein gradient flows. 
For this direction, the $\lambda$-convexity requirement 
can be weakened towards a regularity assumption by \cite[Thm 11.3.2]{BookAmGiSa05}.
Nevertheless, it is still unclear if $\mathcal E_K$ fulfills this regularity assumption.
\end{remark}

\section{Discrepancy Flows} \label{sec:R}

In the following, we determine steepest descent flows of the discrepancy functional $\mathcal F_\nu \coloneqq \mathcal D^2_K(\cdot, \nu)$ 
for the Riesz kernel $K$ with $r \in [1,2)$. 
For $r \in (1,2)$, where the Riesz kernel is differentiable, we characterize the Wasserstein steepest descent flow of $\mathcal F_{\delta_q}$ starting at
$\delta_p$. We provide a numerical simulation via particle flows for $r \in [1,2)$.
In contrast to the case $r \ge 2$, the particle explodes here.

The next theorem, which proof is given in \ref{proof:discr_2}, describes the steepest descent direction of the discrepancy functional.

\begin{theorem}    \label{thm:DKdir}
  Let $\F_{\nu}\coloneqq \mathcal D_K^2(\cdot,\nu)$, where $\nu \in \P_2(\R^d)$ and $K$ is the Riesz kernel \eqref{eq:riesz} with $r \in [1,2)$. 
  Then the following holds true.
  \begin{enumerate}[\upshape(i)] 
  \item
    For $\mu\in\P_2(\R^d)$, where $\mu(\{x\})=0$ for all $x \in \R^d$ in case $r=1$,  the unique steepest descent direction is given by
    \begin{equation}
      \HD_-\F_{\nu}(\mu)= \{(\mathrm{Id}, - \nabla G)_{\#} \mu\},
      \;
      G(x_1) \coloneqq   \int_{\R^d} K(x_1,x_2) \, \d \mu(x_2) + V_{K,\nu}(x_1).
    \end{equation}
  \item
    For $p \in \R^d$ with $\nu(\{p\}) = 0$, the steepest descent direction at $\delta_p$ is given by
    \begin{equation*}
      \HD_-\F_{\nu}(\delta_p) = 
      \begin{cases}
        \delta_p \otimes \big(-\mathcal E_K(\eta^*) \Id - \nabla V_{K,\nu}(p) \big)_\#\eta^*, & r=1,\\
        \delta_p \otimes \delta_{-\nabla V_{K,\nu}(p)}, & r \in (1,2),
      \end{cases}
    \end{equation*}
    where $\eta^*$ is defined in Corollary~\ref{thm:HD-spec} and $V_{K,\nu}$ in \eqref{eq:potential}.
  \end{enumerate}
\end{theorem}

Based on these steepest descent directions, we see in the next proposition that,
for differentiable Riesz kernels, there exists a
steepest descent flow for $\mathcal D_K^2(\cdot,\delta_q)$,
$q \in \R^d$,
which has the form of a particle flow.

\begin{proposition}
    \label{ex:one_particle_disc_flow}
    Let $\mathcal F_\nu \coloneqq \mathcal D^2_K( \cdot, \delta_q)$,
    where $q\in \R^d$, 
    and let $K$ be the Riesz kernel with $r \in (1,2)$.
    A Wasserstein steepest descent flow starting at $\delta_p$,
    $p \in \R^d$,
    is given by
    \begin{equation}
        \label{eq:gamma_one_particle}
        \gamma(t) \coloneqq 
        \begin{cases}
            \delta_{x(t)},& t \in [0,t_*),\\
            \delta_q, & t \in [t_*,\infty),
        \end{cases}
        \quad\text{with}\quad
        t_* \coloneqq \frac{\|q-p\|_2^{2-r}}{r(2-r)},
    \end{equation}
    where
    \begin{equation}
        \label{eq:par-curve}
        x(t) \coloneqq  q - \frac{q-p}{\|q-p\|_2} \left(
        \|q-p\|_2^{2-r} - r(2-r)t  \right)^{\frac{1}{2-r}}.
    \end{equation}
\end{proposition}

\begin{proof}
    The tangent vector of the curve $t \mapsto x(t)$ in \eqref{eq:par-curve} is given by
    \begin{align}
      \dot x(t) 
      &= r \, \frac{q-p}{\|q-p\|_2} \left( \|q-p\|_2^{2-r} - r(2-r)t \right)^{\frac{1}{2-r}-1}
      \\
      &= r \, (q - x(t)) \left( \|q-p\|_2^{2-r} - r(2-r)t \right)^{-1}.
    \end{align}
    Therefore, 
    the particle $x(t)$ solves the gradient flow equation
    \begin{equation}
    \label{eq:DK_one_particle_flow}
       \dot x(t) = - \nabla V_{K,\delta_q}(x(t)), \quad t \in  [0,t_*), \quad x(0) = p,
    \end{equation}
    where
    \begin{equation*}
      V_{K,\delta_q}(x) \coloneqq \|q-x\|_2^r,
      \quad 
      \nabla V_{K,\delta_q}(x) = - r (q-x) \|q-x\|^{r-2}, \quad x \in \R^d \setminus \{q\}.
    \end{equation*}
    Thus, by Theorem~\ref{thm:DKdir},
    the curve \eqref{eq:gamma_one_particle}
    is a Wasserstein steepest descent flow for $t \in [0,t_*)$.
    For $t \ge t_*$,
    we have $\dot \gamma(t) = \delta_q \otimes \delta_0$, 
    which is here the direction of steepest descent
    since $\delta_q$ is the global minimizer of $D_K^2(\cdot, \delta_q)$.
\end{proof}

Moreover, 
we expect that there exists an infinite family of Wasserstein steepest descent flows
similar to the family given in \eqref{eq:EK_descent_flow_family} for the interaction energy.
That means at any time point $0 \le t_0 < t_*$,
the point mass $\gamma(t)$ in \eqref{eq:gamma_one_particle} may explode
to an absolutely continuous measure leading to another Wasserstein steepest descent flow.
Unfortunately,
the analytic computation of the whole flow describing this effect 
is much more difficult than for the interaction energy.
Therefore,
we provide some numerical simulations using an Euler forward scheme.

\paragraph{Numerical simulation.}
Let $K$ be again the Riesz kernel \eqref{eq:riesz} with $r \in [1,2)$, and let $e_1$ be the first unit vector.
In the following, we want to approximate the discrepancy flow with respect to 
$\F_{\delta_{e_1}} =\mathcal D_K^2(\cdot, \delta_{e_1})$ in $\R^d$.
To this end, we restrict the set of feasible measures to the set of point measures located at exactly
$M$ points, i.e., to the set 
  \begin{equation}
    \mathcal S_M\coloneqq \Big\{\frac1M\sum_{i=1}^M \delta_{x_i}:x_i\in\R^d,x_i\neq x_j\text{ for all }i\neq j\Big\}.
  \end{equation}
Then, we compute the Wasserstein gradient flow of the functional
\begin{equation}
\F_M(\mu)\coloneqq \begin{cases}\mathcal D_K^2(\mu, \delta_{e_1}),&$if $\mu\in \mathcal S_M\\+\infty,&$otherwise.$\end{cases}
\end{equation}
By taking the mean field limit $M \to \infty$,
we expect that gradient flows with respect to $\F_M$ approximate the gradient flows with respect to $\F_{\delta_{e_1}}$.
In order to compute the gradient flows with respect to $\F_M$ for some fixed $M\in\N$,
we consider the (rescaled) particle gradient flow for the function
$F_M\colon \R^{dM} \to [0,\infty)$ given by
\begin{equation}
    F_M(x)
    \coloneqq \F_{\delta_{e_1}}\left(\frac1M \sum_{i=1}^M \delta_{x_i}\right) = -\frac{1}{2M^2} \sum_{i,j=1}^M \|x_i - x_j\|_2^r + \frac{1}{M} \sum_{i=1}^M \|x_i - e_1 \|_2^r.
\end{equation}  
More precisely, we are interested in solutions of the ODE
\begin{equation}
\label{eq:ODE}
    \dot u = - M \nabla F_M(u).
\end{equation}
Then, 
we see that the solutions $u=(u_1,...,u_M)\colon(0,\infty)\to\R^{dM}$
of \eqref{eq:ODE} and the Wasserstein gradient flows $\gamma\colon(0,\infty)\to\P_2(\R^d)$ with respect to $\F_M$
are related by
  \begin{equation}
    \gamma(t)\coloneqq  \frac1M\sum_{i=1}^M\delta_{u_i(t)}.
\end{equation}
For further details see \cite{AHS2023}.
Finally, we approximate the solutions of \eqref{eq:ODE} by the explicit Euler-forward scheme
\begin{equation}
\label{eq:euler_scheme}
    x^{(n+1)} \coloneqq - \tau^{(n)} M \nabla F_M(x^{(n)}), \qquad n \in \N_0.
\end{equation}

\begin{figure}[t]
    \begin{center}
    \includegraphics[width=1\textwidth]{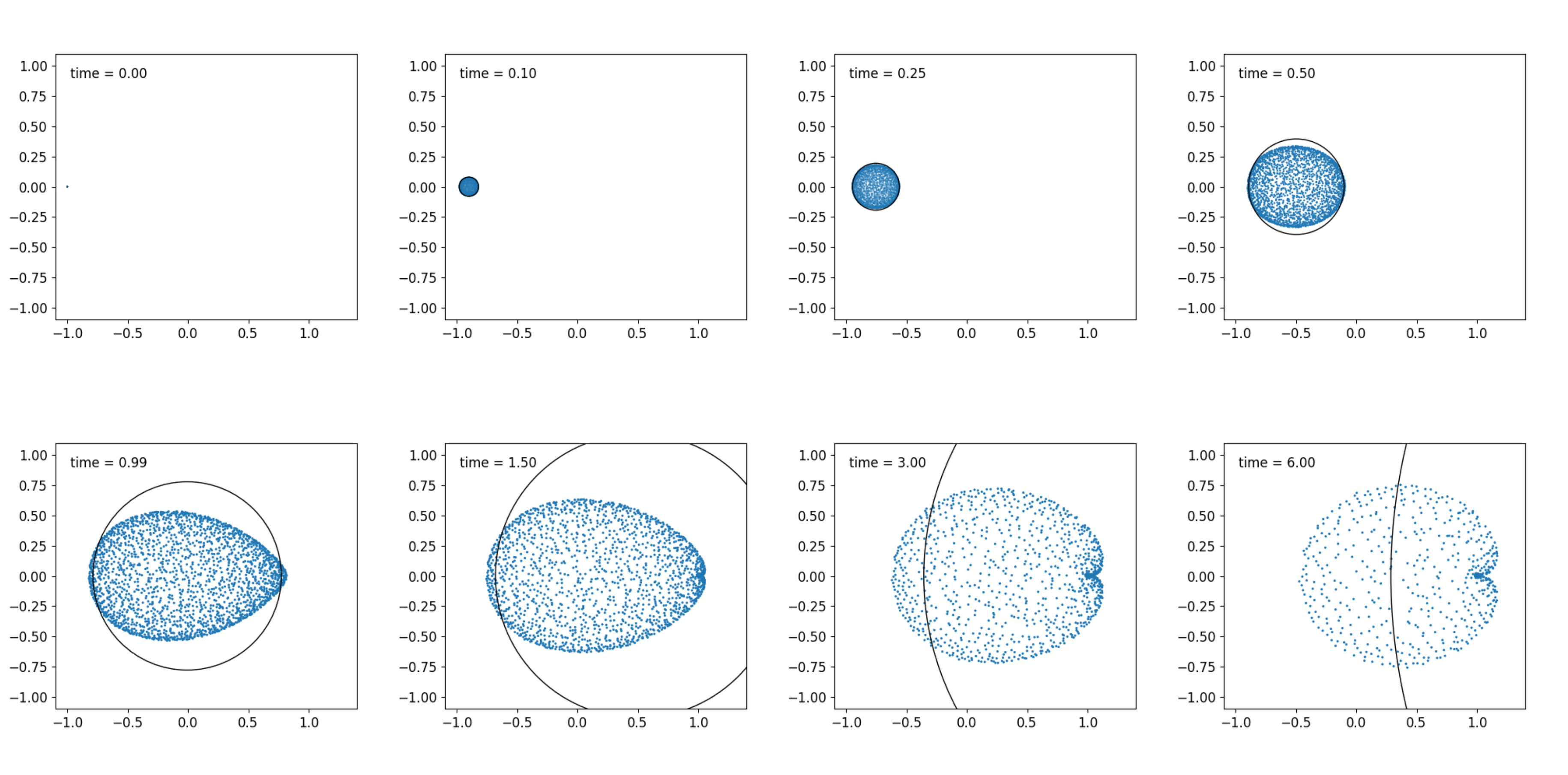}
    \end{center}
    \caption{
    2D particle gradient flow of $\mathcal D_{K}^2(\cdot, \delta_{e_1})$ for the Riesz kernel with $r=1$ starting around  $\delta_{-e_1}$. 
		The black circles depict the border of $\supp \gamma_{\zb v}(t)$ related to the steepest descent direction $\zb v$ at $t=0$ given in 
		\eqref{eq:approx_one_particle_disc_flow_geodesic}.}
    \label{fig:num_grad_flow_2d_delta_delta}
\end{figure}

\begin{figure}[t]
    \begin{center}
        \includegraphics[width=1\textwidth]{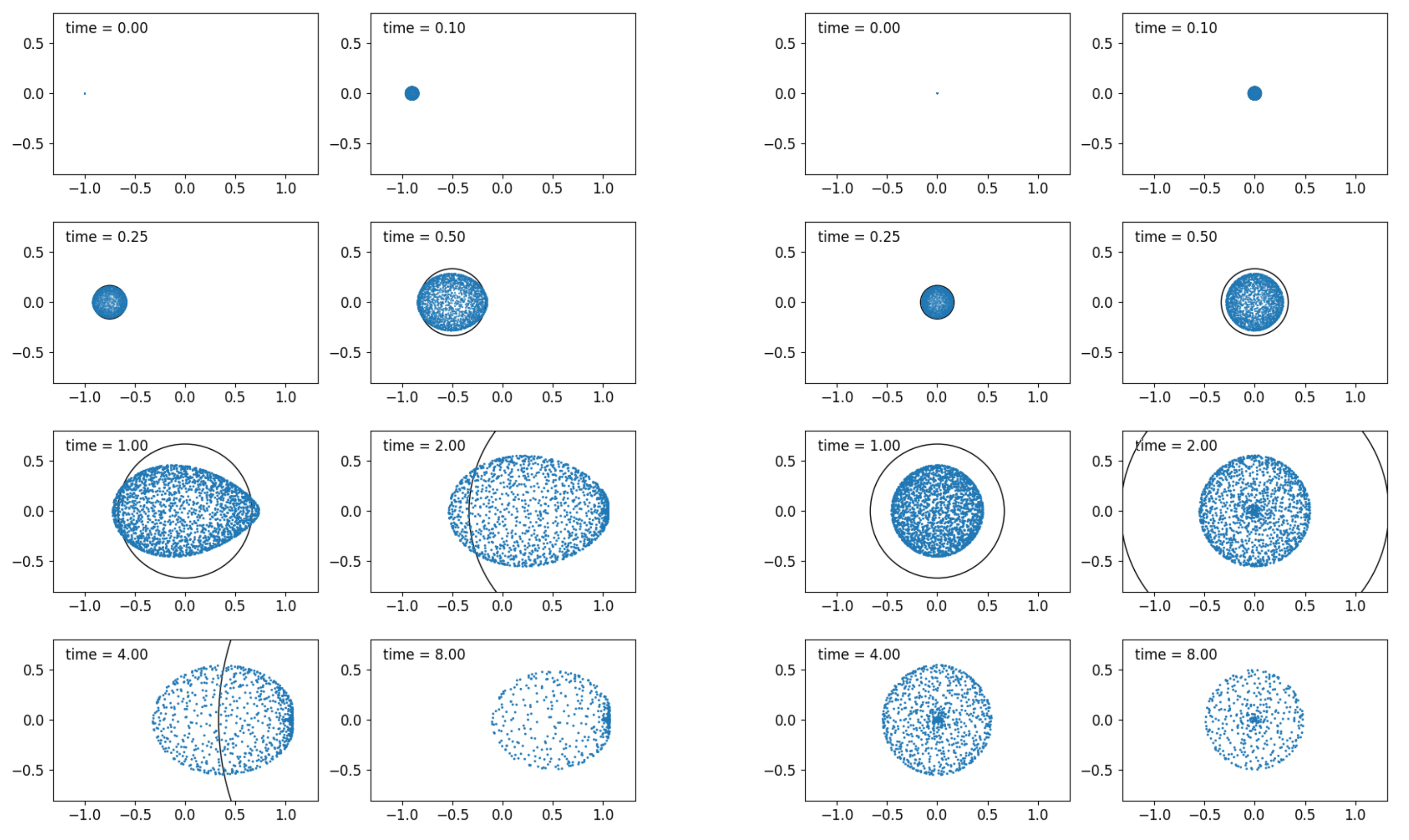}
    \end{center}
    \caption{
    3D particle gradient flow of $\mathcal D_{K}^2(\cdot, \delta_{e_1})$ for the Riesz kernel with $r=1$ starting around  $\delta_{-e_1}$. 
    The left columns show the projection to the $x_1x_2$-plane, the right columns to the $x_3x_2$-plane.
    The black circles depict the border of $\supp \gamma_{\zb v}(t)$ related to the steepest descent direction $\zb v$ at $t=0$ given in 
		\eqref{eq:approx_one_particle_disc_flow_geodesic}.}
    \label{fig:num_grad_flow_3d_delta_delta}
\end{figure}
 
\begin{figure}[t]
    \begin{center}
    \includegraphics[width=1\textwidth]{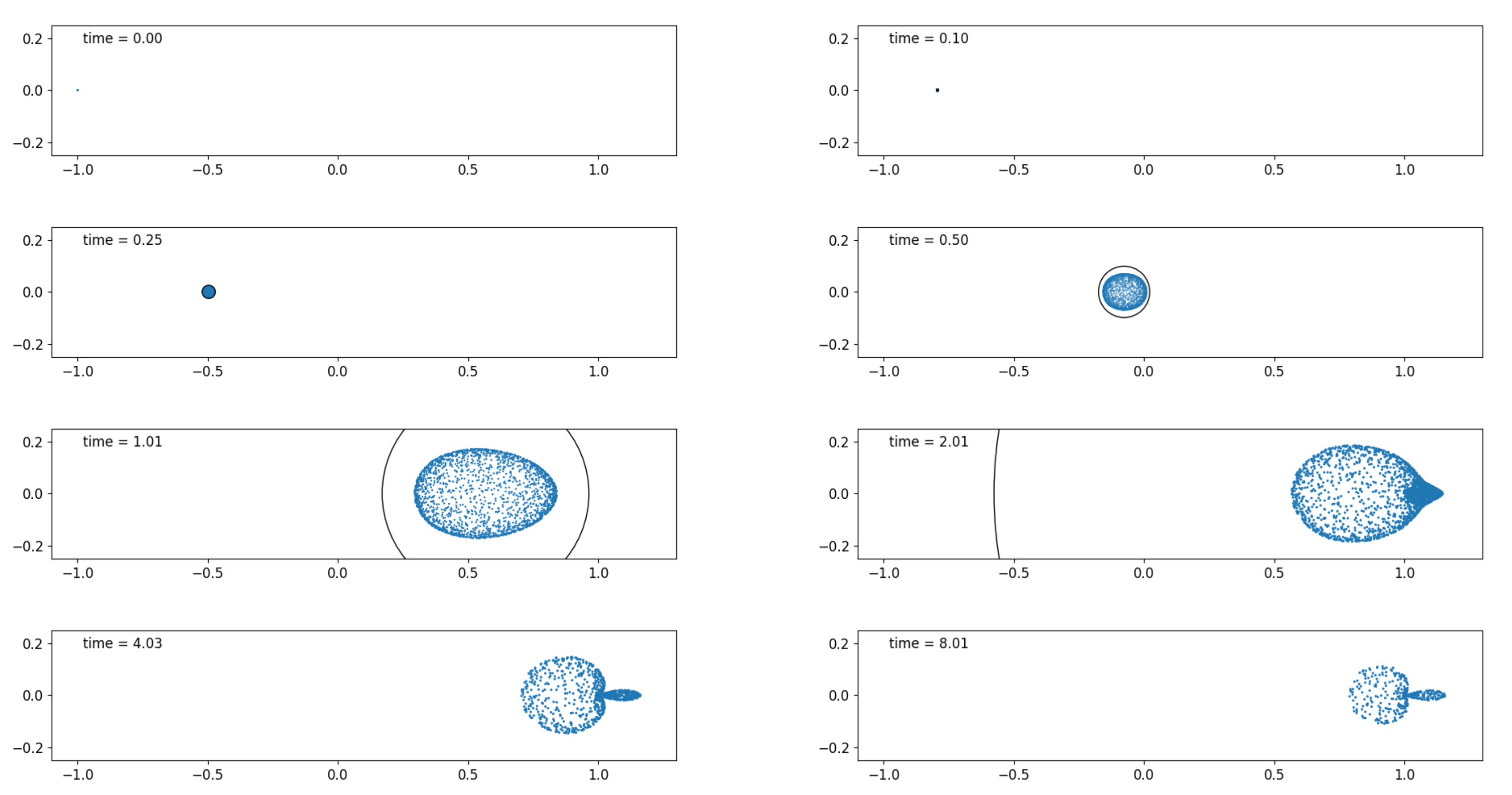}
    \end{center}
    \caption{
    2D particle gradient flow of $\mathcal D_{K}^2(\cdot, \delta_{e_1})$ for the Riesz kernel with $r=\frac{3}{2}$ starting around  $\delta_{-e_1}$.}
    \label{fig:num_grad_flow_2d_delta_delta_r_3_2}
\end{figure}

\begin{example}
    We take $\tau^{(0)} \coloneqq \frac{1}{10 M}$ 
    and set $\tau^{(n)} \coloneqq \min\{n \tau^{(0)}, \tau_{\max} \}$,
    where $\tau_{\max} \sim \frac{1}{M}$ is some maximal step size.
    We aim to compute gradient flows starting at $\delta_{-e_1}$.
    In order to start in a set $\mathcal S_M$, we first 
    perform a forward step from $\delta_{-e_1}$ in the known steepest descent direction, see Theorem \ref{thm:DKdir}.
    So the initial points $x_i^{(0)}$ for $d=2$ and $d=3$ are nearly distributed on a small ball and  sphere, respectively.
    We apply the explicit Euler scheme \eqref{eq:euler_scheme} for $M=2000$ points
    for initial points in a cube of radius $R = 10^{-9}$ 
    and maximal step size $\tau_{\max}= \frac{10}{M}$.
    The results are depicted in 
		\begin{itemize}
		\item[-] Figure~\ref{fig:num_grad_flow_2d_delta_delta} for $d=2$, $r=1$, 
		\item[-] Figure~\ref{fig:num_grad_flow_3d_delta_delta} for $d=3$, $r=1$, and 
		\item[-] Figure~\ref{fig:num_grad_flow_2d_delta_delta_r_3_2} for $d=2$, $r=\frac32$.
		\end{itemize}
    To compare the computed flows with the initial steepest descent direction 
    in Theorem~\ref{thm:DKdir}, 
    we illustrate the support of 
		\begin{equation}    \label{eq:approx_one_particle_disc_flow_geodesic}
         \gamma_{\zb v}(t) 
         =
         \big((t-1) \, e_1 - \mathcal E_K(\eta^*) \, t \Id \big)_{\#} \eta^*
    \end{equation}
    with $\eta^*$ from Corollary~\ref{thm:HD-spec},
    where $\zb v \coloneqq \HD_{-}\F_{\delta_{e_1}}(\delta_{-e_1})$.
    We observe a good accordance with the numerical Euler forward scheme
    indicating that
    the particle explosion in discrepancy flows behaves
    similarly as for interaction energy flows.\hfill $\Box$
\end{example}

For smoother kernels than the considered Riesz kernels,
the repulsion effect leading to particle explosions does not occur.
For instance,
Arbel et al.\ \cite{ArKoSaGr19} consider Lipschitz-continuously differentiable kernels
and show the existence of a unique Wasserstein gradient flow 
for the corresponding discrepancy.
Moreover, 
the forward Euler scheme given in \cite[§~2.2]{ArKoSaGr19} 
converges locally uniformly to this flow. 
In particular, 
this shows 
that the Wasserstein gradient flow $\gamma$ remains atomic 
if the initial measure $\gamma(0) = \mu_0$ is atomic.
The behavior changes completely 
for the Riesz kernel with $r=1$. 
Here the repulsion effect is directly encoded 
in the steepest descent direction given in Theorem~\ref{thm:DKdir}(ii). 
For $r \in (1,2)$,
the repulsion is weaker
such that the steepest descent flow can have the form of a particle flow,
see Proposition~\ref{ex:one_particle_disc_flow};
but our numerical experiments indicate
the existence of an infinite family of Wasserstein descent flows 
similar to the family \eqref{eq:EK_descent_flow_family} observed
for the interaction energy,
where at any time the particle may decide to explode.

\appendix

\section{Proof of Lemma~\ref{lem:chain_rule}}
\label{proof:chain_rule}
Let 
$ 
\zb v_{t,h}\in \exp_{\nu(t)}^{-1}(\nu(t+h))
$.
Then it holds
        \[             \nu(t+h)
        =\gamma\left( f(t+h) \right)
        =\gamma\left( f(t)+h\dot f(t)+r(h) \right),
        \]
        where $r(h)/h\to 0$ as $h\to0$.
        Consequently, the optimal transport plans from $\nu(t)$ to $\nu(t+h)$ and from $\gamma\left( f(t) \right)$ to $\gamma\left(f(t)+h\dot f(t)+r(h)\right)$ coincide such that
        \[
          \zb v_{t,h} = \zb{\tilde v}_{f(t), h\dot f(t) + r(h)}
          \in \exp_{\gamma(f(t))}^{-1}\left(\gamma\left(f(t)+ h\dot f(t) + r(h) \right)\right).
        \]
        We consider the case $\dot f(t)>0$. Since $r(h)/h\to0$ as $h \to 0+$, we have that $h\dot f(t)+r(h)>0$ for $h>0$ small enough.
        Thus, it holds by homogenity of  $W_{\nu(t)}$, see Lemma~\ref{lem:1}, that
        \begin{equation}
            \begin{aligned}
            0
            &\leq W_{\nu(t)}\Big(\dot f(t)\cdot \dot{\gamma}(f(t)),\tfrac1h\cdot \zb v_{t,h}\Big)\\
            &= 
            \frac{h\dot f(t)+r(h)}{h} W_{\gamma(f(t))}
            \Big(\tfrac{h\dot f(t)}{h\dot f(t)+r(h)}\cdot \dot{\gamma}\left( f(t) \right),\tfrac1{h\dot f(t)+r(h)}\cdot \zb v_{t,h}\Big)\\
            &\leq 
            \underbrace{\frac{h\dot f(t)+r(h)}{h}}_{\to 1/ \dot f(t)}\Big[\underbrace{W_{\gamma(f(t))}\Big(\tfrac{h\dot f(t)}{h\dot f(t)+r(h)}\cdot \dot{\gamma}(f(t)),\dot\gamma(f(t))\Big)}_{\text{(I)}}\\
            &\qquad\qquad\qquad\;\;+\underbrace{W_{\gamma(f(t))}\Big(\dot{\gamma} \left(f(t) \right),\tfrac1{h\dot f(t)+r(h)}\cdot \zb{\tilde v}_{f(t),h\dot f(t)+r(h)}\Big)}_{\text{(II)}}\Big].\label{eq_conv}
            \end{aligned}
        \end{equation}
        The term (II) converges to zero, as $\dot{\gamma}(f(t))$ is a tangent vector of $\gamma$ at $f(t)$ and since $h\to0+$ implies $h\dot f(t)+r(h)\to0+$.
        Further, the term (I) can be computed as
        \begin{align}
            \text{(I)}
           &=
            \|\tfrac{h\dot f(t)}{h\dot f(t)+r(h)}\cdot \dot{\gamma} \left(f(t) \right)\|^2_{\gamma ( f(t))}
            +
            \|\dot{\gamma} \left(f(t) \right)\|^2_{\gamma(f(t))}
            -2
            \Big\langle\tfrac{h\dot f(t)}{h\dot f(t)+r(h)}\cdot \dot{\gamma}\left(f(t) \right),\dot{\gamma}(f(t))\Big\rangle_{\gamma(f(t))}\\
        &=
        \left( 1-\tfrac{h\dot f(t)}{h\dot f(t)+r(h) } \right)^2
        \|\dot{\gamma}\left(f(t)\right)\|^2_{\gamma(f(t))}
        \quad \to 0 \quad \text{ as } h\to 0+.
        \end{align}
        Consequently, $W_{\gamma(f(t))}(\dot f(t)\cdot \dot{\gamma}(f(t)),\tfrac1h\cdot \zb \nu_{t,h})$ converges to zero.

        Finally, we consider the case that $\dot f(t)=0$.
        Then we have $\dot f(t)\cdot \dot{\zb\gamma}\left(f(t) \right)=\gamma\left(f(t)\right)\otimes \delta_0$. 
        Thus it holds
                   \begin{align}
                W_{\nu(t)} \left(\dot f(t)\cdot \dot{\zb\gamma}\left(f(t)\right),\tfrac1h\cdot \zb \nu_{t,h} \right)
                & =W_{\nu(t)}\left(\gamma\left(f(t)\right)\otimes \delta_0,\tfrac1h\cdot \zb \nu_{t,h} \right) \\
                & =\|\tfrac1h\cdot \zb{\tilde \nu}_{f(t),r(h)}\|_{\gamma(f(t))}                 
            \end{align}
       	which is zero if $r(h)=0$ for $h>0$ small enough. Otherwise, we have that $r(h)>0$ for $h>0$. 
				Then we obtain, that the above expression is equal to
$$
\tfrac{r(h)}{h}\|\tfrac1{r(h)}\cdot \gamma_{f(t),r(h)}\|_{\gamma(f(t))}.
$$
Now, the first factor converges to zero and the second factor converges to some number $C>0$ since it holds
$\tfrac1{r(h)}\cdot \gamma_{f(t),r(h)}\to \dot\gamma (f(t))$ with respect to $W_{\gamma\left(f(t) \right)}$.
Hence the whole expression converges to zero and we are done.
\hfill$\Box$

\section{Proofs from Section~\ref{sec:frech-diff}}\label{app:proofs_sec_WSDF}
In order to prove the results from Section~\ref{sec:frech-diff}, 
we require the notion of a scalar product, metric velocity 
as well as some further properties of the metric $W_\mu$ in $\zb V(\mu)$.
We give these definitions and properties in \ref{subsec:def_scalar_metric}. 
Afterwards, we prove Proposition~\ref{ggd_case} in \ref{proof:ggd_case}.

\subsection{Scalar product, Metric velocity and Properties of $W_\mu$}\label{subsec:def_scalar_metric}
Besides the metric,
we may define the  scalar product of two velocity plans $\zb v,\zb w\in\zb V(\mu)$ by
\begin{equation}
  \label{inner_mu}
    \langle \zb v, \zb w \rangle_{\mu} 
    \coloneqq 
    \max_{\zb \alpha \in \Gamma_{\mu}(\zb v, \zb w)}
    \langle \zb v, \zb w \rangle_{\zb \alpha}
    \;\;\text{with}\;\;
    \langle \zb v, \zb w \rangle_{\zb \alpha} 
    \coloneqq 
    \int_{\R^d \times \R^d \times \R^d}
    x_2^\T x_3  
    \, \d \zb \alpha(x_1, x_2, x_3)
\end{equation}
and the \emph{metric velocity} as
\begin{equation}
    \|\zb v\|_{\mu}^2
    \coloneqq  
    \langle \zb v,\zb v \rangle_{\mu} 
    = 
    \int_{\R^d \times \R^d} \|x_2\|_2^2 \, \d \zb v(x_1, x_2).
\end{equation}
In particular, we have for $\zb v \in \exp_{\mu}^{-1}(\nu)$ that
    {\small
		\begin{equation}     \label{eq:exp_plan_norm}
        W_2^2(\mu, \nu)
        = \int_{\R^d \times \R^d} \|x - y\|_2^2 \d(\pi_1, \pi_1 + \pi_2)_{\#} \zb v(x, y)
        = \int_{\R^d \times \R^d} \|y\|_2^2 \, \d \zb v(x, y) = \|\zb v\|_{\mu}^2.
    \end{equation}
		}%
The next lemma summarizes further properties. Some of them are proven in \cite{Gi04}, for the others we provide a proof.

		\begin{lemma} \label{lem:1}
		Let $\mu \in \P_2(\R^d)$.
		For all $\zb v, \zb w \in \zb V(\mu)$, the following relations hold true:
		\begin{enumerate}[\upshape(i)]
    \item Homogeneity:
		\begin{align}
				\langle c\cdot \zb v,\zb w\rangle_\mu &=c\langle\zb v,\zb w\rangle_\mu=\langle \zb v,c\cdot \zb w\rangle_\mu, \quad c \ge 0,\\
				\| c \cdot \zb v\|_{\mu} &= |c| \|\zb v\|_{\mu}, \quad c \in \R, \\
              W_{\mu}(c\cdot\zb v,c\cdot\zb w)& = |c| \, W_\mu(\zb v,\zb w), \quad c\in\R.
	\end{align}
	\item Properties of $W_\mu$:
	\begin{align}
          W_{\mu}^2(\zb v, \zb w)
          &=  \|\zb v\|_{\mu}^2 + \|\zb w\|_{\mu}^2 - 2 \langle \zb v, \zb w \rangle_{\mu},
          \\
          W_\mu(\zb v, \zb w)
          &\geq W_2(\gamma_{\zb v}(1), \gamma_{\zb w}(1)),\\
            W_\mu^2(\zb v,\zb 0_\mu)&=\|\zb v\|_\mu^2=\int_{\R^d}\|x_2\|_2^2\d \zb v(x).
        \label{eq:W2-Wmu}
    \end{align}
\item  Cauchy--Schwarz inequality:
    \begin{equation}
      |\langle \zb {v},\zb{w}\rangle_\mu|\leq\|\zb{v}\|_\mu\|\zb{w}\|_\mu,
    \end{equation}
  where it holds $\langle \zb{v},\zb{w}\rangle_\mu=\|\zb{v}\|_\mu\|\zb{w}\|_\mu$ if and only if there exists some $c\geq0$ such that $\zb v=c\cdot\zb w$.
\end{enumerate}
\end{lemma}

 \begin{proof}
		i) The homogeneity proof was given in  \cite[Prop~4.17, 4.27]{Gi04}.
		\\
	ii) The first result was shown in \cite[Prop~4.2]{Gi04} and the second one in \cite[(4.23)]{Gi04}. 
The third item follows by the definition of $\|\zb v\|_\mu^2$ and the first item as it holds by i) that 
$
\|\zb 0_\mu\|_\mu=0\|\zb 0_\mu\|_\mu=0=0\langle \zb v,\zb 0_\mu\rangle_\mu=\langle \zb v,\zb 0_\mu\rangle_\mu
$.
\\
iii) It  holds 
{\small
  \begin{align}
    &\langle\zb{v},\zb{w}\rangle_\mu
    =
      \sup_{\boldsymbol\alpha\in\Gamma_\mu(\zb{v},\zb{w})}\langle\zb{v},\zb{w}\rangle_{\boldsymbol\alpha}
      =\sup_{\boldsymbol\alpha\in\Gamma_\mu(\zb{v},\zb{w})}\int_{\R^d\times\R^d\times\R^d}y^\tT z 
      \d \boldsymbol\alpha(x,y,z)\\
    &
      \leq \sup_{\boldsymbol\alpha\in\Gamma_\mu(\zb{v},\zb{w})}
      \int_{\R^d\times\R^d\times\R^d}\|y\|_2 \|z\|_2 \d\boldsymbol\alpha(x,y,z)\\
    &\leq 
      \sup_{\boldsymbol\alpha\in\Gamma_\mu(\zb{v},\zb{w})}\Big(\int_{\R^d\times\R^d\times\R^d}\|y\|_2^2\d\boldsymbol\alpha(x,y,z)\Big)^{1/2}
    \Big(\int_{\R^d\times\R^d\times\R^d}\|z\|_2^2\d\boldsymbol\alpha(x,y,z)\Big)^{1/2}\\
    &=
      \sup_{\boldsymbol\alpha\in\Gamma_\mu(\zb{v},\zb{w})}\Big(\int_{\R^d\times\R^d}\|y\|_2^2\d\zb{v}(x,y)\Big)^{1/2} \Big(\int_{\R^d\times\R^d}\|z\|_2^2\d\zb{w}(x,z)\Big)^{1/2} 
    =\|\zb{v}\|_\mu\|\zb{w}\|_\mu,
	  \end{align}
		}%
where the first inequality is Cauchy--Schwarz' inequality in $\R^d$ and 
the second inequality is Cauchy--Schwarz's inequality on the functions $(x,y,z)\mapsto\|y\|_2$ and $(x,y,z)\mapsto \|z\|_2$ 
in $L_2(\zb \alpha)$.
Consequently, it holds equality if and only if for $\zb\alpha$-almost every $(x,y,z)$ there exists some $c\geq0$ such that $y=cz$ and if there exists some $c\geq0$ such that $\|y\|_2=c\|z\|_2$ $\zb \alpha$-almost everywhere.
That is, we have equality if and only if there exist some $c\geq0$ such that $y=cz$ $\zb \alpha$-almost everywhere, which is equivalent to $\zb v=c\cdot \zb w$.
Finally, it follows 
  \begin{equation}
-\|\zb{v}\|_\mu\|\zb{w}\|_\mu=-\|-\zb{v}\|_\mu\|\zb{w}\|_\mu\leq -\langle -\zb{v},\zb{w}\rangle_\mu
\leq\langle \zb{v},\zb{w}\rangle_\mu\leq\|\zb{v}\|_\mu\|\zb{w}\|_\mu.
\end{equation}
 \end{proof}

\subsection{Proof of Proposition~\ref{prop:dini-hard}}
\label{proof:dini-hard}
Let $(\zb v_n)_{n \in \mathbb N}$ be a sequence in $\zb V(\mu)$ with $W_\mu(\zb v_n,\zb v)\to0$ and let $(t_n)_{n \in \mathbb N}$ with
$t_n>0$ such that $t_n\to0$.
Then, we find $n_0\in\N$ such that
    \[
       \gamma_{\zb v}(t_n),\gamma_{\zb v_n}(t_n) \in B_r(\mu), \qquad n \ge n_0. 
    \]
  Using the local Lipschitz continuity of $\F$, formula \eqref{eq:scaling}, the second item of Lemma~\ref{lem:1} ii) and the third item of Lemma~\ref{lem:1} i), we infer that
  \begin{align}
    \lim_{n\to\infty}
    \biggl|
    \frac
    {\F(\gamma_{\vb}(t_n)) - \F(\gamma_{\vb_n}(t_n))}
    {t_n}
    \biggr|
    & \le
      \lim_{n\to\infty}
      \frac
      {L \, W_2(\gamma_{\vb}(t_n), \gamma_{\vb_n}(t_n))}
      {t_n}
    \\
    &=
      \lim_{n\to\infty}
      \frac
      {L \, W_2(\gamma_{t_n \vb}(1), \gamma_{t_n \vb_n}(1))}
      {t_n}
      \le
      \lim_{n\to\infty}
      \frac
      {L \, W_\mu(t_n \vb, t_n \vb_n)}
      {t_n}
    \\
    &=
      \lim_{n \to \infty}
      \frac
      {L \, t_n \, W_\mu(\vb, \vb_n)}
      {t_n}
      =
      \lim_{n \to \infty}
      L \, W_\mu(\vb, \vb_n)
      = 
      0.\label{star}
  \end{align}
  On the other hand, by the definition of $\liminf$ and $\limsup$ there exist sequences $t^{\pm}_n \to 0+$ in $\R_{\geq0}$ and $\zb v_n^{\pm}$ in $\zb V(\mu)$ such that $\gamma_{\zb v_n^{\pm}}|_{[0,t_n^{\pm}]}$ 
  are geodesics, $W_\mu(\zb v_n^{\pm},\vb)\to 0$ and
  \begin{equation}
    \begin{aligned}
      \HD_{\zb v}^{-} \F(\mu)
  \coloneqq
  \smashoperator{\liminf_{\substack{\zb w \to \zb v, \, t \to 0+,\\ 
  \gamma_{\zb w} |_{[0,t]} \, \text{is geodesic}}}} 
  \frac{\F(\gamma_{\zb w}(t)) - \F(\mu)}{t}
      =
      \lim_{n \to \infty}
      \frac
      {\F(\gamma_{\zb v_n^{-}}(t_n^{-})) - \F(\mu)}
      {t_n^{-}}, \\
      \HD_{\zb v}^{+} \F(\mu) 
  \coloneqq
  \smashoperator{\limsup_{\substack{\zb w \to \zb v, \, t \to 0+,\\ 
  \gamma_{\zb w} |_{[0,t]} \, \text{is geodesic}}}}
  \frac{\F(\gamma_{\zb w}(t)) - \F(\mu)}{t}
      =
      \lim_{n \to \infty}
      \frac
      {\F(\gamma_{\zb v_n^{+}}(t_n^{+})) - \F(\mu)}
      {t_n^{+}}.
    \end{aligned} 
  \end{equation}
Then, it holds by \eqref{star}  that
  \begin{equation}
    |\D_{\zb v}\F(\mu) - \HD_{\zb v}^{\pm} \F(\mu)| = \lim_{n \to \infty} \biggl|
     \frac
    {\F(\gamma_{\zb v}(t_n^{\pm})) - \F(\gamma_{\zb v_n^{\pm}}(t_n^{\pm}))}
    {t_n^{\pm}}
    \biggr| = 0.\\
  \end{equation}
  Since both the lower and the upper Hadamard derivative coincide with the Dini derivative we arrive at the assertion.\hfill$\Box$

\subsection{Proof of Proposition~\ref{ggd_case}}
\label{proof:ggd_case}

Proposition~\ref{ggd_case} is a special case 
of Theorem~\ref {ggd_case_general} at the end of this subsection.
In order to relate Wasserstein steepest descent flows to Wasserstein gradient flows, we need more technicalities, in particular the notation of subdifferentials for velocity plans. 
The \emph{extended Fr\'echet subdifferential} of a function $\F\colon\P_2(\R^d)\to(-\infty,\infty]$ at $\mu$ is defined by
  \begin{equation}
    \zb\partial \F(\mu)\coloneqq \{\zb v\in\zb V(\mu):\F(\nu)
        \ge \F(\mu) - \sup_{\zb{v} \in \exp^{-1}_\mu(\nu)}
        \langle (-1) \cdot \zb{h},\zb{v} \rangle_\mu
        + o(W_2(\mu,\nu))\},
\end{equation}
cf.~\cite[Def~10.3.1]{BookAmGiSa05}.
In particular, it is shown in \cite[Thm~10.3.11]{BookAmGiSa05} that for functions $\F\colon\P_2(\R^d)\to(-\infty,\infty]$, which are $\lambda$-convex functions along generalized geodesics, and a measure $\mu\in\P_2(\R^d)$ with $\zb \partial\F(\mu)\neq\emptyset$ there exists an unique element of the subdifferential with minimal norm, i.e.,
$
\argmin_{\zb v\in\zb\partial\F}\|\zb v\|_\mu
$
contains exactly one element.
Moreover, the \emph{local slope} $|\partial \F|\colon\P_2(\R^d)\to [0,\infty]$ of $\F$ is defined by
  \begin{equation}
    |\partial\F| \left(\mu \right) \coloneqq \limsup_{\nu \to \mu} \frac{\left(\F(\mu) - \F(\nu) \right)^+}{W_2(\mu,\nu)},
\end{equation}
where $(t)^+ \coloneqq \max\{t,0\}$.

Using these notations, the following theorem from \cite[Thm~11.2.1]{BookAmGiSa05}  characterizes the  tangent vectors of Wasserstein gradient flows for all $t \in [0,+\infty)$. 
Note that the original theorem is formulated for $W_2$ instead of $W_{\gamma(t)}$ in \eqref{eq1}, but the proofs in those book provide indeed the relation below. 

\begin{theorem} \label{thm:nice_case}
Let $\F\colon  \P_2(\R^d) \to (-\infty,+\infty]$ be proper, lsc, 
coercive and $\lambda$-convex along generalized geodesics.
Further, denote by $\gamma\colon (0,\infty)\to\P_2(\R^d)$ the unique Wasserstein gradient flow from Theorem~\ref{thm:existence_gflows_ggd}.
Then, for every $t,h >0$
and $\zb v_{t,h} \in \exp^{-1}_{\gamma(t)}\left(\gamma(t+h) \right)$, 
the right limit
\begin{equation} \label{eq1}
\dot\gamma(t)=\zb v_t \coloneqq \lim_{h \to 0+} -\frac{1}{h} \cdot \zb v_{t,h}, \text{ i.e.,  } \lim_{h \to 0+} W_{\gamma(t)}(\zb v_t,  -\tfrac1h \cdot \zb v_{t, h}) = 0
\end{equation}
exists and satisfies for all $t>0$ the relations
\begin{equation}\label{eq2}
\zb v_t = \argmin_{\zb v \in \zb \partial\F(\gamma(t))} \| \zb v \|_{\gamma(t)} 
\end{equation}
and
\begin{equation}\label{eq3}
\lim_{h \to 0+} \frac{\F(\gamma(t+h)) - \F( \gamma(t) )}{h} 
= - \| \zb v_t \|_{\gamma(t)}^2
= - |\partial\F|^2 \left(\gamma(t) \right)
\end{equation}
Further, \eqref{eq1}, \eqref{eq2} and \eqref{eq3} hold true at $t=0$ 
if and only if 
$\zb \partial \F(\mu_0) \not = \emptyset$.
\end{theorem} 

Using this theorem, we can show that in some cases Wasserstein gradient flows are Wasserstein steepest descent flows.

\begin{lemma}\label{lem:GradientFlow_is_steepest_descent}
Let $\F\colon  \P_2(\R^d) \to (-\infty,+\infty]$ be proper, lsc, 
coercive and $\lambda$-convex along generalized geodesics.
Then, for $\mu\in\mathcal P_2(\R^d)$ with $\zb\partial\F(\mu)\neq \emptyset$ and 
\begin{equation}\label{holds}
    \zb v\coloneqq \argmin_{\zb w\in\zb\partial\F(\mu)}\|\zb w\|_\mu,
\end{equation}
it holds
\begin{equation*}
  \zb h\coloneqq(-1)\cdot\zb v\in\HD_-\F(\mu)
  \quad\text{and}\quad
  \HD^-_{\zb h/\|\zb h\|_\mu} \F(\mu)=-|\partial \F|(\mu).
\end{equation*}
In particular,  for all $\mu_0\in \mathcal P_2(\R^d)$ with $\zb\partial\F(\mu_0)\neq \emptyset$, the Wasserstein gradient flow of $\F$ starting at $\mu_0$ is a Wasserstein steepest descent flow of $\F$. 
\end{lemma}

\begin{proof}
Theorem \ref{thm:nice_case} implies the existence of a unique Wasserstein gradient flow $\gamma\colon (0,\infty)\to\mathcal P_2(\R^d)$ with $\gamma(0)\coloneqq\gamma(0+)=\mu$. 
Since $\zb\partial\F(\mu)\neq \emptyset$, we know that
$\zb v \in \V(\mu)$ in \eqref{holds}
is the velocity field $\zb v_t$ with \eqref{eq1}, \eqref{eq2} and \eqref{eq3} at $t=0$.

We consider two cases.
If  $\zb v = \zb 0_{\gamma(0)}$, then the assertion is straightforward.
If $\zb v \not = \zb 0_{\gamma(0)}$, we conclude 
by \eqref{eq3} and the definition of $\HD^-_{\zb v} \F$ that
\begin{align}
   - |\partial \F|^2(\mu)
   &=
   \lim_{h \to 0+}
  \frac
  {\F(\gamma(h)) - \F(\gamma(0))} h
  =
  \lim_{h \to 0+}
  \frac
  {\F(\gamma_{-h^{-1}\cdot\vb_{0,h}}(h)) - \F(\gamma(0))} h\\
  &\ge
  \HD^-_{(-1)\cdot\zb v} \F(\gamma(0))=\HD^-_{\zb h} \F(\gamma(0)),
\end{align}
and further by dividing by $|\partial\F|(\gamma(0))=\|\zb v\|_{\gamma(0)}=\|\zb h\|_{\gamma(0)}$ that
  \begin{equation}
    \frac{1}{\| \zb h \|_{\gamma(0)}} \HD^-_{\zb h} \F(\gamma(0))
    \le 
    - |\partial \F|(\gamma(0)).
  \end{equation}
 On the other hand, we have by \eqref{eq:exp_plan_norm} that
 $W_2(\mu, \gamma_{\tilde{\zb w}}(t)) = \|t \cdot \tilde{\zb w} \|_\mu = t \| \tilde {\zb w} \|_\mu$, 
 $\tilde{\zb w} \in \T_\mu\left(\P_2(\R^d)\right)$, 
 such that for every 
 $\zb w \in \zb \T_\mu\left(\P_2(\R^d)\right)$, $\zb w \not = \zb 0_\mu$,
 \begin{align}\label{lower}
    \frac
    {(\HD^-_{\zb w} \F(\mu))^-}
    {\| \zb w \|_\mu}
    &\le
    \biggl(
    \liminf_{\substack{\tilde{\zb w} \to \zb w, t \to 0+ \\ \gamma_{\tilde w}|_{[0,t]}\text{ is geodesic}}}
    \frac
    {\F(\gamma_{\tilde{\zb w}}(t)) - \F(\mu)}
    {t \, \| \tilde {\zb w}\|_\mu}
    \biggr)^-\\
    &=
    \limsup_{\substack{\tilde{\zb w} \to \zb w, t \to 0+ \\ \gamma_{\tilde w}|_{[0,t]}\text{ is geodesic}}}
    \frac{(\F(\mu) - \F(\gamma_{\tilde {\zb w}}(t)))^+}{W_2(\mu,\gamma_{\tilde {\zb w}}(t))}
    \le
    | \partial \F | (\mu).
  \end{align}
 Combining both inequalities, we get
   \begin{equation}
 \frac{1}{\| \zb h \|_{\gamma(0)}} 
 \HD^-_{\zb h} \F(\gamma(0)) 
 = 
 \HD^-_{ \frac{1}{\| \zb h \|_{\gamma(0)}} \cdot \zb h} \F(\gamma(0)) 
 =
- | \partial \F | (\gamma(0)).
\end{equation}
Let $\zb w  \in \argmin\limits_{\substack{\tilde{\zb w} \in \zb \T_{\gamma(0)}\P_2(\R^d) \\ \|\tilde{\zb w}\|_{\gamma(0)} = 1}} \HD_{\tilde{\zb w}}^- \F(\mu) $.
Then we obtain by \eqref{lower} that
  \begin{equation}
    \HD^-_{\zb w} \F(\gamma(0)) \ge  - | \partial \F | (\gamma(0)),
  \end{equation}
so that 
 $\frac{1}{\| \zb h \|_{\gamma(0)}} \cdot \zb h \in  \argmin\limits_{\substack{\tilde{\zb w} \in \zb \T_{\gamma(0)}\P_2(\R^d) \\ \|\tilde{\zb w}\|_{\gamma(0)} = 1}} \HD_{\tilde{\zb w}}^- \F(\mu)$
 and
 $\zb h \in \HD_-\F(\gamma(0))=\HD_-\F(\mu)$.
\end{proof}

To show the reverse direction, namely that every Wasserstein steepest descent flow is a Wasserstein gradient flow for
special functions $\F$,
we need some additional assumptions.

We say that $\F\colon\P_2(\R^d)\to(-\infty,\infty]$ is \emph{continuous along geodesics} if $\F\circ\gamma\colon[0,\epsilon]\to(-\infty,\infty]$ is continuous for \emph{all} geodesics $\gamma\colon[0,\epsilon]\to \P_2(\R^d)$ with $\F \circ \gamma(0),\F \circ \gamma(\epsilon)<\infty$.
The following lemma states that $\zb h$ from the previous lemma is the only element in $\HD_-\F(\mu)$
if $\F$ is additionally continuous along geodesics.
This will be the basis of the proof that under mild assumptions Wasserstein steepest descent flows are Wasserstein gradient flows.

\begin{lemma}\label{lem:unique_steepest_descent}
Let $\F\colon  \P_2(\R^d) \to (-\infty,+\infty]$ be proper, lsc, 
coercive, $\lambda$-convex along generalized geodesics and continuous along geodesics.
Then,  it holds for any $\mu\in\mathcal P_2(\R^d)$ with $\zb\partial\F(\mu)\neq \emptyset$ that $\HD_-\F(\mu)=\{\zb h\}$, where $\zb h$ is defined as in Lemma~\ref{lem:GradientFlow_is_steepest_descent}. 
\end{lemma}

\begin{proof}
  Consider some $\zb g \in \HD_{-}\F(\mu)$. Since $\zb h\in\HD_-\F(\mu)$, we have
    \begin{equation}
      \HD_{\zb g}^{-}\F(\mu)=\HD_{\zb h}^-\F(\mu)=- \|\zb h\|_{\mu} |\partial \F|(\mu)=- \|\zb g\|_{\mu} |\partial \F|(\mu).\end{equation}
  Further, by definition of $\HD_{\zb g}^{-}$, there exist $\tilde t_n\to0$ and $\zb g_n\in \zb G(\mu)$ with $ \zb g_n \to \zb g$ in $W_\mu$ such that $\tilde  t_n\cdot\zb g_n\in\exp_\mu^{-1}(\gamma_{\zb g_n}( \tilde t_n))$
  and
  \begin{equation}
    \lim_{n\to \infty} \frac{\F(\gamma_{\zb g_n}(\tilde t_n)) - \F(\mu)}{\tilde t_n}= \HD_{\zb g}^{-}\F(\mu) = - \|\zb g\|_{\mu} |\partial \F|(\mu).
  \end{equation}
  Since the limit is finite, we assume wlog that $\F(\gamma_{\zb g_n}(\tilde t_n))<\infty$.
  Thus, by continuity of $\F$ along geodesics, the functions
  \begin{equation}
    \phi_n \colon  [0,1] \to \R, \qquad \phi_n(s) \coloneqq \F (\gamma_{\zb g_n}(s \tilde  t_n)), \qquad s \in [0,1],
  \end{equation}
  are continuous. 
  Hence we can find a sequence $(s_n)_{n \in \N}$ with 
  $s_n \to 1_-$ as $n\to \infty$ such that 
  \begin{equation}
    |\phi_n(s_n) - \phi_n(1)| = |\F(\gamma_{\zb g_n}(s_n \tilde t_n)) - \F(\gamma_{\zb g_n}( \tilde t_n))| \le \tilde t_n^2
  \end{equation}
  with $\exp_{\mu}^{-1}(\gamma_{\zb g_n}(s_n \tilde t_n)) = \{ s_n \tilde t_n \cdot \zb v_n \}$, cf. \cite[Lem~7.2.1]{BookAmGiSa05}.
  Now, replacing the sequence $\tilde t_n$ by $t_n\coloneqq s_n \tilde t_n$ does not alter the limit, i.e.,
    \begin{align*}
      &\lim_{n\to \infty} \frac{\F(\gamma_{\zb g_n}(t_n)) -\F(\mu)}{t_n} 
      \\
      &=  \lim_{n\to \infty} \frac{\F(\gamma_{\zb g_n}(s_n\tilde t_n)) -\F(\mu)}{s_n \tilde t_n}
      - \lim_{n\to \infty} \frac{\F(\gamma_{\zb g_n}(s_n\tilde t_n))- \F(\gamma_{\zb g_n}(\tilde t_n))}{s_n \tilde t_n}\\  
      &= \lim_{n\to \infty} \frac{\F(\gamma_{\zb g_n}(\tilde t_n))-\F(\mu)}{\tilde t_n}=\HD_{\zb g}^-\F(\mu).
    \end{align*}
  Since $(-1)\cdot \zb h \in \zb \partial \F(\mu)$ and $\exp_{\mu}^{-1}(\gamma_{\zb g_n}(t_n)) = \{t_n \cdot \zb g_n\}$ it holds with \eqref{eq:exp_plan_norm} that
  \begin{equation}\F(\gamma_{\zb g_n}(t_n)) - \F(\mu)  \ge  - \langle \zb h, t_n\cdot \zb g_n \rangle_{\mu} + o(t_n\|\zb g_n\|_\mu) =   - t_n \langle \zb h, \zb g_n \rangle_{\mu} + o(t_n \|\zb g_n\|_{\mu}).
  \end{equation}
  As $\zb g_n\to\zb g$ in $W_\mu$ implies $\|\zb g_n\|_\mu\to\|\zb g\|_\mu$, we obtain by dividing both sides by $t_n$ and letting $n \to \infty$ that
  \begin{equation}
    - \|\zb g\|_{\mu} |\partial \F|(\mu) = \lim_{n\to \infty} \frac{\F(\gamma_{\zb g_n}(t_n)) - \F(\mu)}{t_n} \ge - \langle \zb h, \zb g \rangle_{\mu} \ge - \| \zb h \|_{\mu} \|\zb g\|_{\mu},
  \end{equation}
  where the second implication is the Cauchy--Schwarz inequality from Lemma \ref{lem:1} (iii).
  Since $\|\zb g\|_{\mu} = \|\zb h\|_{\mu} = |\partial \F|(\mu)$ we have equality. 
  By the equality condition of the Cauchy--Schwarz relation, this yields $\zb h = \zb  g$ such that
  $\HD_-\F(\mu)=\{\zb h\}$.
\end{proof}

Now, we can show that under certain assumptions Wasserstein steepest descent flows are Wasserstein gradient flows.

\begin{lemma}\label{lem:steepest_descent_is_GradientFlow}
Let $\F\colon  \P_2(\R^d) \to (-\infty,+\infty]$ be proper, lsc, coercive, $\lambda$-convex along generalized geodesics and continuous along geodesics and let $\mu_0\in \mathcal P_2(\R^d)$ such that $\zb\partial-\F(\mu_0)\neq \emptyset$.
Further assume that $\zb \partial F(\mu)\neq\emptyset$ for any $\mu\in\P_2(\mu)$ with $\HD_-\F(\mu)\neq\emptyset$.
Then, there exists a unique Wasserstein steepest descent flow of $\F$ starting at $\mu_0$, which coincides with the Wasserstein gradient flow of $\F$ starting at $\mu_0$.
\end{lemma}

\begin{proof}
By Theorem~\ref{thm:existence_gflows_ggd}, there exists a unique Wasserstein gradient flow starting at $\mu_0$, which is by Lemma~\ref{lem:GradientFlow_is_steepest_descent} a steepest descent flow. 
Thus, it suffices to show that any Wasserstein steepest descent flow $\gamma\colon [0,\infty) \to\mathcal P_2(\R^d)$ is a Wasserstein gradient flow.
Since $\dot\gamma(t)\in\HD_-(\gamma(t))$, we have that $\HD_-(\gamma(t))$ is non-empty, which implies
by assumption that the subdifferential $\zb \partial \F(\gamma(t))$ is non-empty.
By Lemma \ref{lem:GradientFlow_is_steepest_descent} and \ref{lem:unique_steepest_descent}, we obtain $\dot\gamma(t)\in\HD_-\F(\gamma(t))=\{(-1)\cdot\zb v\}$, 
where $\zb v \coloneqq \argmin_{\zb w\in\zb\partial\F(\gamma(t))}\|\zb w\|_{\gamma(t)}$ 
which implies that $(-1) \cdot \dot\gamma(t)\in \zb\partial\F(\gamma(t))$.
Since $\gamma$ is by definition absolutely continuous, 
there exists a Borel velocity field $v_t\colon\R^d\to\R^d$, $t\in I$ with
\begin{align}\label{eq:cont_equation_fulfilled}
\partial_t\gamma(t)+\nabla_x\cdot(\gamma(t)v_t)&=0\quad\text{in } I\times\R^d.
\end{align}
Moreover, \cite[Prop~8.4.6]{BookAmGiSa05} implies that $(\Id,v_t)_\#\gamma(t)=\dot\gamma(t)$ for almost-every $t\in I$ such that 
\begin{equation}
(\Id,-v_t)_\#\gamma(t)
=(-1) \cdot (\Id,v_t)_\#\gamma(t)
=(-1) \cdot\dot\gamma(t)\in\zb \partial \F(\gamma(t)).
\end{equation}
Finally, we know by \cite[Rem~10.3.3]{BookAmGiSa05}  that $(\Id,v)_\#\mu\in\zb\partial\F(\mu)$ if and only if $v\in\partial\F(\mu)$, so that 
$-v_t\in\partial\F(\gamma(t)$, i.e., $v_t\in-\partial\F(\gamma(t))$.
Together with \eqref{eq:cont_equation_fulfilled}, we can conclude that $\gamma$ is the unique Wasserstein gradient flow with respect to $\F$.
\end{proof}

\begin{remark}
The assumption that $\zb \partial F(\mu)\neq\emptyset$ for any $\mu\in\P_2(\mu)$ with $\HD_-\F(\mu)\neq\emptyset$ 
is automatically fulfilled if the slope $|\partial \F|(\mu)$ is finite for every $\mu\in\P_2(\R^d)$ as \cite[Thm~10.3.10]{BookAmGiSa05} 
states that the subdifferential at $\mu$ is non-empty for all $\mu\in\P_2(\R^d)$ 
with $|\partial \F|(\mu)<\infty$. 
This includes in particular locally Lipschitz continuous functions $\F$ since local Lipschitz continuity implies by definition that $|\partial \F|(\mu)$ is finite  for all $\mu\in\P_2(\R^d)$.
\end{remark}

Summarizing Lemma~\ref{lem:GradientFlow_is_steepest_descent}, \ref{lem:unique_steepest_descent} and \ref{lem:steepest_descent_is_GradientFlow} we obtain the following theorem.

\begin{theorem}\label{ggd_case_general}
Let $\F\colon  \P_2(\R^d) \to (-\infty,+\infty]$ be proper, lsc, coercive and $\lambda$-convex along generalized geodesics and let $\mu_0\in \mathcal P_2(\R^d)$ such that $\zb\partial\F(\mu_0)\neq \emptyset$.
Then, the unique Wasserstein gradient flow starting at $\mu_0$ is a Wasserstein steepest descent flow.
Moreover, if $\F$ is additionally continuous along geodesics and fulfills $\zb \partial F(\mu)\neq\emptyset$ for any $\mu\in\P_2(\mu)$ with $\HD_-\F(\mu)\neq\emptyset$,
then there exists a unique Wasserstein steepest descent flow starting at $\mu_0$ which coincides with the Wasserstein gradient flow.
\end{theorem}

\section{Proof of Proposition~\ref{prop:lip-int-pot}}
\label{proof:lip-int-pot}
In the following, we denote the  Lipschitz constant of $F\colon \R^d \to \R$ on $A \subset \R^d$ by
\begin{equation}
  \Lip(F, A) \coloneqq \sup\Big\{ \frac{|F(x) - F(y)|}{\|x-y\|_2}: x, y \in A, \; x \not = y \Big\}.
\end{equation}
To prove the Proposition~\ref{prop:lip-int-pot}, we need three auxiliary lemmata.

\begin{lemma}   \label{prop:VLip}
  Let $V\colon \R^d \to \R$ be locally Lipschitz continuous and $L>0$ such that
  \begin{equation}\label{eq:Vlip}
    \Lip(V, B_{r}(x)) \le L (1 + \|x\|_2 + r), \qquad x \in \R^d, \quad r>0.
  \end{equation}
    Then the functional $\V\colon \P_2(\R^d) \to \R$ given by
  \begin{equation}
    \V(\mu) \coloneqq \int_{\R^d} V(x) \d \mu(x), \qquad \mu \in \P_2(\R^d),
  \end{equation}
  is locally Lipschitz continuous.
\end{lemma}

\begin{proof}
  For $r > 0$ and $\mu \in \mathcal P_2(\R^d)$,
  let $\nu_1, \nu_2 \in B_r(\mu)$. 
  Choosing $\zb \pi \in \Gamma^{\opt}(\nu_1,\nu_2)$ 
  and $\zb v = (\pi_1, \pi_2 - \pi_1)_\# \zb \pi \in \exp_{\nu_1}^{-1}(\nu_2)$,
  we estimate first applying the triangular inequality 
  \begin{align}
         |\V(\nu_1) - \V(\nu_2)|
      & \le \int_{\R^d \times \R^d} |V(x_1) - V(x_2)| \d \zb \pi(x_1, x_2)\\
      & = \int_{\R^d \times \R^d} |V(x_1) - V(x_1+x_2)| \d \zb v(x_1, x_2)\\
      &  \le \int_{\R^d \times \R^d} \Lip\left(V, B_{\|x_2\|_2}(x_1) \right) \|x_2\|_2 \d \zb v(x_1, x_2)\\
      & \le L \int_{\R^d \times \R^d} 
      \bigl(1 + \|x_1\|_2 + \|x_2\|_2 \bigr) \|x_2\|_2 \d \zb v(x_1, x_2).
   \end{align}
   Using the Cauchy-Schwarz inequality, we get 
   \begin{equation}
   \begin{aligned}
         |\V(\nu_1) - \V(\nu_2)|
      & \le 
      L \left(\int_{\R^d \times \R^d} \|x_2\|_2^2 \d \zb v(x_1,x_2)\right)^{\frac12}
      + L \int_{\R^d \times \R^d} \|x_2\|_2^2 \d \zb v(x_1,x_2)
      \\
      &+ L \left( \int_{\R^d \times \R^d} \|x_1\|_2^2 \d \zb v(x_1, x_2) \right)^\frac12
      \left( \int_{\R^d \times \R^d} \|x_2\|_2^2 \d \zb v(x_1,x_2) \right)^\frac12\\
       & = L \left(1 + \left(\int_{\R^d}\|x_1\|_2^2 \d \nu_1(x_1)\right)^\frac12 + W_2(\nu_1, \nu_2) \right) W_2(\nu_1, \nu_2)\\
      & = L \left(1 + 2r + \left(\int_{\R^d}\|x_1\|_2^2 \d \nu_1(x_1)\right)^\frac12 \right) W_2(\nu_1, \nu_2).
    \end{aligned}
  \end{equation}
  To estimate the remaining integral, 
  let $\tilde{\zb \pi} \in \Gamma^{\opt}(\nu_1,\mu)$.
  Using the triangle inequality,  we obtain
  \begin{equation}
      \label{eq:first-mom}
      \begin{aligned}
        \int_{\R^d} \|x_1\|_2^2 \d \nu_1(x_1)
        &=
        \int_{\R^d \times \R^d}
        \|x_1\|_2^2 \d \tilde{\zb \pi}(x_1,x_2)\\
        &\le 
        2\int_{\R^d \times \R^d}
        \|x_1 - x_2\|_2^2 \d \tilde{\zb \pi}(x_1,x_2)
        +
        2\int_{\R^d \times \R^d}
        \|x_2\|_2^2 \d \tilde{\zb \pi}(x_1,x_2)
        \\
        &= 2 W_2(\mu,\nu_1) + 2 \int_{\R^d} \|x_2\|_2^2 \d \mu(x_2),
      \end{aligned}
  \end{equation}
  and consequently
  \begin{equation*}
      |\V(\nu_1) - \V(\nu_2)|
      \le
       L \left(1 + 2r + \left(2r + 2\int_{\R^d}\|x_2\|_2^2 \d \mu(x_2)\right)^\frac12 \right) W_2(\nu_1, \nu_2). \qedhere
  \end{equation*}
\end{proof}

Note, if $V\colon \R^d \to \R$ is differentiable such that there exists $L>0$ with 
\begin{equation}
  \|\nabla V(x)\|_2 \le L (1 + \|x\|_2), \qquad x \in \R^d,
\end{equation}
then for $x_1, x_2 \in B_r(x)$ with $r>0$, $x \in \R^d$, it holds
\[
\begin{aligned}
  | V(x_2) - V(x_1) | & = \left| \int_0^1 \nabla V(x_1 + t (x_2-x_1))^{\tT} (x_2-x_1) \d t \right| \\
&  \le \int_0^1 \|\nabla V(x_1 + t(x_2-x_1))\|_2 \|x_2-x_1\|_2 \d t \\
&  \le \int_0^1 L (1 + \|x_1 + t(x_2-x_1)\|_2) \|x_2-x_1\|_2 \d t \\
&  \le \int_0^1 L (1 + \|x_1\|_2 + t\|x_2-x_1\|_2) \|x_2-x_1\|_2 \d t \\
&  = L (1 + \|x_1\|_2 + \tfrac12 \|x_2 - x_1\|_2) \|x_2 - x_1\|_2
  \le  L (1 + \|x_1\|_2 + r) \|x_2 - x_1\|
\end{aligned}
\]
and the condition \eqref{eq:Vlip} is satisfied.

\begin{lemma}  \label{lem:EKLip}
  Let $K\colon \R^d \times \R^d \to \R$ be locally Lipschitz continuous and $L>0$ such that
  \begin{equation}
  \label{eq:Klip}
    \Lip(K, B_{r}(x) \times B_{s}(y)) \le L (1 + \|x\|_2 + \|y\|_2 + r + s), \qquad x, y\in \R^d, \quad r, s \ge 0.
  \end{equation}
  Then the interaction energy 
  $\mathcal E_K\colon \P_2(\R^d) \to \R$ in \eqref{eq:interaction}
  is locally Lipschitz continuous.
\end{lemma}

\begin{proof}
  For $r > 0$,
  let $\nu_1, \nu_2 \in B_r(\mu)$. 
  Choosing $\zb \pi \in \Gamma^{\opt}(\nu_1,\nu_2)$ 
  and $\zb v = (\pi_1, \pi_2 - \pi_1)_\# \zb \pi \in \exp_{\nu_1}^{-1}(\nu_2)$,
  we estimate
  \begin{equation}
    \begin{aligned}
      &|\mathcal E_K(\nu_1) - \mathcal E_K(\nu_2)|
      \\
      & \le \int_{\R^{4d}} |K(x_1, x_2)- K(y_1,y_2)| \d \zb \pi(x_1, y_1) \d \zb \pi(x_2, y_2)\\
      & \le \int_{\R^{4d}} |K(x_1, x_2)- K(x_1+y_1, x_2 + y_2)| \d \zb v(x_1, y_1) \d \zb v(x_2, y_2)\\
      & \le \int_{\R^{4d}} \Lip(K, B_{\|y_1\|_2}(x_1) \times B_{\|y_2\|_2}(x_2)) \|(y_1, y_2)\|_2 \d \zb v(x_1, y_1) \d \zb v(x_2, y_2)\\
      & \le \int_{\R^{4d}} \Lip(K, B_{\|y_1\|_2}(x_1) \times B_{\|y_2\|_2}(x_2)) (\|y_1\|_2 + \|y_2\|_2) \d \zb v(x_1, y_1) \d \zb v(x_2, y_2)\\
      & \le L \int_{\R^{4d}} (1 + \|x_1\|_2 +  \|x_2\|_2 + \|y_1\|_2 + \|y_2\|_2) (\|y_1\|_2 + \|y_2\|_2) \d \zb v(x_1, y_1) \d \zb v(x_2, y_2)\\
      & \le L\left(2 + 2 \int_{\R^d}\|x\|_2 \d \nu_1(x) + 2 \left(\int_{\R^d}\|x\|_2^2 \d \nu_1(x) \right)^\frac12 + 4 W_2(\nu_1, \nu_2) \right) W_2(\nu_1, \nu_2)\\
      & \le L\left(2 + 8r + 2 \int_{\R^d}\|x\|_2 \d \nu_1(x) + 2 \left(\int_{\R^d}\|x\|_2^2 \d \nu_1(x) \right)^\frac12 \right) W_2(\nu_1, \nu_2)
    \end{aligned}
  \end{equation}
  Using the Cauchy--Schwarz inequality, we have
  \begin{equation}
      \int_{\R^d}\|x\|_2 \d \nu_1(x)
      \le
      \left(\int_{\R^d}\|x\|_2^2 \d \nu_1(x) \right)^\frac12
  \end{equation}
  Exploiting \eqref{eq:first-mom},
  we have
  \begin{equation}
      |\mathcal E_K(\nu_1) - \mathcal E_K(\nu_2)|
      \le
      L \left(2 + 8r + 4 \left(2r + 2\int_{\R^d}\|y\|_2^2 \d \mu(x)\right)^\frac12 \right) W_2(\nu_1, \nu_2)
  \end{equation}
  and arrive at the assertion.
\end{proof}

\begin{lemma}
\label{lem:Klip}
    Let $K\colon \R^d \times \R^d \to \R$ be differentiable such that there exists $L>0$ with
    \begin{equation}
        \| \nabla K(x, y) \|_2 \le L(1 + \|x\|_2 + \|y\|_2).
    \end{equation}
    Then condition~\eqref{eq:Klip} is satisfied.
\end{lemma}

\begin{proof}
Fix $x, y, \in \R^d$, $r, s > 0$ and let $(x_1, y_1), (x_2, y_2) \in B_r(x) \times B_s(y)$ then we estimate
\[
\begin{aligned}
 &|K(x_2,y_2) - K(x_1, y_1)| 
 \\
 & = \left| \int_0^1 \nabla K(x_1 + t (x_2-x_1), y_1 + t(y_2 - y_1))^{\tT} (x_2 - x_1, y_2 - y_1) \d t \right| \\
 & \le \int_0^1 \|\nabla K(x_1 + t (x_2-x_1), y_1 + t(y_2 - y_1))\|_2 \|(x_2 - x_1, y_2 - y_1)\| \d t \\
 & \le \int_0^1 L(1 + \|x_1 + t (x_2 - x_1)\|_2 + \|y_1 + t (y_2 - y_1)\|_2) \|(x_2 - x_1, y_2 - y_1)\|_2 \d t \\
\ & \le \int_0^1 L(1 + \|x_1\|_2 + t \|x_2 - x_1\|_2 + \|y_1\|_2 + t \|y_2 - y_1\|_2) \|(x_2 - x_1, y_2 - y_1)\|_2 \d t \\
\ & \le L(1 + \|x_1\|_2 + \tfrac12 \|x_2 - x_1\|_2 + \|y_1\|_2 + \tfrac12 \|y_2 - y_1\|_2) \|(x_2 - x_1, y_2 - y_1)\|_2 \\
\ & \le L(1 + \|x_1\|_2 + r + \|y_1\|_2 + s) \|(x_2 - x_1, y_2 - y_1)\|_2 \\
\end{aligned}
\]
which yields the assertion.
\end{proof}

\noindent
{\bf Proof of Proposition~\ref{prop:lip-int-pot}}:
{\bf Part ($\mathcal E_{K}$):}
By Lemma~\ref{lem:EKLip} it is sufficient to show that the kernel $K(x,y) = - \|x-y\|_2^r$, satisfies for $r \in [1,2)$ the condition \eqref{eq:Klip}. For $r=1$ the condition follows from the fact that $K$ is Lipschitz continuous with Lipschitz constant $L=1$. 
For $r \in (1,2)$ the kernels $K$ are differentiable with
gradient
\[
  \nabla K(x,y) = (\nabla_1 K(x,y), - \nabla_1 K(x,y))^{\tT},
  \qquad \nabla_1 K(x,y) = 
  \begin{cases}
      0,& x = y,\\
      \frac{r (x-y)}{\|x-y\|_2^{2-r}}, & x \ne y.\\
  \end{cases}
\]
Lemma~\ref{lem:Klip} leads with the estimate
\[
    \|\nabla K(x,y)\|_2 = \sqrt{2} r \|x-y\|_2^{r-1} \le \sqrt{2} r (1 + \|x-y\|_2) \le L (1 + \|x\|_2 + \|y\|_2), \quad L \coloneqq \sqrt{2} r,
\]
for $r \in (1,2)$ and the previous observation for $r=1$ for $s_1,s_2 >0$
to the assertion
\begin{equation}
    \label{eq:RieszLocalLip}
    \mathrm{Lip}(K, B_{s_1}(x) \times B_{s_2}(y)) \le L (1 + \|x\|_2 + \|y\|_2 + s_1 + s_2), \qquad x, y \in \R^d.
\end{equation}

\noindent
{\bf Part ($\mathcal V_{K,\nu}$):}
By Proposition~\ref{prop:VLip} it is sufficient to show that the potential 
$V_{K, \nu}(x) = \int_{\R^d} K(x,y) \d \nu(y)$, $\nu \in \P_2(\R^d)$ satisfies the condition~\eqref{eq:Vlip}. Fix $x \in \R^d$, $s>0$ and let $x_1, x_2 \in B_s(x)$ with $x_1 \ne x_2$.
Then by the previous findings \eqref{eq:RieszLocalLip} we estimate
\begin{align}
\frac{|V_{K, \nu}(x_1) - V_{K, \nu}(x_2)|}{\|x_1 - x_2\|_2}
& \le \int_{\R^d} \frac{|K(x_1,y) - K(x_2,y)|}{\|x_1 - x_2\|_2} \d \nu(y) 
 \le \int_{\R^d} \mathrm{Lip}(K, B_s(x) \times \{y\}) \d \nu(y) \\
& \le \int_{\R^d} L(1 + \|x\|_2 + s + \|y\|_2) \d \nu(y) \\
& \le L(1 + \int_{\R^d} \|y\|_2 \d \nu(y) + \|x\|_2 + s). 
\label{eq:lip-int-fun}
\end{align}
\hfill $\Box$

\section{Proofs from Section~\ref{sec:mmd}}\label{proofs:mmd}

\subsection{Proof of Theorem~\ref{thm:HD-spec}} \label{app_supp}

To establish the claim, we require the following integrals over the sphere.

\begin{proposition}\label{prop:int_dist_Sd}
  Let $d\in\N$ with $d\ge2$, $r\in(0,2)$ and $R>0$. Then it holds
	{\small
  \begin{equation}         
    \begin{aligned}
      & \int_{R\mathbb S^{d-1}} \|x-y\|_2^r \, \d \mathcal U_{R \mathbb S^{d-1}}(y)
      & = \begin{cases}
        R^r \;_2F_1 \left(-\frac{r}{2}, \frac{2-r-d}{2}; \frac{d}{2}; \frac{\|x\|_2^2}{R^2} \right), & \|x\|_2 \le R,\\
        \|x\|_2^r \;_2F_1 \left(-\frac{r}{2}, \frac{2-r-d}{2}; \frac{d}{2}; \frac{R^2}{\|x\|_2^2} \right), & R \le \|x\|_2.\\
      \end{cases}
    \end{aligned} 
  \end{equation}
	}%
\end{proposition}

\begin{proof}
In order to prove the claim, we consider the orthogonal polynomials $P_n^{(\frac{d-2}{2})}$ with respect 
to the weight function $(1-t^2)^{\frac{d-3}{2}}$.
For $d>2$, these are the Gegenbauer polynomials with normalization
{\small
  \begin{align}
    \int_{-1}^1 P^{(\frac{d-2}{2})}_n(t) P^{(\frac{d-2}{2})}_m(t) (1-t^2)^{\frac{d-3}{2}} \d t =
    \begin{cases}
      \frac{\pi 2^{3-d}\Gamma(n+d-2)}{n! (n+\tfrac{d-2}{2}) \Gamma(\tfrac{d-2}{2})^2}, & n=m,\\
      0, & n \ne m.
    \end{cases}
  \end{align}
	}%
In particular, we have for $n=0$ that
{\small
\begin{align}
\int_{-1}^1 P^{(\frac{d-2}{2})}_0(t) P^{(\frac{d-2}{2})}_0(t) (1-t^2)^{\frac{d-3}{2}} \d t&=\frac{\pi 2^{3-d}\Gamma(d-2)}{\tfrac{d-2}{2} \Gamma(\tfrac{d-2}{2})^2}
=\frac{\pi 2^{3-d}\Gamma(\tfrac{d-2}{2})\Gamma(\frac{d-1}{2})}{ \Gamma(\tfrac{d}{2})\Gamma(\tfrac{d-2}{2})\pi^{1/2}2^{3-d}} \nonumber
\\
&=\frac{\pi^{1/2}\Gamma(\frac{d-1}{2})}{\Gamma(\tfrac{d}{2})}=\Beta(\tfrac{1}{2},\tfrac{d-1}{2}), \label{eq:zero_gegenbauer}
\end{align}
}%
where we used $z\Gamma(z)=\Gamma(z+1)$ and $\Gamma(z)\Gamma(z+1/2)=2^{1-2z}\pi^{1/2}\Gamma(2z)$ and where $\Beta(a,b)$ is the beta function.
For $d=2$, we obtain the Chebyshev polynomials of first kind with normalization
{\small
  \begin{align}
    \int_{-1}^1 P^{(0)}_n(t) P^{(0)}_m(t) (1-t^2)^{-1/2} \d t =
    \begin{cases}
      \pi, & n=m=0,\\
      \pi/2, & n=m\neq 0\\
      0, & n \ne m.
    \end{cases}
  \end{align}
	}%
By definition, we obtain also for $d=2$ that
{\small
\begin{equation}\label{eq:zero_chebyshev}
\int_{-1}^1 P^{(\tfrac{d-2}{2})}_0(t) P^{(\tfrac{d-2}{2})}_0(t) (1-t^2)^{\tfrac{d-3}{2}} \d t=\pi=\Beta(\tfrac12,\tfrac12)=\Beta(\tfrac12,\tfrac{d-1}2).
\end{equation}
}%
Now, it holds by \cite[Section~2]{BaHu01} that for any $c>0$ the function $t\mapsto (2-2t+c^2)^{r/2}$ can be expanded for $t\in[-1,1]$
as
{\small
$$
(2-2t+c^2)^{r/2}=\sum_{n=0}^{\infty} a_n(\tfrac{r}{2},\tfrac{d-2}{2})P_n^{(\frac{d-2}{2})}(t),
$$
}%
for some coefficients $a_n(\tfrac{r}{2},\tfrac{d-2}{2})$.
By definition of Chebyshev and Gegenbauer polynomials it holds that $P^{(\frac{d-2}{2})}_0(t)=1$ for all $t$.
Therefore, we can compute this coefficient for $n=0$ as
{\small
  \begin{align}
  a_0(\tfrac{r}2,\tfrac{d-2}{2})
  &=
  \frac{1}{\Beta(\tfrac12,\tfrac{d-1}2)}
  \int_{-1}^1 (2 - 2t + c^2)^{r/2} \, P_0^{(\frac{d-2}{2})}(t) (1 - t^2)^{\frac{d-3}{2}} \d t\\
  &=\frac{1}{\Beta(\tfrac12,\tfrac{d-1}2)}\int_{0}^2 (4+c^2-2t)^{r/2}(2t-t^2)^{\frac{d-3}{2}} \d t\\
  &=\frac{2}{\Beta(\tfrac12,\tfrac{d-1}2)}\int_{0}^1 (4+c^2-4t)^{r/2}(4t-4t^2)^{\frac{d-3}{2}} \d t\\
  &=\frac{2^{d-2}(4+c^2)^{r/2}}{\Beta(\tfrac12,\tfrac{d-1}2)}\int_{0}^1 \Big(1-\frac{4}{4+c^2}t\Big)^{r/2}t^{\frac{d-3}{2}}(1-t)^{\frac{d-3}{2}} \d t,
  \end{align}
	}%
  where we substitute $t$ by $t-1$ in the first and $t$ by $2t$ in the second equality.
  Using Euler's integral formula \cite[§~2.1.3, (10)]{Ba53},
  this is equal to
	{\small
  \begin{align}
  a_0(\tfrac{r}2,\tfrac{d-2}{2})
  &=
  (4+c^2)^{r/2}{_2F_1}(-\tfrac{r}{2},\tfrac{d-1}2,d-1,\tfrac{4}{4+c^2}).
  \end{align}
	}%
Note that for $d>2$, this is consistent with the computations from \cite[Section~2.1]{BaHu01}.
  Now, let $x\in\R^{d}\setminus\{0\}$, $y\in\mathbb S^{d-1}$ and choose $t={x^{\tT}y}/{\|x\|_2}$ and $c^2={(\|x\|_2-1)^2}/{\|x\|_2}$.
  Then, it holds 
	{\small
  $$
  (2-2t+c^2)=\frac{2\|x\|_2-2x^\tT y+\|x\|_2^2+1-2\|x\|_2}{\|x\|_2}=\frac{\|x-y\|_2^2}{\|x\|_2}
  $$
	}%
  and
	{\small
  $$
  (4+c^2)=\frac{4\|x\|_2+\|x\|_2^2-2\|x\|_2+1}{\|x\|_2}=\frac{(\|x\|_2+1)^2}{\|x\|_2}.
  $$
	}%
  In particular, we have 
	{\small
  \begin{equation}\label{eq:gegenbauer_expansion}
    \|x-y\|^r=\sum_{n=0}^\infty a_n(\tfrac{r}{2},\tfrac{d-2}{2})\|x\|^{\tfrac{r}{2}}P_n^{(\frac{d-2}{2})}(\tfrac{x^{\tT}y}{\|x\|_2})
    \vspace{-5pt}
  \end{equation}
	}%
  with\vspace{-5pt}
	{\small
  \begin{align}\label{eq:zero_coefficient}
  a_0(\tfrac{r}2,\tfrac{d-2}{2})
  &=\frac{(\|x\|+1)^r}{\|x\|^{r/2}}{_2F_1}(-\tfrac{r}{2},\tfrac{d-1}2,d-1,\tfrac{4\|x\|}{(\|x\|+1)^2}).
  \end{align}
	}%
  Due to symmetry, we can choose wlog $x$ as $e\|x\|_2$, where $e$ is the first unit vector.
  Moreover, we denote $y\in\mathbb S^{d-1}$ by $y=y_1 e+y_{(d-1)}$ with $y_1\in[0,1]$ and $y_{(d-1)}\in\{0\}\times\mathbb S^{d-2}$.
  Then, we compute the integral over the unit sphere by \cite[(1.16)]{AH2012} as 
	{\small
  \begin{align}
    \int_{\mathbb S^{d-1}} \|x-y\|_2^r \, \d \mathbb S^{d-1}(y)=\int_{-1}^1\int_{\mathbb S^{d-2}}\|x-y\|_2^r\,\d \mathbb S^{d-2}(y_{(d-1)})(1-y_1^2)^{\tfrac{d-3}{2}}\,\d (y_1)
  \end{align}
	}%
  By inserting \eqref{eq:gegenbauer_expansion} and using $\frac{x^\tT y}{\|x\|_2}=\frac{\|x\|_2 y_1}{\|x\|_2}=y_1$, the above formula becomes
	{\small
  $$
  \|x\|^{\tfrac{r}{2}}\int_{-1}^1\int_{\mathbb S^{d-2}}\sum_{n=0}^\infty a_n(\tfrac{r}{2},\tfrac{d-2}{2})P_n^{(\frac{d-2}{2})}(y_1)\,\d\mathbb S^{d-2}(y_{(d-1)})(1-y_1^2)^{\tfrac{d-3}{2}}\,\d y_1.
  $$
	}%
  This does not depend on $y_{(d-1)}$. Therefore, by using the volume formula over the sphere $\int_{\mathbb S^{d-2}}1\,\d\mathbb S^{d-2}=\frac{2\pi^{(d-1)/2}}{\Gamma(\tfrac{d-1}{2})}$, it is equal to
	{\small
  $$
  \|x\|^{\tfrac{r}{2}}\frac{2\pi^{(d-1)/2}}{\Gamma(\tfrac{d-1}{2})}\int_{-1}^1\sum_{n=0}^\infty a_n(\tfrac{r}{2},\tfrac{d-2}{2})P_n^{(\frac{d-2}{2})}(y_1)(1-y_1^2)^{\tfrac{d-3}{2}}\,\d y_1.
  $$
	}%
  By interchanging the sum and the integral and adding the factor $P_0^{(\frac{d-2}{2})}(y_1)=1$, this is equal to
	{\small
  $$
  \|x\|^{\tfrac{r}{2}}\sum_{n=0}^\infty a_n(\tfrac{r}{2},\tfrac{d-2}{2})\frac{2\pi^{(d-1)/2}}{\Gamma(\tfrac{d-1}{2})}\int_{-1}^1P_0^{(\frac{d-2}{2})}(y_1)P_n^{(\frac{d-2}{2})}(y_1)(1-y_1^2)^{\tfrac{d-3}{2}}\,\d y_1.
  $$
	}%
Due to the orthogonality property of the polynomials $P_n^{(\frac{d-2}{2})}$, all summands despite $n=0$ are zero. Moreover, we can insert for $n=0$ the formulas \eqref{eq:zero_gegenbauer} and \eqref{eq:zero_chebyshev}. Then, the above term is equal to
{\small
\begin{align}
  \|x\|^{\tfrac{r}{2}}\frac{2\pi^{(d-1)/2}}{\Gamma(\tfrac{d-1}{2})}\Beta(\tfrac12,\tfrac{d-1}{2})a_0(\tfrac{r}{2},\tfrac{d-2}{2})
&=  \|x\|^{\tfrac{r}{2}}\frac{2\pi^{(d-1)/2}}{\Gamma(\tfrac{d-1}{2})}\frac{\pi^{1/2}\Gamma(\tfrac{d-1}{2})}{\Gamma(\tfrac{d}{2})}a_0(\tfrac{r}{2},\tfrac{d-2}{2})\\
&=\|x\|^{\tfrac{r}{2}}\frac{2\pi^{d/2}}{\Gamma(\tfrac{d}{2})}a_0(\tfrac{r}{2},\tfrac{d-2}{2}).
\end{align}
}%
Summarizing, we obtain by inserting \eqref{eq:zero_coefficient} 
{\small
  $$
  \int_{\mathbb S^{d-1}} \|x-y\|_2^r \, \d \mathbb S^{d-1}(y)=\frac{2\pi^{d/2}}{\Gamma(\frac{d}{2})}(\|x\|+1)^r {_2F_1}(-\tfrac{r}{2}, \tfrac{d-1}{2}; d-1; \tfrac{4\|x\|}{(\|x\|+1)^2}).
  $$
	}%
  By normalizing the volume of $\mathbb S^{d-1}$ and rescaling with a factor $R$ this is equivalent to
	{\small
  $$
  \int_{R\mathbb S^{d-1}} \|x-y\|_2^r \, \d \mathcal U_{R\mathbb S^{d-1}}(y)=(\|x\|+R)^r {_2F_1}(-\tfrac{r}{2}, \tfrac{d-1}{2}; d-1; \tfrac{4\|x\|R}{(\|x\|+R)^2}).
  $$
	}%
  Finally, the claim follows by the quadratic transformation rule due to Gauss \cite[(2.11(5))]{HTF1953} given by
	{\small
  \[
    (1+t)^{-2a}\,_2F_1(a, b; 2b; 4t/(1+t)^2) = \,_2F_1(a, a+ \tfrac12 - b; b + \tfrac12; t^2), \qquad t \in [0,1],
  \]
	}%
  with $t = \|x\|_2 /R$ for $\|x\|_2 \le R$ and $t = R / \|x\|_2$ for $R \le \|x\|_2$.
\end{proof}

		Further, we will need some auxiliary results on hypergeometric functions.
			
		\begin{lemma}\label{lem:hyper_convex}
Let $d+r<4$.
Then, it holds for any $x\in[0,1]$ that
{\small
\begin{align}
&{_2F_1}(1-\tfrac{r}{2},2-\tfrac{d+r}{2};3-\tfrac{r}{2};x)
\geq {_2F_1}(1-\tfrac{r}{2},2-\tfrac{d+r}{2},3-\tfrac{r}{2};1)\\
& \quad -(1-x)\frac{(2-r)(4-r-d)}{2(6-r)}{_2F_1}(2-\tfrac{r}{2};3-\tfrac{d+r}{2};4-\tfrac{r}{2};1).
\end{align}
}%
\end{lemma}
\begin{proof}
By the definition of hypergeometric functions via Pochhammer symbols, we have for $x\in[0,1]$ that
{\small
\begin{align}
{_2F_1}(1-\tfrac{r}{2},2-\tfrac{d+r}{2};3-\tfrac{r}{2};x)
&=\sum_{n=0}^\infty \frac{(1-\tfrac{r}{2})_n(2-\tfrac{d+r}{2})_n}{(3-\tfrac{r}{2})_n n!} x^n\\
&=\sup_{N_\mathrm{max}\in\N}\Big\{\sum_{n=0}^{N_\mathrm{max}} \frac{(1-\tfrac{r}{2})_n(2-\tfrac{d+r}{2})_n}{(3-\tfrac{r}{2})_n n!} x^n\Big\},
\end{align}
}%
where the last equality is due to the fact that all coefficients are non-negative. 
Further, the non-negativity of the coefficients implies that 
{\small 
$$\sum\limits_{n=1}^{N_\mathrm{max}}\frac{(1-\tfrac{r}{2})_n(2-\tfrac{d+r}{2})_n}{(3-\tfrac{r}{2})_n n!} x^n$$ 
}%
is convex on $[0,1]$ for any $N_\mathrm{max}$.
Therefore $x\mapsto {_2F_1}(1-\tfrac{r}{2},2-\tfrac{d+r}{2};3-\tfrac{r}{2};x)$ 
is convex as a supremum of convex functions.
Now the claim follows by the identity $f(x)\geq f(1)+(x-1)f'(1)$ 
for convex functions and the derivative rule for hypergeometric functions.
\end{proof}

We need the following 
lemma to prove Proposition~\ref{lem:min_2F1_a} below.

\begin{lemma} \label{lem:hyper_decreasing}
Let $d\in \N$ and $r\in(0,2)$ with $d+r\geq 4$. Then it  holds
\begin{enumerate}[\upshape(i)]
    \item 
$
\frac12 {_2F_1}(-\frac{r}{2},-\frac{d+r-2}{2};\frac{d}{2};1) = \frac{d+r-2}{d}{_2F_1}(\frac{2-r}{2},-\frac{d+r-4}{2};\frac{d+2}{2};1)
$, and
\item
${_2F_1}(\tfrac{4-r}{2},\tfrac{6-d-r}{2};\tfrac{d}{2}+2;x)\geq 0$
and
${_2F_1}(\tfrac{2-r}{2},\tfrac{4-d-r}{2};\tfrac{d}{2}+1;x)\geq 0$
for any $x\in(0,1)$.
\end{enumerate}
The hypergeometric function $f(x)= {_2 F_1}(\frac{2-r}2, -\frac{d+r-4}{2}; \frac{d+2}{2}; x)$ is here decreasing on $[0,1]$.
\end{lemma}

\begin{proof}
(i) By Gauss's summation formula for hypergeometric functions \cite[(3.1)]{Ko98}, we have 
{\small
\begin{equation} \label{gs}
{_2F_1}(a,b;c;1)=\frac{\Gamma(c)\Gamma(c-a-b)}{\Gamma(c-a)\Gamma(c-b)}, \quad  c>a+b,\, c>0.
\end{equation}
}%
Thus, we obtain together with $\Gamma(x+1)/x=\Gamma(x)$ the first assertion
{\small
$$
\frac{1}{2} \frac{\Gamma(d/2)\Gamma(d+r-1)}{\Gamma((d+r)/2)\Gamma(d+r/2-1)} = \frac{d+r-2}{d}\frac{\Gamma(d/2+1)\Gamma(d+r-2)}{\Gamma((d+r)/2)\Gamma(d+r/2-1)}.
$$
}%
(ii)
First, we show that ${_2F_1}(\frac{4-r}{2},\frac{6-d-r}{2};\frac{d}{2}+2;x)>0$. 
Since $_2F_1$ is by definition symmetric in the first two arguments, we obtain
{\small
$$
{_2F_1}(\tfrac{4-r}{2},\tfrac{6-d-r}{2};\tfrac{d}{2}+2;x)={_2F_1} (\tfrac{6-d-r}{2},\tfrac{4-r}{2};\tfrac{d}{2}+2;x).
$$
}%
By Euler's integral formula \cite[§~2.1.3, (10)]{Ba53}, this is equal to
{\small
$$
\frac{\Gamma(d/2+2)}{\Gamma((4-r)/2)\Gamma(d/2+r/2)}\int_0^1 t^{(2-r)/2}(1-t)^{d/2-(2-r)/2}(1-tx)^{d/2+r/2-3} \, \d t.
$$   
}%

Since the $\Gamma$-function is positive on $\R_{>0}$ and since for $t,x\in(0,1)$ it holds $t,(1-t),(1-tx)>0$, this is greater than zero and the first claim is proven.
The proof of the second inequality works analogously. 
Finally, by the derivative rule for hypergeometric functions and $d+r\geq 4$, it follows
{\small
$$
f'(x)=-\tfrac{(2-r)(d+r-4)}{2d+4} {_2F_1} (\tfrac{4-r}{2},\tfrac{6-d-r}{2};\tfrac{d}{2}+2;x )\leq 0,
$$
}%
so that $f$ is decreasing.
\end{proof}

\begin{proposition} \label{lem:min_2F1_a}
Let $d\in\N$ and $r\in(0,2)$ with $d+r\geq 4$ and
{\small
  \begin{align}
    h(x)
    &\coloneqq
      - \,_2 F_1(-\tfrac{r}2, -\tfrac{d+r-2}{2}; \tfrac{d}{2}; x^2) 
			+ \tfrac12 \tfrac{r(d+r-2)}{d} \,_2 F_1(\tfrac{2-r}2, -\tfrac{d+r-4}{2}; \tfrac{d+2}{2}; 1) \, x^2,
    \\
    \tilde h (x)
    &\coloneqq
      - _2 F_1(-\tfrac{r}2, -\tfrac{d+r-2}{2}; \tfrac{d}{2}; \tfrac{1}{x^2}) \, x^r
    + \tfrac12 \tfrac{r(d+r-2)}{d} \,_2 F_1(\tfrac{2-r}2, -\tfrac{d+r-4}{2}; \tfrac{d+2}{2}; 1) \, x^2.
    \\[-20pt]
  \end{align}
	}%
Then it holds $1 \in \argmin_{x \in [0,1]}h(x)$ and $1\in \argmin_{x \in [1,\infty)} \tilde h(x)$. 
\end{proposition}

\begin{proof}
(i)
Using the derivative rule for hypergeometric functions, we obtain
{\small
\begin{equation}
h'(x)=\tfrac{r(d+r-2)}{d}x \Big({_2 F_1}(\tfrac{2-r}2, -\tfrac{d+r-4}{2}; \tfrac{d+2}{2}; 1)-{_2 F_1}(\tfrac{2-r}2, -\tfrac{d+r-4}{2}; \tfrac{d+2}{2}; x^2)\Big)
\end{equation}
}%
and in particular $h'(1)=0$.
By Lemma~\ref{lem:hyper_decreasing}(ii) and since $x\mapsto x^2$ is strictly increasing on $[0,1]$, the function  ${_2 F_1}(\frac{2-r}2, -\frac{d+r-4}{2}; \frac{d+2}{2}; x^2)$ is decreasing on $[0,1]$.
In particular, we conclude
{\small
$$
{_2 F_1}(\tfrac{2-r}2, -\tfrac{d+r-4}{2}; \tfrac{d+2}{2}; 1)
-{_2 F_1}(\tfrac{2-r}2, -\tfrac{d+r-4}{2}; \tfrac{d+2}{2}; x^2)
\begin{cases}=0,&$for $x=1,\\\leq 0,&$for $x\in[0,1).\end{cases}
$$
}%
This implies $h'(x)\leq 0$ on $(0,1)$ such that $h$ is decreasing on $[0,1]$ which yields the first claim.
\\[1ex]
(ii)
We show that $\tilde h'(x) \geq 0$ for $x\in[1,\infty)$.
We obtain
{\small
\begin{align}
\tilde h'(x)
&=-rx^{r-1}{_2F_1}(-\tfrac{r}{2},-\tfrac{d+r-2}{2};\tfrac{d}{2};\tfrac{1}{x^2})+\tfrac{r(d+r-2)}{d x^{3-r}} {_2F_1}(\tfrac{2-r}{2},-\tfrac{d+r-4}{2};\tfrac{d+2}{2};\tfrac{1}{x^2})\\
&
\quad+\tfrac{r(d+r-2)x}{d} {_2F_1}(\tfrac{2-r}{2},-\tfrac{d+r-4}{2};\tfrac{d+2}{2};1).
\end{align}
}%
Now $h'(x)\geq0$ on $[1,\infty)$ is equivalent to $\tfrac{h'(x)}{r x^{r-1}}\geq 0$ on $[1,\infty)$.
By Lemma~\ref{lem:hyper_decreasing}(ii) and since $x\mapsto1/x^2$ is decreasing, 
the function $x\mapsto{_2F_1}(\frac{2-r}{2},-\frac{d+r-4}{2};\frac{d+2}{2};\frac{1}{x^2})$ is increasing on $[1,\infty)$ such that we have for $x\in[1,\infty)$ that
{\small
\begin{align}
\tfrac{\tilde h'(x)}{rx^{r-1}}\geq g(x)
&\coloneqq-{_2F_1}(-\tfrac{r}{2},-\tfrac{d+r-2}{2};\tfrac{d}{2};\tfrac{1}{x^2})+\tfrac{(d+r-2)}{dx^{2}} {_2F_1}(\tfrac{2-r}{2},-\tfrac{d+r-4}{2};\tfrac{d+2}{2};1)\\
&\quad
+\tfrac{(d+r-2)x^{2-r}}{d}{_2F_1}(\tfrac{2-r}{2},-\tfrac{d+r-4}{2};\tfrac{d+2}{2};1)\\
&=-{_2F_1}(-\tfrac{r}{2},-\tfrac{d+r-2}{2};\tfrac{d}{2};\tfrac{1}{x^2})
\\
&\quad
+(x^{2-r}
+\tfrac1{x^{2}})\tfrac{(d+r-2)}{d}{_2F_1}(\tfrac{2-r}{2},-\tfrac{d+r-4}{2};\tfrac{d+2}{2};1).\label{eq:H_der_lower_bound}
\end{align}
}%
For $x=1$, we obtain by Lemma~\ref{lem:hyper_decreasing}(i) that
{\small
\begin{align}
g(1)
&=-{_2F_1}(-\tfrac{r}{2},-\tfrac{d+r-2}{2};\tfrac{d}{2};1)+\tfrac{2\, (d+r-2)}{d}{_2F_1}(\tfrac{2-r}{2},-\tfrac{d+r-4}{2};\tfrac{d+2}{2};1)= 0.
\end{align}
}%

Thus, it suffices to show that $g$ is increasing on $[1,\infty)$.
Taking the derivative of $g$ on $[1,\infty)$ gives 
{\small
\begin{align}
g'(x)&=\tfrac{r}{x^3}\tfrac{(d+r-2)}{d}{_2F_1}(\tfrac{2-r}{2},-\tfrac{d+r-4}{2};\tfrac{d+2}{2};\tfrac{1}{x^2})\\
&\quad+\big((2-r)x^{1-r}-\tfrac{2}{x^{3}}\big)\tfrac{(d+r-2)}{d}{_2F_1}(\tfrac{2-r}{2},-\tfrac{d+r-4}{2};\tfrac{d+2}{2};1)\\
&\geq\tfrac{r}{x^3}\tfrac{(d+r-2)}{d}{_2F_1}(\tfrac{2-r}{2},-\tfrac{d+r-4}{2};\tfrac{d+2}{2};1)\\
&\quad+\big((2-r)x^{1-r}-\tfrac{2}{x^{3}}\big)\tfrac{(d+r-2)}{d}{_2F_1}(\tfrac{2-r}{2},-\tfrac{d+r-4}{2};\tfrac{d+2}{2};1)\\
&=
(2-r)\big(x^{1-r}-\tfrac{1}{x^{3}}\big) \tfrac{(d+r-2)}{d}{_2F_1}(\tfrac{2-r}{2},-\tfrac{d+r-4}{2};\tfrac{d+2}{2};1)\geq 0,
\end{align}
}%
and we are done.
\end{proof}
\vspace{0.5cm}

\noindent		
		{\bf Proof of Theorem~\ref{thm:HD-spec}(i):}
		In \cite{CaHu17} it was shown that for $d+r < 4$  the measure $\eta^*_\tau$ in (i)
    fulfills the equality condition in \eqref{eq:opt_cond}, see also \cite[Lem~2.4]{GCO2021}.
		We have to show the inequality condition.
		\\[2ex]
		\textbf{Case: $d=1$}: Since $\supp(\eta^*_\tau)=[-s_\tau,s_\tau]$, it remains to show that 
	{\small	
$$
f(x_1)\coloneqq \int_\R K(x_1,x_2)\d\eta^*_\tau(x_2)+\frac1{2\tau}\|x_1\|^2
$$    
}%
is increasing on $[s_\tau,\infty)$ and decreasing on $(-\infty,s_\tau]$. 
Due to the symmetry, it suffices to show that $f$ is increasing on $[s_\tau,\infty)$.
For $x_1\in[s_\tau,\infty)$, we have $x_1>x_2$ for all $x_2\in\supp(\eta^*)$ such that we can reformulate $f$ as
{\small
$$
f(x_1)=-A_{s_\tau}\int_{-s_\tau}^{s_\tau} (x_1-x_2)^r (s_\tau^2-x_2^2)^{(1-r)/2}\d x_2+\tfrac1{2\tau}x_1^2
$$ 
}%
Thus, its derivative on $(s_\tau,\infty)$ is given by
\begin{equation}\label{eq:derivative_opt_cond}
f'(x_1)=g(x_1)\coloneqq-rA_{s_\tau}\int_{-s_\tau}^{s_\tau} (x_1-x_2)^{r-1} (s_\tau^2-x_2^2)^{(1-r)/2}\d x_2+\tfrac1{\tau}x_1
\end{equation}
Now, we show first that $g(s_\tau)=0$ and second that $g$ is increasing on $(s_\tau,\infty)$. 
Together this shows that $f'>0$ such that $f$ is increasing on $[s_\tau,\infty)$ and we are done.

By \cite[Cor.~2.5]{GCO2021}, with $m=1$, $\beta=r-1$ and $\alpha=r$ it holds 
{\small
\begin{align}
\tfrac{g(s_\tau)}{s_\tau}
&=
-rA_{s_\tau}\frac{\pi^{\frac{1}{2}}}{\Gamma(\tfrac12)}\Beta(\tfrac{r}{2},\tfrac{3-r}{2}) \, {_2F_1}(\tfrac{1-r}{2},-\tfrac12;\tfrac12;1)+\tfrac1\tau.
\end{align}
}%
Using Gauss's summation formula for hypergeometric functions \eqref{gs}, the above equation can be reformulated as
{\small
\begin{align}
    \tfrac{g(s_\tau)}{s_\tau}
    &=
    -rs_\tau^{-(2-r)}
    \tfrac{
        \Gamma(\frac12)\Gamma(2-\frac{r}{2})
        }{
        \pi^{\frac{1}{2}}\Gamma(\frac12)\Gamma(\frac{3-r}{2})}
    \tfrac{
        \pi^{\frac{1}{2}}
        }{
        \Gamma(\frac12)}
    \tfrac{
        \Gamma(\frac{r}{2})\Gamma(\frac{3-r}{2})
        }{
        \Gamma(\frac32)}
    \tfrac{
        \Gamma(\frac12)\Gamma(\frac{1+r}{2})
        }{
        \Gamma(\frac{r}{2})\Gamma(1)}
    +\tfrac1\tau\\
    &=
    -rs_\tau^{-(2-r)}
    \tfrac{
        \Gamma(2-\frac{r}{2})\Gamma(\frac{1+r}{2})
        }{
        \Gamma(\frac32)}
    +\tfrac1\tau\\
    &=
    -r
    \tfrac{
        \frac{1}{2}\Gamma(\frac12)
        }{
        \Gamma(2-\frac{r}{2})\Gamma(\frac{1+r}{2})r\tau}
    \frac{
        \Gamma(2-\frac{r}{2})\Gamma(\frac{1+r}{2})
        }{
        \Gamma(\tfrac32)}
    +\tfrac1\tau
    =
    -\tfrac1\tau
    \tfrac{
        \Gamma(\frac12)
        }{
        2\Gamma(\frac32)}
    +\tfrac1\tau=0.
\end{align}
}%
This implies that $g(s_\tau)=0$.

Further, it holds for $r\in(0,1]$ that $(x_1-x_2)^{r-1}$ is decreasing on $[s_{\tau},\infty)$ 
for any $x_2\in[-s_\tau,s_\tau]$. 
Therefore both terms in \eqref{eq:derivative_opt_cond} are increasing and we are done for this case.
For $r\in(1,2)$, we take again the derivative of $g$ and arrive at
$$
g'(x_1)\coloneqq-r(r-1)A_{s_\tau}\int_{-s_\tau}^{s_\tau} (x_1-x_2)^{r-2} (s_\tau^2-x_2^2)^{(1-r)/2}\d x_2+\tfrac1{\tau}.
$$
By the same arguments as above we have that $g'$ is increasing on $([s_\tau,\infty)$ and
using \cite[Cor.~2.5]{GCO2021}, with $m=1$, $\beta=r-2$ and $\alpha=r$ and Gauss's summation formula, we obtain 
{\small
\begin{align}
    g'(s_\tau)
    &=
    -r(r-1)A_{s_\tau}
    \tfrac{\pi^{\frac{1}{2}}}{\Gamma(\frac12)}
    \Beta(\tfrac{r-1}{2},\tfrac{3-r}{2}) {_2F_1}(\tfrac{2-r}{2},0,\tfrac12;1)+\tfrac1\tau\\
    &=
    -r(r-1)s_\tau^{-(2-r)}
    \tfrac{
        \Gamma(\frac12)\Gamma(2-\frac{r}{2})
        }{
        \pi^{\frac{1}{2}}\Gamma(\frac12)\Gamma(\frac{3-r}{2})}
    \tfrac{
        \pi^{\frac{1}{2}}
        }{
        \Gamma(\frac12)}
    \tfrac{
        \Gamma(\frac{r-1}{2})\Gamma(\frac{3-r}{2})
        }{
        \Gamma(1)}
    +\tfrac1\tau\\
    &=
    -r(r-1)s_\tau^{-(2-r)}
    \tfrac{
        \Gamma(2-\frac{r}{2})\Gamma(\frac{r-1}{2})
        }{
        \Gamma(\frac12)}
    +\tfrac1\tau\\
    &=
    -r(r-1)
    \tfrac{
        \frac{1}{2}\Gamma(\frac12)
        }{
        \Gamma(2-\frac{r}{2})\Gamma(\frac{1+r}{2})r\tau}
    \tfrac{
        \Gamma(2-\frac{r}{2})\Gamma(\frac{r-1}{2})
        }{
        \Gamma(\frac12)}
    +\tfrac1\tau\\
    &=
    -\tfrac{
        \frac{r-1}{2}\Gamma(\frac{r-1}{2})
        }{
        \Gamma(\frac{1+r}{2})\tau}
    +\tfrac1\tau
    =-\tfrac{1}{\tau}+\tfrac1\tau 
    = 0.
\end{align}   
}%
This implies that $g$ is increasing on $[s_\tau,\infty)$ and we are done.

\medskip

\noindent
\textbf{Case $d\geq 2$:}
Let $e_1$ be the first unit vector in $\mathbb R^d$. Choose $\tau$ such that $s_\tau=1$ and let $t\geq1$, i.e., 
{\small
$$
\tau
=
\tfrac{
    \frac{d}{2}\Gamma(\frac{d}{2})
    }{
    r\Gamma(2-\frac{r}2)\Gamma(\frac{d+r}{2})}
=
\tfrac{
    \Gamma(1+\frac{d}{2})
    }{
    r\Gamma(2-\frac{r}2)\Gamma(\frac{d+r}{2})}.
$$
}%
We consider
{\small
$$
h(t)\coloneqq
-\int_{\R^d}\|te_1-x\|^r\d\eta_\tau^*(x)+\frac{1}{2\tau}t^2
=
-A_1\int_{B_1}\|te_1-x\|^r(1-\|x\|^2)^{1-\frac{d+r}{2}}\d x
+\tfrac{1}{2\tau}t^2.
$$
}%
We aim to show that $h(t)\geq h(1)$ for all $t\geq 1$.
Changing the order of integration, the integral over $B_1$ is equal to
{\small
$$
I(t) = \int_{0}^1\int_{R\mathbb{S}^{d-1}}\|te_1-x\|^r\d \mathcal U_{R\mathbb{S}^{d-1}}(x) \,
\tfrac{2\pi^{\frac{d}{2}} R^{d-1}}{\Gamma(\frac{d}{2})}(1-R^2)^{1-\tfrac{r+d}{2}} \, \d R.
$$
}%
Now, the inner integral can be computed by Propostion~\ref{prop:int_dist_Sd}. Then the above formula becomes
{\small
$$
I(t) =\frac{\pi^{\frac{d}{2}}t^r}{\Gamma(\tfrac{d}{2})} \int_0^1 {_2F_1}(-\tfrac{r}{2},1-\tfrac{r+d}{2};\tfrac{d}{2};\tfrac{R^2}{t^2})
R^{d-2}(1-R^2)^{1-\tfrac{r+d}{2}}2R \, \d R.
$$ 
}%
Using the substitution $S=R^2$ (``$\d S=2R\d R $''), this is equal to
{\small
$$
I(t) =\frac{\pi^{\frac{d}{2}}t^r}{\Gamma(\tfrac{d}{2})}\int_0^1 
{_2F_1}(-\tfrac{r}{2},1-\tfrac{r+d}{2};\tfrac{d}{2};\tfrac{S}{t^2})S^{\tfrac{d}{2}-1}(1-S)^{1-\tfrac{d+r}{2}}\d S.
$$}

Now using Euler's integral transform \cite[eqt (4.1.2)]{S1966} for generalized hypergeometric functions
{\small
\begin{align*}
&{_3F_2}(a_1,a_2,a_3;b_1,b_2;z)
\\
&=
\frac{\Gamma(b_2)}{\Gamma(a_3)\Gamma(b_2-a_3)}\int_0^1 S^{a_3-1}(1-S)^{b_2-a_3-1} \, {_2F_1}(a_1,a_2;b_1;Sz) \,\d 
\end{align*}
}%
with $a_1=-\tfrac{r}2$, $a_2=1-\tfrac{r+d}{2}$, $a_3=\tfrac{d}{2}$, $b_1=\tfrac{d}{2}$, $b_2=2-\tfrac{r}{2}$ and $z=\tfrac{1}{t^2}$,
we obtain 
{\small
$$
I(t) 
=
\tfrac{
    \pi^{\frac{d}{2}}t^r
    }{
    \Gamma(\frac{d}{2})} 
\tfrac{
    \Gamma(\frac{d}{2})\Gamma(2-\frac{r+d}{2})
    }{
    \Gamma(2-\frac{r}{2})}
{_3F_2}(-\tfrac{r}{2},1-\tfrac{r+d}{2},\tfrac{d}{2};\tfrac{d}{2},2-\tfrac{r}{2};\tfrac{1}{t^2}).
$$
}%
Using the definition of generalized hypergeometric functions
{\small
$$
{_3F_2}(a_1,a_2,c;c,b;z)=\sum_{n=0}^\infty \frac{(a_1)_n(a_2)_n(c)_n}{(c)_n(b)_n n!}z^n=\sum_{n=0}^\infty \frac{(a_1)_n(a_2)_n}{(b)_n n!}z^n={_2F_1}(a_1,a_2;b;z)
$$
}%
with Pochhammer symbol $(z)_0=1$ and $(z)_n=(z+n-1)\,(z)_{n-1}$, we conclude
{\small
\begin{equation}
    I(t) 
    =\tfrac{
        \pi^{\frac{d}{2}}t^r\Gamma(2-\frac{r+d}{2})
        }{
        \Gamma(2-\frac{r}{2})}
    {_2F_1}(-\tfrac{r}{2},1-\tfrac{r+d}{2};2-\tfrac{r}{2};\tfrac{1}{t^2}).
\end{equation}
}%
Thus, the function $h$ can be rewritten as
{\small
\begin{align*}
&h(t)
= -A_1 
\tfrac{
    \pi^{\frac{d}{2}}t^r\Gamma(2-\frac{r+d}{2})
    }{
    \Gamma(2-\frac{r}{2})}
{_2F_1}(-\tfrac{r}{2},1-\tfrac{r+d}{2};2-\tfrac{r}{2};\tfrac{1}{t^2})+\tfrac{t^2}{2\tau}\\
&= 
-\tfrac{
    \Gamma(2-\frac{r}{2})
    }{
    \pi^{\frac{d}{2}} \Gamma(2-\frac{r+d}{2})
    }
\tfrac{
    \pi^{\frac{d}{2}}t^r\Gamma(2-\frac{r+d}{2})
    }{
    \Gamma(2-\frac{r}{2})
    }
{_2F_1}(-\tfrac{r}{2},1-\tfrac{r+d}{2};2-\tfrac{r}{2};\tfrac{1}{t^2})
+\tfrac{
    r\Gamma(2-\frac{r}2)\Gamma(\frac{d+r}{2})t^2
    }{
    2\Gamma(1+\frac{d}{2})
    }\\
&=  
-t^r{_2F_1}(-\tfrac{r}{2},1-\tfrac{r+d}{2};2-\tfrac{r}{2};\tfrac{1}{t^2})
+\tfrac{
    r\Gamma(2-\frac{r}2)\Gamma(\frac{d+r}{2})t^2
    }{
    2\Gamma(1+\frac{d}{2})
}.
\end{align*}
}%
In order to show that $h$ is increasing on $[1,\infty)$, we consider its derivative. 
It is given by
{\small
\begin{align*}
    h'(t)
    &=
    -rt^{r-1}{_2F_1}(-\tfrac{r}{2},1-\tfrac{r+d}{2};2-\tfrac{r}{2};\tfrac{1}{t^2})\\
    &\quad
    +t^{r-3}
    \tfrac{r(r+d-2)}{4-r}
    {_2F_1}(1-\tfrac{r}{2},2-\tfrac{r+d}{2};3-\tfrac{r}{2};\tfrac{1}{t^2})
    +\tfrac{r\Gamma(2-\frac{r}2)\Gamma(\frac{d+r}{2})t}{\Gamma(1+\frac{d}{2})}.
\end{align*}
}%
Since $r+d-2\geq 0$, we have by Lemma~\ref{lem:hyper_convex} that
{\small
\begin{align*}
    \tfrac{h'(t)}{rt^{r-1}}
    \geq
    g(t)
    &\coloneqq
    -{_2F_1}(-\tfrac{r}{2},1-\tfrac{r+d}{2};2-\tfrac{r}{2};\tfrac{1}{t^2})\\
    &\quad
    -(t^{-2}-t^{-4})
    \tfrac{r+d-2}{4-r}
    \tfrac{(2-r)(4-r-d)}{2(6-r)}
    {_2F_1}(2-\tfrac{r}{2},3-\tfrac{d+r}{2},4-\tfrac{r}{2};1)\\
    &\quad
    +t^{-2}
    \tfrac{r+d-2}{4-r}
    {_2F_1}(1-\tfrac{r}{2},2-\tfrac{r+d}{2};3-\tfrac{r}{2};1) 
    +\tfrac{\Gamma(2-\frac{r}2)\Gamma(\frac{d+r}{2})t^{2-r}}{\Gamma(1+\frac{d}{2})}.
\end{align*}
}%
Next, we show that $g$ is non-negative on $[1,\infty)$, which implies that $h'$ 
is non-negative such that $h$ is increasing on $[1,\infty)$.
Using Gauss's summation formula for hypergeometric functions \eqref{gs}, 
we can evaluate $g(1)$ using the identity $z\Gamma(z)=\Gamma(z+1)$ as
{\small
\begin{align}
    g(1)
    &=
    -\tfrac{
        \Gamma(2-\frac{r}{2})\Gamma(1+\frac{r+d}{2})
        }{
        \Gamma(2)\Gamma(1+\frac{d}{2})}
    -\tfrac{
        (2-r-d)\Gamma(3-\frac{r}{2})\Gamma(\frac{d+r}{2})
        }{
        (4-r)\Gamma(2)\Gamma(1+\frac{d}{2})} 
    +\tfrac{
        \Gamma(2-\frac{r}2)\Gamma(\frac{d+r}{2})
        }{
        \Gamma(1+\tfrac{d}{2})}\\
    &=
    -\tfrac{
        \Gamma(2-\frac{r}{2})\frac{r+d}{2}\Gamma(\frac{r+d}{2})
        }{
        \Gamma(1+\frac{d}{2})}
    -\tfrac{
        (2-r-d)(2-\frac{r}{2})\Gamma(2-\frac{r}{2})\Gamma(\frac{d+r}{2})
        }{
        (4-r)\Gamma(1+\frac{d}{2})}
    +\tfrac{
        \Gamma(2-\frac{r}2)\Gamma(\frac{d+r}{2})
        }{
        \Gamma(1+\frac{d}{2})}\\
    &=
    \tfrac{
        \Gamma(2-\frac{r}{2})\Gamma(\frac{r+d}{2})
        }{
        \Gamma(1+\frac{d}{2})}
    \bigl(-\tfrac{r+d+2-r-d}{2}+1\bigr)
    =0.
\end{align}
}%
Thus it suffices to show that $g$ is increasing. The derivative of $g$ is given by
{\small
\begin{align}
    g'(t)
    &=t^{-3}\tfrac{r(r+d-2)}{4-r}{_2F_1}(1-\tfrac{r}{2},2-\tfrac{r+d}{2};3-\tfrac{r}{2};\tfrac{1}{t^2})\\
    &\quad+(2t^{-3}-4t^{-5})\tfrac{r+d-2}{4-r}\tfrac{(2-r)(4-r-d)}{2(6-r)}{_2F_1}(2-\tfrac{r}{2},3-\tfrac{d+r}{2},4-\tfrac{r}{2};1)\\
    &\quad-2t^{-3}\tfrac{r+d-2}{4-r}
    {_2F_1}(1-\tfrac{r}{2},2-\tfrac{r+d}{2};3-\tfrac{r}{2};1)
    +\tfrac{
        (2-r)\Gamma(2-\frac{r}2)\Gamma(\frac{d+r}{2})t^{1-r}
        }{
        \Gamma(1+\frac{d}{2})}\\
    &\geq -(t^{-3}-t^{-5})r\tfrac{r+d-2}{4-r}\tfrac{(2-r)(4-r-d)}{2(6-r)}{_2F_1}(2-\tfrac{r}{2},3-\tfrac{d+r}{2},4-\tfrac{r}{2};1)\\
    &\quad+t^{-3}\tfrac{r(r+d-2)}{4-r}{_2F_1}(1-\tfrac{r}{2},2-\tfrac{r+d}{2};3-\tfrac{r}{2};1)\\
    &\quad+(2t^{-3}-4t^{-5})\tfrac{r+d-2}{4-r}\tfrac{(2-r)(4-r-d)}{2(6-r)}{_2F_1}(2-\tfrac{r}{2},3-\tfrac{d+r}{2},4-\tfrac{r}{2};1)\\
    &\quad-2t^{-3}\tfrac{r+d-2}{4-r}{_2F_1}(1-\tfrac{r}{2},2-\tfrac{r+d}{2};3-\tfrac{r}{2};1)
    +\tfrac{
        (2-r)\Gamma(2-\frac{r}2)\Gamma(\frac{d+r}{2})t^{1-r}
        }{
        \Gamma(1+\frac{d}{2})},
\end{align}
}%
where the inequality is again the application of Lemma~\ref{lem:hyper_convex}.
Reformulation yields
{\small
\begin{align}
    &g'(t) \geq -(2-r)t^{-3}\tfrac{r+d-2}{4-r}{_2F_1}(1-\tfrac{r}{2},2-\tfrac{r+d}{2};3-\tfrac{r}{2};1)\\
    &+(2t^{-3}-rt^{-3}-4t^{-5}+rt^{-5})\tfrac{r+d-2}{4-r}\tfrac{(2-r)(4-r-d)}{2(6-r)}{_2F_1}(2-\tfrac{r}{2},3-\tfrac{d+r}{2},4-\tfrac{r}{2};1)\\
    &+\tfrac{
        (2-r)\Gamma(2-\frac{r}2)\Gamma(\frac{d+r}{2})t^{1-r}
        }{
        \Gamma(1+\frac{d}{2})}.
\end{align}
}%
Applying Gauss's summation formula for hypergeometric functions and $z\Gamma(z)=\Gamma(z+1)$, the above formula becomes
{\small
\begin{align*}
    g'(t)
    &\ge 
    -(2-r)t^{-3}
    \tfrac{r+d-2}{4-r}
    \tfrac{
        \Gamma(3-\frac{r}{2})\Gamma(\frac{d+r}{2})
        }{
        \Gamma(1+\frac{d}{2})}
    +\tfrac{
        (2-r)
        \Gamma(2-\frac{r}2)\Gamma(\frac{d+r}{2})t^{1-r}
        }{
        \Gamma(1+\frac{d}{2})}\\
    &\quad
    +(2t^{-3}-rt^{-3}-4t^{-5}+rt^{-5})
    \tfrac{r+d-2}{4-r}
    \tfrac{(2-r)(4-r-d)}{2(6-r)}
    \tfrac{
        \Gamma(4-\frac{r}{2})
        \Gamma(\tfrac{d+r}{2}-1)
        }{
        \Gamma(1+\frac{d}{2})}\\
    &=
    -(2-r)t^{-3}
    \tfrac{
        (r+d-2)\Gamma(2-\frac{r}{2})\Gamma(\frac{d+r}{2})
        }{
        2\Gamma(1+\frac{d}{2})}
    +\tfrac{
        (2-r)
        \Gamma(2-\frac{r}2)\Gamma(\frac{d+r}{2})t^{1-r}
        }{
        \Gamma(1+\frac{d}{2})}\\
    &\quad
    +(2t^{-3}-rt^{-3}-4t^{-5}+rt^{-5})
    \tfrac{(2-r)(4-r-d)}{4}
    \tfrac{
        \Gamma(2-\frac{r}{2})\Gamma(\frac{d+r}{2})
        }{
        \Gamma(1+\frac{d}{2})}\\
    &=(2-r)
    \tfrac{
        \Gamma(2-\frac{r}{2})\Gamma(\frac{d+r}{2})
        }{
        \Gamma(1+\frac{d}{2})}
        \bigl(t^{1-r}
        +\bigl(\tfrac{(2-r)(4-r-d)}{4}-\tfrac{r+d-2}{2}\bigr)
        t^{-3}-\tfrac{(4-r)(4-r-d)}{4}t^{-5}\bigr)\\
    &\geq(2-r)
    \tfrac{
        \Gamma(2-\frac{r}{2})\Gamma(\frac{d+r}{2})
        }{
        \Gamma(1+\frac{d}{2})}
    \bigl(t^{1-r}
        +\bigl(\tfrac{(2-r)(4-r-d)}{4}-\tfrac{r+d-2}{2}\bigr)
        t^{-3}-\tfrac{(4-r)(4-r-d)}{4}t^{-3}\bigr)\\
    &=(2-r)
    \tfrac{
        \Gamma(2-\frac{r}{2})\Gamma(\frac{d+r}{2})
        }{
        \Gamma(1+\frac{d}{2})}
        \bigl(t^{1-r}-t^{-3}\bigr)
    \geq 0.
\end{align*}
}%
In summary, we have that $g(1)=0$ and that $g$ is increasing on $[1,\infty)$. 
Therefore, $g$ is non-negative on $[1,\infty)$ which yields that $f'$ is non-negative on $[1,\infty)$.
In particular $f$ is increasing on $[1,\infty)$ and we are done.
\hfill $\Box$

\medskip

\noindent
    {\bf Proof of Theorem~\ref{thm:HD-spec}(ii):} 
		We have to check that $\eta_\tau^*$
    fulfills the conditions in \eqref{eq:opt_cond}.
Since the constraint for $\eta^*$ in \eqref{eq:constraint_problem} is fulfilled
for $R=1$, it remains to show \eqref{eq:opt_cond} just for $R=1$
and the appropriate $\tau$, and then use Proposition \ref{prop:penalizer_vs_constraint}(i) to get $\eta_\tau^*$.
To this end, we consider the functions
{\small
\begin{align}
 h(x) &= -2F_1 \left(-\tfrac{r}{2}, \tfrac{2-r-d}{2}; \tfrac{d}{2};x^2 \right)
 + \tfrac{1}{2\tau} x^2,\\
 \tilde h(x) &= -2F_1 \left(-t\frac{r}{2}, \tfrac{2-r-d}{2}; \tfrac{d}{2};\tfrac{1}{x^2} \right) x^r
 + \tfrac{1}{2\tau} x^2.
\end{align}
}%
Clearly, the equality condition in \eqref{eq:opt_cond} is fulfilled
with the constant $C_\tau = h(1) = \tilde h(1)$.
We have to show that $h(x) \ge h(1)$ for $x \in [0,1]$ 
and
$\tilde h(x) \ge \tilde h(1)$ for $x \in [1,\infty)$.
With
{\small
\[
\tfrac{1}{\tau} = \tfrac{r(d+r-2)}{d} \,_2 F_1(\tfrac{2-r}2, -\tfrac{d+r-4}{2}; \tfrac{d+2}{2}; 1)
\]
}%
this is shown in  Proposition~\ref{lem:min_2F1_a}. 
This implies $\eta^* = \mathcal U_{\S^{d-1}}$ and it follows directly from Proposition~\ref{prop:int_dist_Sd} that
$\mathcal E_K(\eta^*)=-\tfrac12{_2F_1}\big(-\tfrac{r}{2},\tfrac{2-r-d}{2};\tfrac{d}{2};1\big)$.
\hfill $\Box$


\subsection{Proof of Theorem~\ref{thm:directional_derivatives_abs_cont}}
\label{proof:directional_derivatives_abs_cont}
We prove the more general statement that for any symmetric and locally Lipschitz continuous kernel $K\colon\R^d\times\R^d\to\R$ 
with Lipschitz constant 
\begin{equation}
  \Lip(F, A) \coloneqq \sup\Big\{ \tfrac{|F(x) - F(y)|}{\|x-y\|_2}: x, y \in A, \; x \not = y \Big\}.
\end{equation} 
fulfilling
\begin{equation}
\Lip(K, B_{r}(x) \times B_{s}(y)) \le L (1+ \|x\|_2 + \|y\|_2 + r + s), \qquad x, y \in \R^d, \quad r, s \ge 0,
\end{equation}
and any $\mu\in\P_2(\R^d)$ with $\mu(X)=0$, where $X\subseteq\R^d\times\R^d$ is the set, where $K$ is not differentiable,
it holds 
\begin{equation}\label{eq:to_show_directional_derivatives}
\HD_{\zb v}\mathcal E_K(\mu)=\D_{\zb v}\mathcal E_K(\mu) = \langle (\Id, \nabla G)_{\#} \mu, \zb v \rangle_{\mu}, \qquad
\zb v \in \zb \T_{\mu}\P_2(\R^d).
\end{equation}
Since the Riesz kernel with $r\in[1,2)$ fulfills these properties, this implies by the equality condition of the Cauchy--Schwarz inequality from Lemma~\ref{lem:1} (iii) that
  \begin{equation}
    \argmin_{\zb v\in\zb\T_\mu\P_2(\R^d),\|\zb v\|_\mu=1}\HD_{\zb v}\mathcal E_K(\mu)=\Big\{-\|(\mathrm{Id},\nabla G)_\#\mu\|_\mu^{-1}\cdot(\mathrm{Id},\nabla G)_\#\mu\Big\}
\end{equation}
such that
\begin{equation}
\HD_-\mathcal E_K(\mu)=\{ (-1) \cdot (\mathrm{Id},\nabla G)_\#\mu\} = \{  (\mathrm{Id},-\nabla G)_\#\mu\}.
\end{equation}
Observe that the Riesz kernel is everywhere differentiable for $r\in(1,2)$ and that it is not differentiable exactly at $\{(x,x):x\in\R^d\}$ for $r=1$. 
Moreover, by Fubini's Theorem the assertion $\mu(\{x\})=0$ for all $x\in\R^d$ leads to
\begin{align}
\mu\otimes\mu(\{(x,x):x\in\R^d\})&=\int_{\R^d}\int_{\R^d}1_{\{(x,x):x\in\R^d\}}(y,z)\d\mu(y)\d\mu(z)\\&=\int_{\R^d}\int_{\R^d}1_{\{z\}}(y)\d\mu(y)\d\mu(z)=\int_{\R^d}\mu(\{z\})\d\mu(z)=0,
\end{align}
which proves Theorem~\ref{thm:directional_derivatives_abs_cont}.
Thus, it remains to show \eqref{eq:to_show_directional_derivatives}.

Since $K$ is locally Lipschitz, it holds for a.e. $(x,y) \in \R^{2d}$ that $K$ is differentiable, hence
\begin{equation}
\lim_{t \to 0+}  \frac{K(x+tv, y+tw) - K(x,y)}{t} = \nabla K(x,y)^{\tT}(v, w), \qquad (v,w) \in \R^{2d}.
\end{equation}
We can apply the dominated convergence theorem, since the right hand side in
\begin{equation}\label{eq:abs_integrability}
\left|  \frac{K(x+tv, y+tw) - K(x,y)}{t}  \right| \le \Lip(K, B_{\|v\|_2}(x) \times B_{\|w\|_2}(y))  \|(v,w)\|_2
\end{equation}
is absolutely integrable.
Hence, for any measure $\mu$ with $\mu\otimes \mu(X)=0$ and direction $\zb v \in \zb \T_{\mu}\P_2(\R^d)$, we obtain by symmetry of $K$ and Fubini's theorem that
\begin{equation}
\begin{aligned}
& \lim_{t \to 0+} \frac{\mathcal E_{K}(\gamma_{\zb v}(t))- \mathcal E_K(\mu)}{t}
= \lim_{t \to 0+} \frac{\mathcal E_{K}(\gamma_{\zb t \cdot v}(1))- \mathcal E_K(\mu)}{t}
\\
&=  \lim_{t\to 0+} \frac12 \int_{\R^{2d}}\int_{\R^{2d}}\frac{K(x+tv, y+tw) - K(x,y)}{t} \, \d \zb v(x,v) \d \zb v(y,w)\\
&=  \frac12 \int_{\R^{2d}}\int_{\R^{2d}} \nabla K(x,y)^{\tT} (v,w) \, \d \zb v(x,v) \d \zb v(y,w)\\
&=  \frac12 \int_{\R^{2d}}\int_{\R^{2d}} \nabla_1 K(x,y)^\tT v + \nabla_2 K(x,y)^{\tT} w \, \d \zb v(x,v) \d \zb v(y,w)\\
&=  \int_{\R^{2d}} v^{\tT} \left( \int_{\R^d} \nabla_1 K(x,y) \d \mu(y) \right) \, \d \zb v(x,v).
\end{aligned}
\end{equation}
Using again the dominating convergence theorem with \eqref{eq:abs_integrability} we note that
\begin{equation}
\int_{\R^d} \nabla_1 K(x,y) \d \mu(y) = \nabla \int_{\R^d} K(x,y) \d \mu(y), \qquad \mu-\text{a.e.}\quad x \in \R^d,
\end{equation}
so that 
{\small
\begin{align}
\lim_{t \to 0+} \frac{\mathcal E_{K}(\gamma_{\zb v}(t))- \mathcal E_K(\mu)}{t}
&=
\int_{\R^{2d}} v^{\tT} \nabla G(x) \, \d \zb v(x,v)
=
\int_{\R^{3d} } x_2^\tT x_3 \, \d \alpha(x_1,x_2,x_3)
\end{align}
}%
with $\zb \alpha = (\pi_1,\pi_2, \nabla G \circ \pi_1)_\# \zb v$.
Since $\zb \alpha$ is the unique plan with 
$(\pi_{1,2})_\# \alpha = \zb v$ and $(\pi_{1,3})_\# \alpha = (\Id,\nabla G)_\# \mu$
and we arrive at \eqref{eq:to_show_directional_derivatives}
by definition \eqref{inner_mu}. \hfill$\Box$

\subsection{Proof of Theorem~\ref{prop:EK_steepest_descent_flow}} \label{app:desc_flow}
The curve of interest \eqref{eq:EK_steepest_descent_flow}
is of the form $\gamma_{c_{2-r}}$, where 
$c_\tau \coloneqq \left(-\tau r \mathcal E_K(\eta^*) \right)^{\frac{1}{2-r}}$
and 
\begin{equation*}
  \gamma_c(t) \coloneqq \gamma_{\delta_0 \otimes \eta^*}(c t^\frac{1}{2-r}) = (c t^\frac{1}{2-r} \mathrm{Id})_{\#} \eta^*, \quad c >0.
\end{equation*}
In order to apply Theorem~\ref{thm:directional_derivatives_abs_cont}, we compute the gradient of the function 
  \begin{equation}
    G_{\gamma_{c_\tau}(t)}(x) \coloneqq \int_{\R^d} K(x,y) \d \gamma_{c_{\tau}}(t)(y) .
  \end{equation}
Here, we distinguish two cases.
First, in case $d+r<4$, 
Theorem~\ref{thm:HD-spec_0} yields that $\eta^*$ is absolute continuous.
Moreover, we know
from Proposition~\ref{prop:penalized_problem} and Proposition~\ref{prop:penalizer_vs_constraint} that $\eta_{\tau}^*=(c_{\tau} \mathrm{Id})_{\#} \eta^*$
satisfies the optimality conditions \eqref{eq:opt_cond}, in particular,
\begin{equation*}
  \int_{\R^d} K(x,y) \, \d(c_{\tau} \mathrm{Id})_\# \eta^*(y) 
  =
  \int_{\R^d} K(x,y) \, \d \eta_\tau^* (y) 
  = 
  C_{K,\tau} - \frac{1}{2\tau} \|x\|_2^2 
\end{equation*}
for all $\quad x \in \mathrm{supp} \, \eta_\tau^*$.
Then we can compute 
\begin{equation*}
  \begin{aligned}
    G_{\gamma_{c_\tau}(t)}(x) 
    & = 
    - \int_{\R^d} \|x - y\|_2^r \, \d (c_{\tau} t^{\frac{1}{2-r}} \mathrm{Id})_{\#} \eta^*(y) 
    \\
    &= t^{\frac{r}{2-r}} \int_{\R^d} K(x t^{-\frac{1}{2-r}}, y) \, \d (c_{\tau} \mathrm{Id})_{\#} \eta^*(y) \\
    &= 
    t^{\frac{r}{2-r}} \biggl(C_{K, \tau} -\frac{t^{\frac{-2}{2-r}}}{2\tau} \|x\|_2^2 \biggr) 
    \\
    &= 
    t^{\frac{r}{2-r}} C_{K, \tau} - \frac{1}{2\tau t} \|x\|_2^2, 
    \quad x \in \mathrm{supp}(\gamma_{c_{\tau}}(t)).
  \end{aligned}
\end{equation*}
    In the interior of $\mathrm{supp}(\gamma_{c_\tau}(t))$, we thus have 
    \begin{equation*}
    \nabla G_{\gamma_{c_\tau}(t)}(x) = - \frac{x}{\tau t}, \quad t>0,
  \end{equation*}
    which holds also true $\gamma_{c_\tau}(t)$-a.e.\  since $\eta^*$ is absolutely continuous.
        
    Second, if $d+r\geq 4$, we have $\eta^*=\mathcal U_{\mathbb{S}^{d-1}}$ by Theorem~\ref{thm:HD-spec_0}.
    Proposition \ref{prop:int_dist_Sd} for $R=c_\tau t^{\tfrac{1}{2-r}}$ and $x\in\R^d$ with $\|x\|_2\leq R$ implies
    \begin{align}\label{eq:rep_G}
    G_{\gamma_{c_\tau}(t)}(x)=-\int_{R\mathbb{S}^{d-1}}\|x-y\|_2^r\, \d \mathcal U_{R\mathbb{S}^{d-1}}(y)
    =-R^r {_2F_1}\big(-\tfrac{r}{2},\tfrac{2-r-d}{2};\tfrac{d}{2};\tfrac{\|x\|_2^2}{R^2}\big).
    \end{align}
    Using the derivative rule for hypergeometric functions and the chain rule, we have 
    \begin{equation}
    \nabla G_{\gamma_{c_\tau}(t)}(x)=-R^{r-2} \tfrac{r(d+r-2)}{d} {_2F_1}\big(\tfrac{2-r}{2},\tfrac{4-r-d}{2};\tfrac{d+2}{2};\tfrac{\|x\|_2^2}{R^2}\big) x, \quad \|x\| < R.
  \end{equation}
Since $K(x,y)$ is continuously differentiable whenever $x\neq y$, we conclude that $G_{\gamma_{c_\tau}(t)}$ is continuously differentiable on $R\mathbb{S}^{d-1}$ such that the above formula is also true for $\|x\|_2=R$.
Inserting $R$, we obtain 
  \begin{align}
    \nabla G_{\gamma_{c_\tau}(t)}(x)
    &=
      -\frac{c_\tau^{r-2}}{t} \,
      \frac{r(d+r-2)}{d} \,
      {_2F_1}\big(\tfrac{2-r}{2},\tfrac{4-r-d}{2};\tfrac{d+2}{2};1\big) \,
      x
    \\
    &=
      \frac{(d+r-2)}{d} \,
      \frac{{_2F_1}\big(\tfrac{2-r}{2},\tfrac{4-r-d}{2};\tfrac{d+2}{2};1\big)}{\mathcal E_K(\eta^*)} \,
      \frac{x}{\tau t}
  \end{align}
for $x\in\mathrm{supp}(\gamma_{c_\tau}(t))=R\mathbb{S}^{d-1}$.
On the basis of Proposition~\ref{prop:int_dist_Sd}, 
we have $\mathcal E_K(\eta^*)=-\tfrac12{_2F_1}\big(-\tfrac{r}{2},\tfrac{2-r-d}{2};\tfrac{d}{2};1\big)$ so that
  \begin{equation}
    \frac{(d+r-2)}{d} \,
    \frac{{_2F_1}\big(\tfrac{2-r}{2},\tfrac{4-r-d}{2};\tfrac{d+2}{2};1\big)}{\mathcal E_K(\eta^*)}=-1
  \end{equation}
by Lemma~\ref{lem:hyper_decreasing}(i).
Inserting this in the previous equation, we obtain also in this case that $\nabla G_{\gamma_{c_\tau}(t)}(x)=-\tfrac{x}{\tau t}$.
Setting $s \coloneqq  (2-r)^ {-1}$, and using Lemma~\ref{lem:chain_rule} and \ref{lem:tan-geo},
we deduce 
  \begin{align}
    \label{eq:gamma_c}
    \dot \gamma_c(t)
    & = c s t^{s-1} \cdot \dot \gamma_{\delta_0 \otimes \eta^*}(c t^{s})
    \\
    &=  c s t^{s-1} \cdot (\pi_1 + c t^{s} \pi_2, \pi_2)_{\#} (\delta_0 \otimes \eta^*) \\
    & =  (\pi_1, c s t^{s-1} \pi_2)_{\#} (\pi_1 + c t^s \pi_2, \pi_2)_{\#} 
      (\delta_0 \otimes \eta^*)
    \\
    &=  (\pi_1 + c t^{s} \pi_2, c s t^{s-1} \pi_2)_{\#} (\delta_0 \otimes \eta^*) \\
    & =  (c t^s \Id , c s t^{s-1} \Id )_{\#} \eta^*
      =  (\Id , s t^{-1}\Id )_{\#} (c t^s \Id )_{\#} \eta^*  \\
    & =  (\Id ,  s^{-1}t^{-1}\Id )_{\#} \gamma_c(t). \label{fin}
  \end{align}
Inserting the computed gradient for $t>0$, we get
  \begin{equation*}
    \dot \gamma_{c_{2-r}}(t) = (\Id , (2-r)^{-1}t^{-1} \Id )_{\#} \gamma_{c_{2-r}}(t) = 
    (\Id , - \nabla G_{\gamma_{c_{2-r}}(t)})_{\#} \gamma_{c_{2-r}}(t),
  \end{equation*}
which is the unique element of $\HD_- \mathcal E(\gamma_{c_{2-r}}(t))$ by Theorem~\ref{thm:directional_derivatives_abs_cont}.

Finally, we consider the case $t=0$.
For $t \to 0$ in \eqref{eq:gamma_c}, we obtain
\begin{equation*}
  \dot \gamma_{c_{2-r}}(0) 
  = 
  \begin{cases}
    (-\mathcal E_K(\eta^*)) \cdot \delta_0 \otimes \eta^*, & r = 1,\\
    0 \cdot \delta_0 \otimes \eta^*, &  r \in (1,2),
  \end{cases}
\end{equation*}
which is by Theorem~\ref{the:EKdirDF} exactly the direction of steepest descent at $\delta_0$. This concludes the proof.
\hfill $\Box$

\section{Proof of Theorem~\ref{thm:DKdir}}
\label{proof:discr_2}
{\bf Part (i)}
We note that the discrepancy may be written as
$\mathcal D_K ^2(\mu, \nu) = \mathcal E_{\tilde K}(\mu)$
with
\begin{equation*}
  \tilde K(x_1,x_2) = K(x_1,x_2) + V_{K,\nu}(x_1) + V_{K,\nu}(x_2) + \mathcal E_K(\nu), \qquad x_1,x_2 \in \R^d.
\end{equation*}
For $r\in(1,2)$ we have that the kernel $\tilde K$ is by definition differentiable.
In the case $r=1$, the first term is not differentiable for $(x_1,x_2)\not\in X_1\coloneqq\{(x,x):x\in\R^d\}$, the second term is not differentiable for $(x_1,x_2) \not \in X_2\coloneqq \{(q,x)\in \R^d\times\R^d: \nu(\{q\}) \ne 0\}$ and the third term is not differentiable when $(x_1,x_2) \not \in X_3\coloneqq \{(x,q)\in \R^d\times\R^d: \nu(\{q\}) \ne 0\}$.
In the following, we prove that $X_1$, $X_2$ and $X_3$ are zero-sets under $\mu\otimes\mu$.
Then, the statement follows analogously to the proof of Theorem~\ref{thm:directional_derivatives_abs_cont} in \ref{proof:directional_derivatives_abs_cont}.

As in the proof of Theorem~\ref{thm:directional_derivatives_abs_cont}, we have $\mu\otimes\mu(X_1)=0$ since $\mu(\{x\})=0$ for all $x\in\R^d$.
Moreover, we have  $X_2=\{q\in\R^d:\nu(\{q\}) \ne 0\}\times \R^d$ such that 
  \begin{equation}
\mu\otimes\mu(X_2)=\mu(\{q\in\R^d:\nu(\{q\}) \ne 0\})=\sum_{q\in\R^d\text{ with }\nu(\{q\}) \ne 0} \mu(\{q\})=0,
\end{equation}
where we used that $\{q\in\R^d:\nu(\{q\}) \ne 0\}$ is countable as any probability measure has only countable many points with positive mass.
Finally, $\mu\otimes\mu(X_3)=0$ follows analogously.

{\bf Part (ii)} The discrepancy functional is locally Lipschitz. Therefore, to compute the Hadamard derivative, 
we can exploit Proposition~\ref{prop:dini-hard}
and the decomposition \eqref{eq:dis-decomp}.
As in the proof of Theorem~\ref{thm:directional_derivatives_abs_cont},
the function $V_{K,\nu}$ is differentiable in $p$ if $\nu(\{p\})=0$.
In view of Lebesgue's dominated convergence theorem,
the Dini derivative of the interaction energy 
in direction $\zb v \coloneqq \delta_p \otimes \eta$, $\eta \in \P_2(\R^d)$,
is thus given by
\begin{align}
    \D_{\zb v} \mathcal V_{K,\nu}(\delta_p)
    &=
    \lim_{t \to 0+}
    \frac{\mathcal V_{K,\nu}(\gamma_{\zb v}(t)) - \mathcal V_{K,\nu}(\delta_p)}{t}
    \\
    &=
    \lim_{t \to 0+}
    \frac{1}{t} \,
    \int_{\R^d \times \R^d}
    V_{K,\nu}(x_1 + t x_2) - V_{K,\nu}(x_1) \, \d \zb v(x_1, x_2)
    \\
    &=
    \int_{\R^d}
    \langle \nabla V_{K,\nu}(p), x_2 \rangle \, \d \eta(x_2)
    = \langle \nabla V_{K,\nu}(p), v_\eta \rangle,
\end{align}
where $v_\eta \coloneqq \int_{\R^d} x \, \d \eta(x)$.
The steepest descent directions $\HD_{-} \F_{\nu}(\mu)$ are now given 
by the tangents $(\HD_{\delta_p \otimes  \hat \eta}^{-} \F_\nu (\delta_p))^{-} \cdot (\delta_p \otimes \hat \eta)$,
where $\hat \eta$ solves
\begin{equation}    \label{eq:disc_HD_problem}
    \min_{\eta \in \P_2(\R^d)}
    \HD_{\delta_p \otimes \eta}^{-} \mathcal E_K(\delta_p)
    + \nabla V_{K,\nu}(p)^{\tT} v_\eta 
    \quad \text{s.t.} \quad 
    \int_{\R^d} \|x\|_2^2 \, \d \eta(x) = 1.
\end{equation}
We want to bring the problem into an equivalent form, where the minimizer can be easier computed.
First we have for $\eta \in 
S_1 \coloneqq \{\eta \in \P_2(\R^d): \int_{\R^d} \|x\|_2^2 \, \d \eta(x) = 1\}
$
and $v_\eta$ as above
that
\begin{align}
  \int_{\R^d} \|x\|_2^2 \, \d (\Id - v_\eta)_\# \eta(x)
  &=
   \int_{\R^d} \|x\|_2^2 \, \d \eta(x)
  - 2 \, v_\eta^\tT \int_{\R^d} x \, \d \eta(x)
  + \|v_\eta\|_2^2 
  \\
  &=
 \int_{\R^d} \|x\|_2^2 \d \eta(x) - \|v_\eta \|_2^2  = 1 - \|v_\eta\|^2. \label{haha}
\end{align}
In particular, 
we have $\|v_\eta\|_2^2 \le 1$ with equality if and only if $\eta = \delta_{v_\eta}$.
Now, we show that the set
$S_1$ coincides with the set
\begin{align*}
S_2 &\coloneqq \{(\Id + v)_\# \big((1-\|v\|^2)^\frac12 \Id \big)_\# \tilde \eta:
(\tilde \eta,v) \in \P_2(\R^d)\times \R^d, \\
&\quad \int_{\R^d} x \, \d \tilde\eta(x) = 0, 
\,
 \int_{\R^d} \|x\|_2^2 \, \d \tilde\eta (x) = 1,
\, 
 \|v\| \le 1\}.
\end{align*}
Straightforward computation shows that
$S_2 \subseteq S_1$. 
For the other direction, let $\eta \in S_1$ and $v_\eta$ as above.
In the case that $\|v_\eta\|_2^2<1$ we consider $\tilde \eta$ defined by
\begin{equation}\label{eq:eta_tilde_eta}
    \eta =     (\Id + v_\eta)_\# 
    ((1-\|v_\eta\|_2^2)^{\frac{1}{2}} \Id)_\#.
    \tilde\eta
\end{equation}
Then we obtain by \eqref{haha} that
$  \int_{\R^d} \|x\|_2^2 \, \d \tilde \eta(x)
  = 1$
and
\begin{equation*}
  \int_{\R^d}x\,\d \tilde \eta(x)=(1-\|v_\eta\|_2^2)^{-\frac12}\int_{\R^d}x-v_\eta \, \d \eta(x)=0.
\end{equation*}
Thus, $\tilde\eta$ fulfills the constraints from $S_2$.
For $\|v_\eta\|_2^2=1$, we have by definition that $\tilde\eta=\mathcal U_{\mathbb{S}^{d-1}}$ fulfills the constraints from $S_2$ and formula \eqref{eq:eta_tilde_eta} such that we obtain $S_1=S_2$.

For $r=1$, we have by Theorem~\ref{the:EKdirDF}
that 
$\HD_{\delta_p \otimes \eta}^{-} \mathcal E_K(\delta_p) = \mathcal E_K(\eta)$ so that
by the translational invariance of $\mathcal E_K$ and since $S_1 = S_2$
problem~\eqref{eq:disc_HD_problem} is equivalent to
\begin{equation}
    \label{eq:sub-prob}
    \begin{aligned}
          & 
          \min_{\tilde \eta \in \P_2(\R^d)}\min_{v \in \R^d}
          (1 - \|v\|_2^2)^{\frac{1}{2}} \,
          \mathcal E_K(\tilde \eta)
          + \nabla V_{K,\nu}(p)^{\tT} v,  \\
          &\;\text{s.t.} \quad
          \int_{\R^d}\|x\|_2^2 \, \d \tilde \eta(x) = 1, 
          \quad \int_{\R^d} x \, \d \tilde\eta(x) = 0, 
          \quad \|v\|_2^2 \le 1.
    \end{aligned}
\end{equation}
Applying Cauchy--Schwarz's inequality,
we estimate the objective function by
\begin{align}
    (1 -\|v\|_2^2)^{\frac{1}{2}} \, \mathcal E_K(\tilde \eta) 
    +  \nabla V_{K,\nu}(p)^{\tT} v 
    &\ge  (1 -\|v\|_2^2)^{\frac{1}{2}} \, \mathcal E_K(\tilde \eta) 
    - \|\nabla V_{K,\nu}(p)\|_2 \|v\|_2
     \\
     &= \Big\langle 
     \begin{pmatrix}
     (1 -\|v\|_2^2)^{\frac{1}{2}}\\\|v\|_2
     \end{pmatrix},
     \begin{pmatrix}
     \mathcal E_K(\tilde \eta)\\
     -\|\nabla V_{K,\nu}(p)\|_2
     \end{pmatrix}
     \Big\rangle
\end{align}
with equality if and only if $v = a \, \nabla V_{K,\nu}(p)$ for some $a < 0$.
Applying Cauchy--Schwarz's inequality once more, 
we obtain
\begin{equation*}
    (1 -\|v\|_2^2)^{\frac{1}{2}} \, \mathcal E_K(\tilde \eta) 
    +  \nabla V_{K,\nu}(p)^{\tT} v
    \ge - \sqrt{\mathcal E_K(\tilde\eta)^2 + \|\nabla V_{K,\nu}(p) \|_2^2}
\end{equation*}
with equality if and only if
\begin{equation*}
  \bigl( (1-\|v\|_2^2)^{\frac{1}{2}},  \|v\|_2 \bigr)
  = b \, ( -\mathcal E_K(\tilde\eta), \|\nabla V_{K,\nu}(p)\|_2 )
\end{equation*}
for some $b > 0$.
Since the norm of the left-hand side is one, 
and due to the equality within the second component,
equality can only hold if
\begin{equation*}
 b = \left(\mathcal E_K(\tilde \eta)^2 + \|\nabla V_{K,\nu}(p)\|_2^2\right)^{-\frac12} = -a.
\end{equation*}
Hence, for any fixed measure $\tilde\eta$, the vector
\begin{equation*}
 \hat v = -(\mathcal E_K(\tilde\eta)^2 + \|\nabla V_{K,\nu}(p)\|_2^2)^{-\frac12} \nabla V_{K,\nu}(p)
\end{equation*}
minimizes the objective in \eqref{eq:sub-prob} ,
which then simplifies to
\begin{align}
  &\min_{\tilde\eta \in \P_2(\R^d)} - \left(\mathcal E_K(\tilde\eta)^2 + \|\nabla V_{K,\nu}(p)\|_2^2\right)^{\frac12}
  \\
  &\quad\text{s.t}\quad
  \int_{\R^d}\|x\|_2^2 \, \d \tilde\eta(x) = 1,
  \;
  \int_{\R^d} x \, \d \tilde\eta(x) = 0.
\label{eq:disc_HD_problem_b}
\end{align}
Due to the non-positiveness $\mathcal E_K$,
problem \eqref{eq:disc_HD_problem_b} is equivalent to \eqref{eq:constraint_problem}
up to the additional condition $\int_{\R^d} x \, \d \tilde \eta(x) = 0$.
However, since by Proposition~\ref{prop:penalized_problem} every solution of \eqref{eq:constraint_problem} is orthogonally invariant,
the solutions $\eta^*$ of both problems coincide.
Hence,
any solution of \eqref{eq:disc_HD_problem} can be represented as
\begin{equation*}
    \hat \eta
    = 
    \bigl(\Id - b^* \, \nabla V_{K,\nu}(p) \bigr)_{\#} 
    \bigl( - b^* \mathcal E_K(\eta^*) \Id \bigr)_{\#}
    \eta^*
    \bigl( b^* \bigl( - \mathcal E_K( \eta^*) \Id - \nabla V_{K,\nu}(p) \bigr) \bigr)_{\#} \eta^*
\end{equation*}
and
\begin{equation*}
 b^{*} =  (\mathcal E_K(\eta^*)^2 + \|\nabla V_{K,\nu}(p)\|_2^2)^{-\frac12}.
\end{equation*}
Since the minimum of \eqref{eq:disc_HD_problem} is
$-(b^*)^{-1} = \HD_{\delta_p \otimes \eta^*}^- \F_\nu (\delta_p)$,
we get the assertion.

For $r \in (1,2)$,
we again apply Theorem~\ref{the:EKdirDF} to conclude 
that \eqref{eq:disc_HD_problem} is equivalent to
\begin{equation*}
    \min_{v \in \R^d} \nabla V_{K,\nu}(p)^\tT v
    \quad\text{s.t.}\quad 
    \|v\|_2^2 \le 1.
\end{equation*}
Here the minimizer is given by $\hat v = -\nabla V_{K,\nu}(p) / \|\nabla V_{K,\nu}(p)\|_2$
such that it holds $\hat \eta = \delta_{-\nabla V_{K,\nu}(p) / \|\nabla V_{K,\nu}(p)\|_2}$,
and the minimum is given by $\HD_{\delta_p \otimes \hat \eta}^- \F_\nu (\delta_p) = -\|\nabla V_{K,\nu}(p)\|_2$,
which yields the assertion.
\hfill $\Box$
\medskip
	
	\noindent	
 {\bf Acknowledgements.} 
 \small{Funding by German Research Foundation (DFG) within the project STE 571/16-1, 
 by the DFG excellence cluster MATH+ and by the 
 BMBF project “VI-Screen” (13N15754) 
 are gratefully acknowledged.}

{\small
\bibliographystyle{abbrv} 
\bibliography{references}}

\begin{thebibliography}{10}

\bibitem{AHS2023}
F.~Altekrüger, J.~Hertrich, and G.~Steidl.
\newblock Neural {W}asserstein gradient flows for maximum mean discrepancies
  with {R}iesz kernels.
\newblock In A.~Krause, E.~Brunskill, K.~Cho, B.~Engelhardt, S.~Sabato, and
  J.~Scarlett, editors, {\em Proceedings of the 40th International Conference
  on Machine Learning}, volume 202 of {\em Proceedings of Machine Learning
  Research}, pages 664--690. PMLR, 2023.

\bibitem{BookAmGiSa05}
L.~Ambrosio, N.~Gigli, and G.~Savare.
\newblock {\em Gradient Flows}.
\newblock Lectures in Mathematics ETH Zürich. Birkh\"auser, Basel, 2005.

\bibitem{ArKoSaGr19}
M.~Arbel, A.~Korba, A.~Salim, and A.~Gretton.
\newblock Maximum mean discrepancy gradient flow.
\newblock In H.~Wallach, H.~Larochelle, A.~Beygelzimer, F.~d~Alch\'{e}-Buc,
  E.~Fox, and R.~Garnett, editors, {\em Advances in Neural Information
  Processing Systems}, volume~32, pages 1--11, New York, USA, 2019. Curran
  Associates Inc.

\bibitem{AH2012}
K.~Atkinson and W.~Han.
\newblock {\em Spherical harmonics and approximations on the unit sphere: an
  introduction}, volume 2044 of {\em Lecture Notes in Mathematics}.
\newblock Springer, Heidelberg, 2012.

\bibitem{BaCaLaRa13}
D.~Balagué, J.~A. Carrillo, T.~Laurent, and G.~Raoul.
\newblock Dimensionality of local minimizers of the interaction energy.
\newblock {\em Archive for Rational Mechanics and Analysis}, 209:1055--1088,
  2013.

\bibitem{Ba53}
H.~Bateman.
\newblock {\em Higher transcendental functions}, volume~1.
\newblock McGraw-Hill Book Company, 1953.

\bibitem{BaHu01}
B.~J. Baxter and S.~Hubbert.
\newblock Radial basis functions for the sphere.
\newblock In {\em Recent progress in multivariate approximation}, pages 33--47.
  Springer, 2001.

\bibitem{ThesisBo11}
G.~Bonaschi.
\newblock {\em Gradient ﬂows driven by a non-smooth repulsive interaction
  potential}.
\newblock Master's thesis. University of Pavia, 2011.

\bibitem{BoCaFrPe15}
G.~A. Bonaschi, J.~A. Carrillo, M.~D. Francesco, and M.~A. Peletier.
\newblock Equivalence of gradient flows and entropy solutions for singular
  nonlocal interaction equations in 1d.
\newblock {\em ESAIM Control Optimization and Calculus of Variation},
  21:414--441, 2015.

\bibitem{BookBoHaSa19}
S.~V. Borodachov, D.~P. Hardin, and E.~B. Saff.
\newblock {\em Discrete Energy on Rectifiable Sets}.
\newblock Springer Monographs in Mathematics. Springer, New York, 2019.

\bibitem{Brenier1987}
Y.~Brenier.
\newblock D\'ecomposition polaire et r\'earrangement monotone des champs de
  vecteurs.
\newblock {\em Comptes Rendus de l'Acad\'emie des Sciences Paris Series I
  Mathematics}, 305(19):805--808, 1987.

\bibitem{CDEFS2020}
J.~A. Carrillo, M.~Di~Francesco, A.~Esposito, S.~Fagioli, and M.~Schmidtchen.
\newblock Measure solutions to a system of continuity equations driven by
  newtonian nonlocal interactions.
\newblock {\em Discrete and Continuous Dynamical Systems}, 40(2):1191--1231,
  2020.

\bibitem{CaHu17}
J.~A. Carrillo and Y.~Huang.
\newblock Explicit equilibrium solutions for the aggregation equation with
  power-law potentials.
\newblock {\em Kinetic and Related Models}, 10(1):171--192, 2017.

\bibitem{CaMaMoRoScVe21}
J.~A. Carrilo, J.~Mateu, M.~G. Mora, L.~Rondi, L.~Scardia, and J.~Verdera.
\newblock The equilibrium measure for an anisotropic nonlocal energy.
\newblock {\em Calculus of Variations and Partial Differential Equations},
  60:109, 2021.

\bibitem{CaSaWo22}
D.~Chafaï, E.~B. Saff, and R.~S. Womersley.
\newblock On the solution of a {R}iesz equilibrium problem and integral
  identities for special functions.
\newblock {\em Journal of Mathematical Analysis and Applications}, 515:126367,
  2022.

\bibitem{ChSaWo22b}
D.~Chafaï, E.~B. Saff, and R.~S. Womersley.
\newblock Threshold condensation to singular support for a {R}iesz equilibrium
  problem.
\newblock {\em arXiv:2206.04956v1}, 2022.

\bibitem{BookDeRu00}
V.~Demyanov and A.~Rubinov, editors.
\newblock {\em Quasidifferentiability and Related Topics}.
\newblock Nonconvex Optimization and Its Applications. Springer New York, NY,
  2000.

\bibitem{EGNS2021}
M.~Ehler, M.~Gr\"af, S.~Neumayer, and G.~Steidl.
\newblock Curve based approximation of measures on manifolds by discrepancy
  minimization.
\newblock {\em Foundations of Computational Mathematics}, 21(6):1595--1642,
  2021.

\bibitem{HTF1953}
A.~Erd\'elyi, A.~Magnus, W.~Oberhettinger, and F.~Tricomi.
\newblock {\em Higher Transcendental Functions, Vol I}.
\newblock McGraw-Hill Book Company, 1953.

\bibitem{FHS2013}
M.~Fornasier, J.~Haskovec, and G.~Steidl.
\newblock Consistency of variational continuous-domain quantization via kinetic
  theory.
\newblock {\em Applicable Analysis}, 92(6):1283--1298, 2013.

\bibitem{GHLS2019}
A.~Garbuno-Inigo, F.~Hoffmann, W.~Li, and A.~M. Stuart.
\newblock Interacting langevin diffusions: Gradient structure and ensemble
  {K}alman sampler.
\newblock {\em arXiv:1903.08866v3}, 2019.

\bibitem{Gi04}
N.~Gigli.
\newblock {\em On the geometry of the space of probability measures in Rn
  endowed with the quadratic optimal transport distance}.
\newblock Phd Thesis. Scuola Normale Superiore di Pisa, 2004.

\bibitem{Gi93}
E.~D. Giorgi.
\newblock New problems on minimizing movements.
\newblock In P.~Ciarlet and J.-L. Lions, editors, {\em Boundary Value Problems
  for Partial Differential Equations and Applications}, pages 81--98. Masson,
  1993.

\bibitem{GPS2012}
M.~Gr\"af, M.~Potts, and G.~Steidl.
\newblock Quadrature errors, discrepancies and their relations to halftoning on
  the torus and the sphere.
\newblock {\em SIAM Journal on Scientific Computing}, 34(5):2760--2791, 2012.

\bibitem{GCO2021}
T.~S. Gutleb, J.~A. Carrillo, and S.~Olver.
\newblock Computation of power law equilibrium measures on balls of arbitrary
  dimension.
\newblock {\em arXiv:2109.00843v1}, 2021.

\bibitem{HHS2021}
P.~Hagemann, J.~Hertrich, and G.~Steidl.
\newblock {\em Generalized normalizing flows via {M}arkov chains}.
\newblock Cambridge University Press, 2023.

\bibitem{HBGS2023}
J.~Hertrich, R.~Beinert, M.~Gr{\"a}f, and G.~Steidl.
\newblock Wasserstein gradient flows of the discrepancy with distance kernel on
  the line.
\newblock In L.~Calatroni, M.~Donatelli, S.~Morigi, M.~Prato, and
  M.~Santacesaria, editors, {\em Scale Space and Variational Methods in
  Computer Vision}, pages 431--443, Cham, 2023. Springer.

\bibitem{HWAH2023}
J.~Hertrich, C.~Wald, F.~Altekrüger, and P.~Hagemann.
\newblock Generative sliced {MMD} flows with {R}iesz kernels.
\newblock {\em arXiv:2305.11463}, 2023.

\bibitem{JKO1998}
R.~Jordan, D.~Kinderlehrer, and F.~Otto.
\newblock The variational formulation of the {F}okker--{P}lanck equation.
\newblock {\em SIAM Journal on Mathematical Analysis}, 29(1):1--17, 1998.

\bibitem{Ko98}
W.~Koepf.
\newblock Hypergeometric summation.
\newblock {\em Vieweg, Braunschweig/Wiesbaden}, 5(6), 1998.

\bibitem{BookLa72}
N.~Landkof.
\newblock {\em Foundations of Modern Potential Theory}.
\newblock Grundlehren der mathematischen Wissenschaften. Springer, Berlin,
  1972.

\bibitem{LBADDP2022}
R.~Laumont, V.~Bortoli, A.~Almansa, J.~Delon, A.~Durmus, and M.~Pereyra.
\newblock Bayesian imaging using plug \& play priors: when {L}angevin meets
  {T}weedie.
\newblock {\em SIAM Journal on Imaging Sciences}, 15(2):701--737, 2022.

\bibitem{Liu2017}
Q.~Liu.
\newblock Stein variational gradient descent as gradient flow.
\newblock In I.~Guyon, U.~V. Luxburg, S.~Bengio, H.~Wallach, R.~Fergus,
  S.~Vishwanathan, and R.~Garnett, editors, {\em Advances in Neural Information
  Processing Systems}, volume~30, pages 1--9. Curran Associates, Inc., 2017.

\bibitem{MoRoSc19}
M.~G. Mora, L.~Rondi, and L.~Scardia.
\newblock The equilibrium measure for a nonlocal dislocation energy.
\newblock {\em Communications on Pure and Applied Mathematics}, 72:136--158,
  2019.

\bibitem{NS2023}
S.~Neumayer and G.~Steidl.
\newblock From optimal transport to discrepancy.
\newblock In K.~Chen, C.-B. Sch\"onlieb, X.-C. Tai, and L.~Younes, editors,
  {\em Handbook of Mathematical Models and Algorithms in Computer Vision and
  Imaging: Mathematical Imaging and Vision}, pages 1--36. Springer, 2023.

\bibitem{NR2021}
N.~N\"usken and D.~M. Renger.
\newblock Stein variational gradient descent: many-particle and long-time
  asymptotics.
\newblock {\em arXiv:2102.12956v1}, 2021.

\bibitem{Ot01}
F.~Otto.
\newblock The geometry of dissipative evolution equations: the porous medium
  equation.
\newblock {\em Communications in Partial Differential Equations}, 26:101--174,
  2001.

\bibitem{OW2005}
F.~Otto and M.~Westdickenberg.
\newblock Eulerian calculus for the contraction in the {W}asserstein distance.
\newblock {\em SIAM Journal on Mathematical Analysis}, 37(4):1227--1255, 2005.

\bibitem{Pav2014}
G.~A. Pavliotis.
\newblock {\em Stochastic processes and applications: Diffusion Processes, the
  Fokker-Planck and Langevin Equations}.
\newblock Number~60 in Texts in Applied Mathematics. Springer, New York, 2014.

\bibitem{Pi19}
B.~Piccoli.
\newblock Measure differential equations.
\newblock {\em Archive for Rational Mechanics and Analysis}, 233:1289--1317,
  2019.

\bibitem{RC2013}
S.~Reich and C.~J. Cotter.
\newblock Ensemble filter techniques for intermittent data assimilation. large
  scale inverse problems.
\newblock {\em Computational Methods and Applications in the Earth Sciences},
  13:91--134, 2013.

\bibitem{BookRoWe}
R.~T. Rockafellar and R.~J.-B. Wets.
\newblock {\em Variational Analysis}.
\newblock Number 317 in Grundlehren der mathematischen Wissenschaftenvalue.
  Springer Berlin, 2009.

\bibitem{BookSaTo97}
E.~Saff and V.~Totik.
\newblock {\em Logarithmic Potentials with External Fields}.
\newblock Grundlehren der mathema\-tischen Wissenschaften. Springer, Berlin,
  1997.

\bibitem{S2015}
F.~Santambrogio.
\newblock {\em Optimal Transport for Applied Mathematicians}, volume~87 of {\em
  Progress in Nonlinear Differential Equations and their Applications}.
\newblock Birkh\"{a}user, Basel, 2015.

\bibitem{SGBW2010}
C.~Schmaltz, P.~Gwosdek, A.~Bruhn, and J.~Weickert.
\newblock Electrostatic halftoning.
\newblock {\em Comp. Graph. For.}, 29(8):2313--2327, 2010.

\bibitem{SSGF2013}
D.~Sejdinovic, B.~Sriperumbudur, A.~Gretton, and K.~Fukumizu.
\newblock {Equivalence of distance-based and RKHS-based statistics in
  hypothesis testing}.
\newblock {\em The Annals of Statistics}, 41(5):2263 -- 2291, 2013.

\bibitem{S1966}
L.~J. Slater.
\newblock {\em Generalized hypergeometric functions}.
\newblock Cambridge University Press, Cambridge, 1966.

\bibitem{Szekely2002}
G.~Sz\'ekely.
\newblock E-statistics: The energy of statistical samples.
\newblock {\em Techical Report, Bowling Green University}, 2002.

\bibitem{TSGSW2011}
T.~Teuber, G.~Steidl, P.~Gwosdek, C.~Schmaltz, and J.~Weickert.
\newblock Dithering by differences of convex functions.
\newblock {\em SIAM Journal on Imaging Sciences}, 4(1):79--108, 2011.

\bibitem{TS2020}
N.~G. Trillos and D.~Sanz-Alonso.
\newblock The {B}ayesian update: variational formulations and gradient flows.
\newblock {\em Bayesian Analysis}, 15(1):29--56, 2020.

\bibitem{WT2011}
M.~Welling and Y.-W. Teh.
\newblock Bayesian learning via stochastic gradient {L}angevin dynamics.
\newblock In L.~Getoor and T.~Scheffer, editors, {\em ICML'11: Proceedings of
  the 28th International Conference on International Conference on Machine
  Learning}, pages 681--688, Madison, 2011. Omnipress.

\bibitem{wendland2005}
H.~Wendland.
\newblock {\em Scattered Data Approximation}.
\newblock Cambridge University Press, 2005.

\end{thebibliography}
\end{document}